\documentclass[12pt]{amsart}

\usepackage{amscd,latexsym,amsthm,amsfonts,amssymb,amsmath,amsxtra,}
\usepackage[mathscr]{eucal}
\renewcommand\theequation{\thesection.\arabic{equation}}

\newcommand{\BA}{{\mathbb {A}}}

\newcommand{\BC}{{\mathbb {C}}}

\newcommand{\CA}{{\mathcal {A}}}
\newcommand{\CB}{{\mathcal {B}}}

\newcommand{\CE}{{\mathcal {E}}}

\newcommand{\CH}{{\mathcal {H}}}
\newcommand{\CI}{{\mathcal {I}}}
\newcommand{\CJ}{{\mathcal {J}}}

\newcommand{\CL}{{\mathcal {L}}}

\newcommand{\CN}{{\mathcal {N}}}

\newcommand{\CP}{{\mathcal {P}}}

\newcommand{\CS}{{\mathcal {S}}}

\newcommand{\CU}{{\mathcal {U}}}

\newcommand{\CY}{{\mathcal {Y}}}
\newcommand{\CZ}{{\mathcal {Z}}}

\newcommand{\Fo}{{\mathfrak {o}}}

\newcommand{\RU}{{\mathrm {U}}}

\newcommand{\GL}{{\mathrm{GL}}}

\newcommand{\Hom}{{\mathrm{Hom}}}

\newcommand{\Ind}{{\mathrm{Ind}}}
\newcommand{\ind}{{\mathrm{ind}}}

\newcommand{\I}{{\mathrm{I}}}

\renewcommand{\Re}{{\mathrm{Re}}}

\newcommand{\Res}{{\mathrm{Res}}}

\newcommand{\SO}{{\mathrm{SO}}}

\newcommand{\Sym}{{\mathrm{Sym}}}
\newcommand{\sgn}{{\mathrm{sgn}}}

\newcommand{\Span}{{\mathrm{Span}}}
\newcommand{\tr}{{\mathrm{tr}}}

\newcommand{\ud}{\,\mathrm{d}}

\newcommand{\wi}{{\mathrm{Witt}}}

\newcommand{\udl}{\underline}

\newcommand{\apair}[1]{\left\langle {#1} \right\rangle}
\newcommand{\bpair}[1]{\left[{#1}\right]}
\newcommand{\cpair}[1]{\left\{{#1}\right\}}
\newcommand{\ppair}[1]{\left( {#1} \right)}

\def\bks{{\backslash}}

\def\diag{{\rm diag}}

\def\eps{{\epsilon}}

\def\lam{{\lambda}}
\def\Lam{{\Lambda}}

\def\ome{{\omega}}

\def\sig{{\sigma}}

\newtheorem{thm}{Theorem}[section]
\newtheorem{cor}[thm]{Corollary}
\newtheorem{lem}[thm]{Lemma}
\newtheorem{prop}[thm]{Proposition}

\newtheorem {ques/conj}[thm]{Question/Conjecture}

\makeatletter

\newcommand{\Rmnum}[1]{\expandafter\@slowromancap\romannumeral #1@}
\makeatother

\begin{document}
\renewcommand{\theequation}{\arabic{equation}}
\numberwithin{equation}{section}

\title[$L$-functions of Classical Groups of Hermitian Type]
{A Product of Tensor Product $L$-functions of Quasi-split Classical Groups of Hermitian Type}

\author{Dihua Jiang}
\address{School of Mathematics\\
University of Minnesota\\
Minneapolis, MN 55455, USA}
\email{dhjiang@math.umn.edu}

\author{Lei Zhang}
\address{Department of Mathematics,\\
Boston College, Chestnut Hill, MA 02467, USA}
\email{lei.zhang.2@bc.edu}

\subjclass[2010]{Primary 11F70, 22E50; Secondary 11F85, 22E55}

\date{March, 2013}


\keywords{Bessel Periods of Eisenstein Series, Global Zeta Integrals, Tensor Product $L$-functions, Classical Groups of Hermitian Type}

\thanks{The work of the first named author is supported in part by the NSF Grant DMS--1001672}

\begin{abstract}
A family of global integrals representing a product of tensor product (partial) $L$-functions:
$$
L^S(s,\pi\times\tau_1)L^S(s,\pi\times\tau_2)\cdots L^S(s,\pi\times\tau_r)
$$
are established in this paper, where $\pi$ is an irreducible cuspidal automorphic representation of a quasi-split classical group
of Hermitian type and $\tau_1,\cdots,\tau_r$ are irreducible unitary cuspidal automorphic representations of $\GL_{a_1},\cdots,\GL_{a_r}$,
respectively. When $r=1$ and the classical group is an orthogonal group, this was studied by Ginzburg, Piatetski-Shapiro and Rallis
in 1997 and when $\pi$ is generic and $\tau_1,\cdots,\tau_r$ are not isomorphic to each other, this is considered by Ginzburg, Rallis and Soudry
in 2011. In this paper, we prove that the global integrals
are eulerian and finish the explicit calculation of unramified local $L$-factors in general. The remaining local and global theory
for this family of global integrals will be considered in our future work.
\end{abstract}

\maketitle


\section{Introduction}

Let $F$ be a number field and $E$ be a quadratic extension of $F$ when we discuss unitary groups and $E$ be equal to $F$ when we
discuss orthogonal groups. Let $G_n$ be a quasi-split group, which is either $\RU_{n,n}$, $\RU_{n+1,n}$, $\SO_{2n+1}$, or
$\SO_{2n}$, defined over $F$. Let $\BA_E$ be the ring of adeles of $E$ and $\BA$ be the ring of adeles of $F$.
Take $\tau$ to be an irreducible generic automorphic representation of $\Res_{E/F}(\GL_a)(\BA)=\GL_a(\BA_E)$ of isobaric type,
i.e.
\begin{equation}
\tau=\tau_1\boxplus\tau_2\boxplus\cdots\boxplus\tau_r
\end{equation}
where $a=\sum_{i=1}^ra_i$ is a partition of $a$ and $\tau_i$ is an irreducible unitary cuspidal automorphic representation of
$\GL_{a_i}(\BA_E)$. Let $\pi$ be an irreducible cuspidal automorphic
representation of $G_n(\BA)$. We consider a family of global zeta integrals (see Section 3 for definition),
which represent the family of the tensor product (partial)
$L$-functions  $L^S(s,\pi\times\tau)$, which is expressed as follows:
\begin{equation}
L^S(s,\pi\times\tau)=L^S(s,\pi\times\tau_1)L^S(s,\pi\times\tau_2)\cdots L^S(s,\pi\times\tau_r).
\end{equation}
It is often
interesting and important in Number Theory and Arithmetic to consider certain simultaneous behavior at a certain given point $s=s_0$.
For instance, the nonvanishing at $s=\frac{1}{2}$, the center of the symmetry of the $L$-functions
$L^S(s,\pi\times\tau_1)$, $L^S(s,\pi\times\tau_2)$, $\cdots$, $L^S(s,\pi\times\tau_r)$, or particularly, taking
$\tau_1=\tau_2=\cdots=\tau_r$, which yields the $r$-th power $L^S(s,\pi\times\tau_1)^r$ for all positive integers $r$. As remarked
at the end of this paper, the arguments and the methods still work if one replaces the single variable $s$ by multi-variable
$(s_1,\cdots,s_r)$. However, we focus on the case of single variable $s$ in this paper.

We use a family of the Bessel periods (discussed in Section 2) to define the family of global zeta integrals, following closely
the formulation of Ginzburg, Piatetski-Shapiro and Rallis in \cite{GPSR97}, where the case when $r=1$ and $G_n$ is an orthogonal group
was considered. When $\pi$ is generic, i.e. has a nonzero Whittaker-Fourier coefficient, and $\tau_1,\cdots,\tau_r$ are not isomorphic to each other,
this family of tensor
product $L$-functions were studied by Ginzburg, Rallis and Soudry in their recent book \cite{GRS11}. However, the global integrals studied in
\cite{GRS11} are formulated in a different way and can not cover the general situation considered in this paper, 
while the global zeta integrals here are the most general version of
this kind, the origin of which goes back to the pioneer work of Piatetski-Shapiro and Rallis and of Gelbart and Piatetski-Shapiro (\cite{GPSR87}).
Some more special cases were studied earlier by various authors and we refer to the relevant discussions in \cite{GPSR97} and \cite{GRS11}.

In addition to the potential application towards the simultaneous nonvanishing of the central values of the tensor product $L$-functions,
the basic relation between the product of the tensor product (partial) $L$-functions and the family of global zeta integrals is also
an important ingredient in the proof of the nonvanishing of the certain explicit constructions of endoscopy correspondences as indicated
for some special cases in the work of Ginzburg in \cite{G08}, and as generally formulated in the work of the first name author in
(\cite{J11} and \cite{J12}). We will come back to this topic in our future work (\cite{JZ13}).

In general, the meromorphic continuation to the whole complex plane of the product of the tensor product (partial) $L$-functions is
known from the work of R. Langlands on the explicit calculation of the constant terms of Eisenstein series (\cite{L71}). However, when
$\pi$ is not generic, i.e. has no nonzero Whittaker-Fourier coefficients, the Langlands conjecture on the standard functional equation
and the finite number of poles for $\Re(s)\geq\frac{1}{2}$ is still not known (\cite{Sh10}).

According to the recent work of Arthur (\cite{Ar12}) and also of C.-P. Mok (\cite{Mk12}), any irreducible cuspidal automorphic representation 
$\pi$ of $G_n(\BA)$ has the Arthur-Langlands transfer $\pi_\psi$ from $G_n$ to the corresponding general linear group, which is 
compatible with the corresponding local Langlands functorial transfers at all unramified local places of $\pi$. 
Hence we have an identity for partial $L$-functions
$$
L^S(s,\pi\times\tau)=L^S(s,\pi_\psi\times\tau).
$$
The partial $L$-function on the right hand side is the Rankin-Selberg convolution $L$-function for general linear groups studied by 
Jacquet, Piatetski-Shapiro and Shalika (\cite{JPSS83}). When $\pi$ has a generic global Arthur parameter, 
the Arthur-Langlands transfer from $G_n$ to general linear groups is compatible with the corresponding local Langlands functorial 
transfer at all local places. Hence one may define the complete tensor product $L$-function by 
$$
L(s,\pi\times\tau):=L(s,\pi_\psi\times\tau),
$$
just as in (\cite{Ar12} and \cite{Mk12}). 

However, when the global Arthur parameter $\psi$ is not generic, there exists irreducible cuspidal automorphic representation $\pi$ 
with Arthur parameter $\psi$, but there exists a ramified local place $\nu$ such that $\pi_\nu$ and $(\pi_\psi)_\nu$ are 
{\it not} related by the local Langlands functorial transfer at $\nu$. Hence it is impossible to define the local 
tensor product $L$-factors (and also $\gamma$-factors and $\epsilon$-factors) of the pair $(\pi_\nu,\tau_\nu)$ in terms of those of the pair 
$((\pi_\psi)_\nu,\tau_\nu)$. Therefore, it is still an {\it open problem} to define the local ramified $L$-factors 
(and also $\gamma$-factors and $\epsilon$-factors) for an irreducible cuspidal automorphic representation $\pi$ of $G_n(\BA)$ when 
$\pi$ has a non-generic global Arthur parameter. At this point, it seems that the integral representation of Rankin-Selberg type for 
automorphic $L$-functions is the only available method to attack this open problem. 

For quasi-split classical groups of skew-Hermitian type, 
some preliminary work has been done in \cite{GJRS11}, using Fourier-Jacobi periods. Further work is in progress,
including the work of X. Shen in his PhD thesis in University of Minnesota, 2013, which has produced two preprints \cite{Sn12} and \cite{Sn13}.
A parallel theory for this case will also be considered in future.

In Section 2, we introduce a family of Eisenstein series, which is needed for the construction of the family of global zeta integrals,
and discuss the family of Bessel periods which are needed to formulate
the family of global zeta integrals. In Section 3, we unfold the global zeta integrals and prove that they are eulerian.
In Section 4, we do the explicit calculations for local zeta integrals with unramified data to produce the unramified
local $L$-factors of the tensor product type, following the main arguments, which can be viewed as natural extension of those used 
in \cite{GPSR97}. The main global result is Theorem 4.12, which
is started at the end of Section 4. The ideas and methods used in the proofs in this paper will be described with more
details in each section, which are essentially the extension of those used in \cite{GPSR97} for the orthogonal group $G_n$ and
with $r=1$ to the generality considered in this paper.

\section{Certain Eisenstein series and Bessel periods}

We introduce a family of Eisenstein series which will be used in the definition of a family of global zeta integrals, representing
the family of the product of the tensor product $L$-functions as discussed in the introduction. The global zeta integrals are
basically a family of Bessel periods of those Eisenstein series. We recall first the general notion of
the Bessel periods of automorphic forms from \cite{GPSR97}, \cite{GJR09}, \cite{BS09} and \cite{GRS11}.

Let $F$ be an number field. Define $E=F$ or $E=F(\sqrt{\rho})$, a quadratic extension of $F$, depending on that the classical group
we considered is orthogonal or unitary, accordingly.  It follows that the Galois group of $E/F$ is
either trivial or generated by non-trivial automorphism $x\mapsto \bar{x}$.
The ring of adeles of $F$ is denoted by $\BA$, while the ring of adeles of $E$ is denoted by $\BA_{E}$.

Let $V$ be a $E$-vector space of dimension $m$ with a non-degenerate quadratic form $q_{_{V}}$ if $E=F$ or a non-degenerate
Hermitian form (also denoted by $q_{_{V}}$) if $E=F(\sqrt{\rho})$.
Let $\RU(q_{_V})$ be the connected component of isometry group of $(V,q_{_V})$ defined over $F$. It follows that $\RU(q_{_V})$ is a
special orthogonal group or a unitary group. Let $\tilde{m}=\wi(V)$ be the Witt index of $V$.
Let $V^{+}$ be a maximal totally isotropic subspace of $V$ and $V^-$ be its dual, so that $V$ has the following
polar decomposition
$$
V=V^+\oplus V_0\oplus V^-,
$$
where $V_{0}=(V^{+}\oplus V^{-})^{\perp}$ denotes the anisotropic kernel of $V$.
We choose a basis $\cpair{e_{1},e_{2},\dots,e_{\tilde{m}}}$ of $V^{+}$ and a basis
$\cpair{e_{-1},e_{-2},\dots,e_{-\tilde{m}}}$ of $V^-$ such that $q_{_V}(e_{i},e_{-j})=\delta_{i,j}$
for all $1\leq i, j\leq \tilde{m}$.

We assume in this paper that the algebraic $F$-group $\RU(q_{_V})$ is $F$-quasi-split.
Then the anisotropic kernel $V_{0}$ is at most two dimensional. More precisely, when $E=F$, if $\dim_EV=m$ is even, then
$\dim_EV_0$ is either $0$ or $2$, and if $\dim_EV=m$ is odd, then $\dim_EV_0$ is $1$; and
when $E=F(\sqrt{\rho})$, $\dim_EV_0$ is $0$ or $1$ according to that $\dim_EV=m$ is even or odd.

When $\dim V_{0}=2$, we choose an orthogonal basis $\{e^{(1)}_{0},e^{(2)}_{0}\}$ of $V_{0}$ with the property that
$$
q_{_{V_0}}(e^{(1)}_{0},e^{(1)}_{0})=1,\qquad q_{_{V_0}}(e^{(2)}_{0},e^{(2)}_{0})=-c,
$$
where $c\in F^{\times}$ is not a square and $q_{_{V_0}}=q_{_V}|_{_{V_0}}$. When $\dim V_{0}=1$, we choose an anisotropic basis vector $e_{0}$
for $V_{0}$. We put the basis in the following order
\begin{equation}
e_{1},e_{2},\dots,e_{\tilde{m}},e^{(1)}_{0},e^{(2)}_{0},e_{-\tilde{m}},\cdots,e_{-2},e_{-1}
\end{equation}
if $\dim_EV_0=2$;
\begin{equation}
e_{1},e_{2},\dots,e_{\tilde{m}},e_{0},e_{-\tilde{m}},\cdots,e_{-2},e_{-1}
\end{equation}
if $\dim_EV_0=1$; and
\begin{equation}
e_{1},e_{2},\dots,e_{\tilde{m}},e_{-\tilde{m}},\cdots,e_{-2},e_{-1}
\end{equation}
if $\dim_EV_0=0$.

With the choice of the ordering of the basis vectors, the $F$-rational points $\RU(q_{_V})(F)$ of the algebraic group
$\RU(q_{_V})$ are realized as an algebraic subgroup of $\GL_m(E)$. Define $n=[\frac{m}{2}]$ and put $G_n=\RU(q_{_V})$, which is
the same as given in the introduction.
From now on, for any $F$-algebraic subgroup $H$ of $G_n$, the $F$-rational points $H(F)$ of $H$ are regarded as a
subgroup of $\GL_m(E)$. Similarly, the $\BA$-rational points $H(\BA)$ of $H$ are regarded as a subgroup of $\GL_m(\BA_E)$.

The corresponding standard flag of $V$ (with respect to the given ordering of the basis vectors) defines an $F$-Borel subgroup $B$.
We may write $B=TN$ with $T$ a maximal $F$-torus, whose elements are diagonal matrices, and with $N$ the unipotent radical of $B$,
whose elements are upper-triangular matrices.
Let $T_0$ be the maximal $F$-split torus of $G_n$ contained in $T$. We define the root system $\Phi(T_0,G_n)$ with the set of positive roots $\Phi^+(T_0,G_n)$
corresponding to the Borel subgroup given above.

Let $V^{+}_{\ell}$ be the totally isotropic subspace generated by $\cpair{e_{1},e_{2},\dots,e_{\ell}}$ and $P_{\ell}=M_{\ell}U_{\ell}$
be a standard maximal parabolic subgroup of $G_n$, which stabilizes $V^{+}_{\ell}$. The Levi subgroup $M_{\ell}$ is isomorphic to
$
\GL(V_{\ell}^+)\times G_{n-\ell}.
$
Here $\GL(V_{\ell}^+)=\Res_{E/F}(\GL_{\ell})$ and $G_{n-\ell}=\RU(q_{_{W_{\ell}}})$ with $q_{_{W_{\ell}}}=q_{_V}|_{_{W_{\ell}}}$
and
$
W_\ell=(V^{+}_{\ell}\oplus V^{-}_{\ell})^{\perp}.
$

Let $\udl{\ell}:=[\ell_{1}\ell_{2}\dots\ell_{p}]$ be a partition of $\ell$. Then
$
P_{\udl{\ell}}=M_{\udl{\ell}}U_{\udl{\ell}}
$
is a standard parabolic subgroup of $\Res_{E/F}\GL_{\ell}$, whose Levi subgroup
$$
M_{\udl{\ell}}\cong\Res_{E/F}\GL_{\ell_{1}}\times \Res_{E/F}\GL_{\ell_{2}}\times\cdots\times\Res_{E/F}\GL_{\ell_{p}}.
$$
\subsection{Bessel periods}
Define $N_{\ell}$ to be the unipotent subgroup of $G_n$ consisting of elements of following type,
\begin{equation}\label{nell}
N_{\ell}=\cpair{n=\begin{pmatrix}z&y&x\\&I_{m-2\ell}&y'\\&&z^{*} \end{pmatrix}\in G_n \mid z\in Z_{\ell}},
\end{equation}
where $Z_{\ell}$ is the standard maximal (upper-triangular) unipotent subgroup of $\Res_{E/F}\GL_{\ell}$. It is clear that
$N_\ell=U_{[1^\ell]}$ where $[1^\ell]$ is the partition of $\ell$ with $1$ repeated $\ell$ times.

Fix a nontrivial character $\psi_{0}$ of $F\bks \BA_{F}$ and define a character $\psi$ of $E\bks \BA_{E}$ by
\begin{equation}\label{psiE}
\psi(x):=
\begin{cases}
\psi_0(x) &\ \textit{if}\ E=F;\\
\psi_{0}(\frac{1}{2}\tr_{E/F}(\frac{x}{\sqrt{\rho}}) &\ \textit{if}\ E=F(\sqrt{\rho}).
\end{cases}
\end{equation}
Then take $w_{0}$ to be an anisotropic vector in $W_{\ell}$ and define a character $\psi_{\ell,w_{0}}$ of $N_{\ell}$ by
\begin{equation}\label{chw0}
\psi_{\ell,w_{0}}(n):=\psi(\sum^{\ell-1}_{i=1}z_{i,i+1}+q_{_{W_\ell}}(y_{\ell}, w_{0})),
\end{equation}
where $y_\ell$ is the last row of $y$ in $n\in N_\ell$ as defined in \eqref{nell}, which is regarded as a vector in $W_{\ell}$.

If $\ell=\tilde{m}$, $\psi_{\ell,w_{0}}$ is a generic character on the maximal unipotent group $N=N_{\tilde{m}}$.
We will not consider this case here and hence we always assume that $\ell<\tilde{m}$ from now on.

For $\kappa\in F^{\times}$, we choose
\begin{equation}\label{w0}
w_{0}=y_{\kappa}=e_{\tilde{m}}+(-1)^{m+1}\frac{\kappa}{2}e_{-\tilde{m}},
\end{equation}
which implies that $q(y_{\kappa},y_{\kappa})=(-1)^{m+1}\kappa$ and that the corresponding character is
\begin{equation}\label{chalpha}
\psi_{\ell,\kappa}(n)=\psi_{\ell,w_{0}}(n)=
\psi(\sum^{\ell-1}_{i=1}z_{i,i+1}+y_{\ell,\tilde{m}-\ell}+(-1)^{m+1}\frac{\kappa}{2}y_{\ell,m-\tilde{m}-\ell+1}).
\end{equation}
The Levi subgroup $M_{[1^\ell]}=(\Res_{E/F}\GL_{1})^{\times\ell}\times G_{n-\ell}$ normalizes the unipotent subgroup $N_{\ell}$
by the adjoint action, and acts on the set of the characters $\psi_{\ell,\kappa}$, with $\kappa\in F^{\times}$, of $N_\ell(F)$.
The $M_{[1^\ell]}(F)$-orbits are classified by the Witt Theorem and give all $F$-generic characters of $N_\ell(F)$.
The stabilizer of $\psi_{\ell,w_{0}}$ in the Levi subgroup $M_{[1^\ell]}$ is the subgroup
\begin{equation}\label{Lellw0}
L_{\ell,w_{0}}=
\cpair{\begin{pmatrix} I_{\ell}&&\\&\gamma&\\&&I_{\ell} \end{pmatrix}\in G_n\mid \gamma J_{m-2\ell}w_{0}=J_{m-2\ell}w_{0}}\cong H_{n-\ell},
\end{equation}
where $H_{n-\ell}$ is defined to be $\RU(q_{_{W_{\ell}\cap w^{\perp}_{0}}})$ with
$q_{_{W_{\ell}\cap w^{\perp}_{0}}}=q_{_V}|_{_{W_{\ell}\cap w^{\perp}_{0}}}$, and $J_{k}$ is the $k\times k$ matrix defined inductively by
$J_{k}=\ppair{\begin{smallmatrix} &1\\J_{k-1}&\end{smallmatrix}}$ and $J_{1}=1$.
Define
\begin{equation}
R_{\ell,w_{0}}:=H_{n-\ell}N_\ell=\RU(q_{_{W_{\ell}\cap w^{\perp}_{0}}})N_{\ell}.
\end{equation}
Note that $\dim_EV$ and $\dim_EW_{\ell}\cap w^{\perp}_{0}$ have the different parity.
If $\ell=0$, the unipotent subgroup $N_{0}$ is trivial and we have that
$$
R_{0,w_{0}}=\RU(q_{_{V\cap w_{0}^{\perp}}}).
$$
When taking $w_{0}=y_{\kappa}$,  we will use the notation $\psi_{\ell,y_\kappa}=\psi_{\ell,\kappa}$, $L_{\ell, y_{\kappa}}=L_{\ell,\kappa}$
and $R_{\ell,y_{\kappa}}=R_{\ell,\kappa}$, respectively.

Let $\phi$ be an automorphic form on $G_n(\BA)$. Define the {\bf Bessel-Fourier coefficient} (or Gelfand-Graev model)
of $\phi$ by
\begin{equation}\label{bfc}
\CB^{\psi_{\ell,w_{0}}}(\phi)(h)
:=
\int_{N_\ell(F)\bks N_\ell(\BA)}\phi(nh)\psi_{\ell,w_{0}}^{-1}(n) \ud n.
\end{equation}
This defines an automorphic function on the stabilizer $L_{\ell,w_{0}}(\BA)=H_{n-\ell}(\BA)$. Take a cuspidal automorphic form
$\varphi$ on $H_{n-\ell}(\BA)$ and define the $(\psi_{\ell,w_{0}},\varphi)$-{\bf Bessel period} or simply {\bf Bessel period}
of $\phi$ by
\begin{equation}\label{bp}
\CP^{\psi_{\ell,w_{0}}}(\phi,\varphi)
:=
\int_{H_{n-\ell}(F)\bks H_{n-\ell}(\BA)}\CB^{\psi_{\ell,w_{0}}}(\phi)(h)\varphi(h) \ud h.
\end{equation}
We will apply this Bessel period to a family of Eisenstein series next.

\subsection{Eisenstein series}
We follow the notation of \cite{MW95} to define a family of Eisenstein series. Let $P_j=M_jU_j$ be a standard maximal parabolic
$F$-subgroup of $G_n$ with the Levi subgroup
$$
M_j=\Res_{E/F}(\GL_j)\times G_{n-j},
$$
for some $j$ with $1\leq j\leq \tilde{m}$. When $j=\tilde{m}$, the group $G_{n-\tilde{m}}$ disappears, if $\dim_EV_0=0$, or $\dim_EV_0=1$ and $E=F$.
Following \cite[Page 5]{MW95}, the space $X_{M_j}$ of all continuous homomorphisms from $M_j(\BA)$ to $\BC^\times$,
which is trivial on $M_j(\BA)^1$, can be
identified with $\BC$ by the mapping $\lambda_s\leftrightarrow s$, which is normalized as in \cite{Sh10}.

Let $\tau$ be an irreducible unitary generic automorphic representation
of $\GL_j(\BA_E)$ of the following isobaric type:
\begin{equation}\label{tau}
\tau=\tau_1\boxplus\tau_2\boxplus\cdots\boxplus\tau_r,
\end{equation}
where $\udl{j}=[j_1j_2\cdots j_r]$ is a partition of $j$ and $\tau_i$ is an irreducible unitary cuspidal automorphic
representation of $\GL_{j_i}(\BA_E)$.
Let $\sigma$ be an irreducible automorphic representation of $G_{n-j}(\BA)$, which may not be cuspidal. Note that
$\sigma$ is irrelevant if $j=\tilde{m}$ and the group $G_{n-\tilde{m}}$ disappears. Following the
definition of automorphic forms in \cite[I.2.17]{MW95}, take an automorphic form
\begin{equation}\label{af}
\phi=\phi_{\tau\otimes\sigma}\in \CA(U_j(\BA)M_j(F)\bks G_n(\BA))_{\tau\otimes\sigma}.
\end{equation}
For $\lam_s\in X_{M_{j}}$, the Eisenstein series associated to $\phi(g)$ is defined by
\begin{equation}\label{es}
E(\phi,s)(g)=E(\phi_{\tau\otimes\sigma}, \lam_s)(g)=\sum_{\delta\in P_{j}(F)\bks G_n(F)}\lam_s \phi(\delta g).
\end{equation}
It is absolutely convergent for $\Re(s)$ large and uniformly converges for $g$ over any compact subset of $G_n(\BA)$,
has meromorphic continuation to $s\in\BC$ and satisfies the standard functional equation.

Recall that $H_{n-\ell}$ is defined to be $\RU(q_{_{W_{\ell}\cap w^{\perp}_{0}}})$ and that
$\dim_EV$ and $\dim_EW_{\ell}\cap w^{\perp}_{0}$ have the different parity.
Let $\pi$ be an irreducible \emph{cuspidal} automorphic representation of $H_{n-\ell}(\BA)$
and take a cuspidal automorphic form
\begin{equation}\label{caf}
\varphi\in \CA_0(H_{n-\ell}(F)\bks H_{n-\ell}(\BA))_{\pi}.
\end{equation}

The global zeta integral $\CZ(s,\phi_{\tau\otimes\sigma},\varphi_\pi,\psi_{\ell,w_{0}})$ is defined to be the following
Bessel period
\begin{equation}\label{gzi}
\CZ(s,\varphi_\pi,\phi_{\tau\otimes\sigma},\psi_{\ell,w_{0}})
:=
\CP^{\psi_{\ell,w_{0}}}(E(\phi_{\tau\otimes\sigma},s),\varphi_\pi).
\end{equation}
Because $\varphi_\pi$ is cuspidal, it is easy to see that the following holds.

\begin{prop}\label{pp}
The global zeta integral $\CZ(s,\phi_{\tau\otimes\sigma},\varphi_\pi,\psi_{\ell,w_{0}})$
converges absolutely at any $s\in\BC$ when the Eisenstein series $E(\phi_{\tau\otimes\sigma},s)$ has no pole at $s$
and hence is holomorphic, and has possible poles at the locations where the Eisenstein series has poles.
\end{prop}

\section{The eulerian property of the global integrals}\label{sec:eulerian}

We prove here that the global zeta integral $\CZ(s,\phi_{\tau\otimes\sigma},\varphi_\pi,\psi_{\ell,w_{0}})$ will be expressed as
an eulerian product of local zeta integrals. When $j=n=[\frac{m}{2}]$, such global zeta integrals with generic $\pi$ have been studied in
\cite[Chapter 10]{GRS11}. Hence we assume from now on that $j<n$ and also $\ell<\tilde{m}\leq n$.

Recall from \eqref{gzi} that $\CZ(s,\phi_{\tau\otimes\sigma},\varphi_\pi,\psi_{\ell,w_{0}})$ is
the $(\psi_{\ell,w_{0}},\varphi_\pi)$-{\bf Bessel period} of the Eisenstein series $E(\phi_{\tau\otimes\sigma}, \lam_s)(g)$, which is
given by
\begin{equation}\label{bpes}
\int_{H_{n-\ell}(F)\bks H_{n-\ell}(\BA)}\CB^{\psi_{\ell,w_{0}}}(E(\phi_{\tau\otimes\sigma},s))(h)\varphi_\pi(h)\ud h,
\end{equation}
where the Bessel-Fourier coefficient $\CB^{\psi_{\ell,w_{0}}}(E(\phi_{\tau\otimes\sigma},s))(h)$ is given, as in \eqref{bfc}, by
\begin{equation}\label{bfces}
\CB^{\psi_{\ell,w_{0}}}(E(\phi_{\tau\otimes\sigma},s))(h)
:=
\int_{N_\ell(F)\bks N_\ell(\BA)}E(\phi_{\tau\otimes\sigma},s)(nh)\psi_{\ell,w_{0}}^{-1}(n)\ud n.
\end{equation}
We first calculate the Bessel-Fourier coefficient $\CB^{\psi_{\ell,w_{0}}}(E(\phi_{\tau\otimes\sigma},s))$.

\subsection{Calculation of Bessel-Fourier coefficients}
In order to calculate the Bessel-Fourier coefficient $\CB^{\psi_{\ell,w_{0}}}(E(\phi_{\tau\otimes\sigma},s))$, i.e.
the integral in \eqref{bfces}, we assume that the $\Re(s)$ is large, and unfold the Eisenstein series. This leads to consider
the double cosets decomposition $P_{j}\bks G_n /P_{\ell}$, whose set of representatives $\epsilon_{\alpha,\beta}$ is explicitly given
in \cite[Section 4.2]{GRS11}. In our situation, we put it into three cases for discussion.

{\bf Case (1):}\ $G_n$ is not the $F$-split even special orthogonal group. In this case, the set of representatives $\epsilon_{\alpha,\beta}$
of the double coset decomposition $P_{j}\bks G_n /P_{\ell}$
is in bijection with the set of pairs of nonnegative integers
\begin{equation}
 \cpair{(\alpha,\beta)\mid 0\leq \alpha\leq \beta\leq j \text{ and } j\leq \ell+\beta-\alpha\leq \tilde{m}}.
\end{equation}
Recall that $\tilde{m}$ is the Witt index of $(V,q_{_V})$ defining $G_n$.

{\bf Case (2-1):}\ $G_n$ is the $F$-split even special orthogonal group and $\ell+\beta-\alpha<\tilde{m}=n$. In this case,
the set of representatives $\epsilon_{\alpha,\beta}$
of the double coset decomposition $P_{j}\bks G_n /P_{\ell}$
is in bijection with the set of pairs of nonnegative integers
\begin{equation}
 \cpair{(\alpha,\beta)\mid 0\leq \alpha\leq \beta\leq j \text{ and } j\leq \ell+\beta-\alpha\leq\min\{n-1,\tilde{m}\}}.
\end{equation}

{\bf Case (2-2):}\ $G_n$ is the $F$-split even special orthogonal group and $\ell+\beta-\alpha=n$. In this case,
there are two double cosets corresponding to each pair $(\alpha,\beta)$, and hence
we may choose representatives $\eps_{\alpha,\beta}$ and $\tilde{\eps}_{\alpha,\beta}=w_{q}\eps_{\alpha,\beta}w_{q}$ of the two double cosets
corresponding to such pairs $(\alpha,\beta)$.

In all cases, we denote by
$
P^{\epsilon_{\alpha,\beta}}_{\ell}:=\epsilon_{\alpha,\beta}^{-1}P_{j}\epsilon_{\alpha,\beta}\cap P_{\ell}
$
the stabilizer in $P_{\ell}$, whose elements have the following form as matrices in $\GL_m(E)$:
\begin{equation} \label{eq:stable-l}
 g^{(\alpha,\beta)}_{\ell}=\left( \begin{array}{ccccccccc}
a & x_1 & x_2 & y_1 & y_2 & y_3 & z_1 & z_2 & z_3 \\
0 & b & x_3 & 0 & y_4 & y_5 & 0 & z_4 & z_2' \\
0 &  0 & c & 0 & 0 & y_6 & 0 & 0 & z_1' \\
   &     &     & d & u & v & y_6' & y_5' & y_3' \\
   &     &     & 0 & e & u' & 0 & y_4' & y_2'\\
   &     &     & 0 & 0 & d^* & 0 & 0 & y_1'\\
   &     &     &    &     &     & c^* & x_3' & x_2' \\
    &     &     &    &     &     & 0 & b^* & x_1'\\
     &     &     &    &     &     & 0 & 0 & a^*
\end{array} \right)
\end{equation}
where the block sizes are determined by $a, a^*\in\GL_{\alpha}$, $b, b^*\in\GL_{\ell-\alpha-j+\beta}$, $c, c^*\in\GL_{j-\beta}$, $d, d^*\in\GL_{\beta-\alpha}$, and
$e\in\GL_{m-2(\ell+\beta-\alpha)}$. In case $i=0$, $\GL_{i}$ disappears.

The stabilizer in $P_{j}$ consists of elements of the following form, which are the indicated matrices conjugated by $w^{t}_{q}$:
\begin{equation}\label{eq:stable-j}
g^{(\alpha,\beta)}_{j}=\epsilon_{\alpha,\beta}g \epsilon^{-1}_{\alpha,\beta} = \left( \begin{array}{ccccccccc}
a & y_1 &  z_1 & x_1 & y_2 & z_2 & x_2 & y_3 & z_3 \\
0 & d & y_6' & 0 & u & y_5' & 0 & v & y_3' \\
0 &  0 & c^* & 0 & 0 &  x_3' & 0 & 0 & x_2' \\
   &     &     & b & y_4 & z_4 & x_3 & y_5 & z_2' \\
   &     &     & 0 & e & y_4' & 0 & u' & y_2'\\
   &     &     & 0 & 0 & b^* & 0 & 0 & x_1'\\
   &     &     &    &     &     & c & y_6 & z_1' \\
    &     &     &    &     &     & 0 & d^* & y_1'\\
     &     &     &    &     &     & 0 & 0 & a^*
\end{array} \right)^{w_q^t}
\end{equation}
with the block sizes as before and $w^{t}_{q}$ being the $t$-th power of the element $w_{q}$ for $t=j-\beta$.
Also, when $(V,q_{_V})$ is Hermitian, $w_{q}=I_{m}$; when $E=F$ and $(V,q_{_V})$ is of odd dimension, $w_{q}=-I_{m}$;
when $E=F$ and anisotropic kernel $(V_0,q_{_{V_0}})$ is of dimension two,
take $w_{q}=\diag(I_{\tilde{m}},w^{0}_{q},I_{\tilde{m}})$,
where $w^{0}_{q}=\diag\{1,-1\}$;
and finally, when $E=F$ and the anisotropic kernel $(V_0,q_{_{V_0}})$ is a zero space,
take $w_{q}=\diag(I_{\tilde{m}-1},w^{0}_{q},I_{\tilde{m}-1})$,
where $w^{0}_{q}=\begin{pmatrix} &1\\1&\end{pmatrix}$.
Note that $\ell, j<n=[\frac{m}{2}]$, where $m=\dim_EV$ and $\tilde{m}$ is the Witt index of $V$.

In {\bf Case (2-2)}, i.e. $G_n$ is the $F$-split even special orthogonal group and $\ell+\beta-\alpha=\tilde{m}$,
we have two double cosets corresponding to each pair $(\alpha,\beta)$.
For the double coset $P_{j}\eps_{\alpha,\beta}P_{\ell}$,  we get exactly the same form for the stabilizer as above.
For the other double coset $P_{j}\tilde{\eps}_{\alpha,\beta}P_{\ell}$, the stabilizer in $P_{\ell}$ consists of all elements of the form $(g^{(\alpha,\beta)}_{\ell})^{w_{q}}$.

To continue the calculation, we consider further double cosets decomposition $P^{\eps_{\alpha,\beta}}_{\ell}\bks P_{\ell}/ R_{\ell,w_{0}}$.
Recall that $H_{n-\ell}=\RU(q_{_{W_{\ell}\cap w^{\perp}_{0}}})$,
$\dim_EV$ and $\dim_EW_{\ell}\cap w^{\perp}_{0}$ have the different parity, and
$R_{\ell,w_0}=H_{n-\ell}N_\ell$ with $H_{n-\ell}\cong L_{\ell,w_0}$.
By \cite[Section 5.1]{GRS11}, we may choose a set of representatives of form:
\begin{equation}\label{eta}
\eta_{\eps,\gamma}:=\begin{pmatrix} \epsilon&&\\&\gamma&\\&&\epsilon^{*} \end{pmatrix}
\end{equation}
where $\eps$ is a representative in the quotient of Weyl groups
$$
W_{\GL_{\alpha}\times\GL_{\ell-\alpha-t}\times\GL_{t}}\bks W_{\GL_{\ell}}
$$
and $\gamma$ is a representative $P'_{w}\bks G_{n-\ell}/ H_{n-\ell}$, where
$P'_{w}$ is the maximal parabolic subgroup of $G_{n-\ell}$ defined as follows.

In {\bf Case (1)} or {\bf Case (2-1)}, \label{page:hgamma} i.e. when
$G_n$ is not the $F$-split even special orthogonal group or when $G_n$ is the $F$-split even special orthogonal group
with $\ell+\beta-\alpha<n$, then $P'_{w}$ is the parabolic subgroup of $G_{n-\ell}$, which preserves
the standard $\beta-\alpha$ dimensional totally isotropic subspace $V^{+}_{\ell,\beta-\alpha}$ of $W_{\ell}$, where
\begin{equation}\label{vpm}
V^{\pm}_{\ell,f}={\rm Span}_{E}\cpair{e_{\pm(\ell+1)},\dots,e_{\pm(\ell+f)}},
\end{equation}
for a possible integer $f$.

In {\bf Case (2-2)}, i.e. when $G_n$ is the $F$-split even special orthogonal group with $\ell+\beta-\alpha=n$ (with $j,\ell<n$), then,
when $w=\eps_{\alpha,\beta}$, $P'_{w}$ is the parabolic subgroup
of $G_{n-\ell}$, which preserves $V^{+}_{\ell,m-\ell}$; and when $w=\tilde{\eps}_{\alpha,\beta}$, $P'_{w}$ is the parabolic subgroup of
$G_{n-\ell}$, which preserves $w_{q}V^{+}_{\ell,m-\ell}$.

Denote the stabilizer in $H_{n-\ell}$ of the double coset $P'_w\gamma H_{n-\ell}$ with $\eta_{\eps,\gamma}$ as defined in \eqref{eta} by
\begin{equation}\label{hgamma}
H_{n-\ell}^{\eta_{\eps,\gamma}}=H_{n-\ell}^{\gamma}=H_{n-\ell}\cap\gamma^{-1}P'_{w}\gamma
=L_{\ell,w_{0}}\cap\gamma^{-1} P'_{w}\gamma.
\end{equation}

With the above preparation, we are ready to calculate the Bessel-Fourier coefficient
$\CB^{\psi_{\ell,w_{0}}}(E(\phi_{\tau\otimes\sigma},\lam))(h)$ by assuming that $\Re(s)$ is large so that we are able to unfold the Eisenstein series.
\begin{align*}
&\CB^{\psi_{\ell,w_{0}}}(E(\phi,s))(h)\\
=&\int_{N_{\ell}(F)\bks N_{\ell}(\BA)} E(\phi,s)(nh)\psi^{-1}_{\ell,w_{0}}(n)\ud n\\
=&\sum_{\epsilon_{\alpha,\beta}\in \CE_{j,\ell}}\int_{N_{\ell}(F)\bks N_{\ell}(\BA)}\sum_{\delta\in P^{\epsilon_{\alpha,\beta}}_{\ell}(F)\bks P_{\ell}(F)}\lam\phi(\epsilon_{\alpha,\beta}\delta nh)\psi^{-1}_{\ell,w_{0}}(n)\ud n,
\end{align*}
where $\CE_{j,\ell}$ is the set of representatives of $P_{j}(F)\bks G_n(F)/P_{\ell}(F)$. Set $\CN_{\alpha,\beta,\ell,w_0}$ to be the set
of representatives of $P^{\epsilon_{\alpha,\beta}}_{\ell}(F)\bks P_{\ell}(F)/ R_{\ell,w_{0}}(F)$ and deduce that the above is equal to
$$
\sum_{\epsilon_{\alpha,\beta}}\sum_{\eta\in \CN_{\alpha,\beta,\ell,w_0}}
\int_{N_{\ell}(F)\bks N_{\ell}(\BA)}\sum_{\delta\in R^{\eta}_{\ell,w_{0}}(F)\bks R_{\ell,w_{0}}(F)}\lam\phi(\epsilon_{\alpha,\beta}\eta\delta nh)\psi^{-1}_{\ell,w_{0}}(n)\ud n.
$$
Since $R_{\ell,w_0}=H_{n-\ell}N_\ell$, by re-arranging the summation in $\delta$ and the integration of $dn$,
we obtain that the above is equal to
$$
\sum_{\epsilon_{\alpha,\beta}}\sum_{\eta} \sum_{\delta\in H^{\eta}_{n-\ell}(F)\bks H_{n-\ell}(F)}\int_{N^{\eta}_{\ell}(F)\bks N_{\ell}(\BA)}\lam\phi(\epsilon_{\alpha,\beta}\eta\delta nh)\psi^{-1}_{\ell,w_{0}}(n)\ud n.
$$
By factoring the integration of ${\rm d}n$, we obtain that when $\Re(s)$ is large, the Bessel-Fourier coefficient
$\CB^{\psi_{\ell,w_{0}}}(E(\phi_{\tau\otimes\sigma},s))(h)$ is equal to
\begin{equation}\label{rs}
\sum_{\epsilon_{\alpha,\beta};\eta; \delta }
\int_{N^{\eta}_{\ell}(\BA)\bks N_{\ell}(\BA)}\int_{N^{\eta}_{\ell}(F)\bks N^{\eta}_{\ell}(\BA)}\lam\phi(\epsilon_{\alpha,\beta}\eta\delta unh)\psi^{-1}_{\ell,w_{0}}(un)\ud u\ud n.
\end{equation}

In order to determine the summands in \eqref{rs}, we need the following two lemmas, which are the global version of Propositions 5.1 and 5.2 in \cite[Chapter 5]{GRS11}.

\begin{lem}\label{lm:r>0}
If $\alpha>0$, then the inner integral in \eqref{rs} has the following property:
$$
\int_{N^{\eta}_{\ell}(F)\bks N^{\eta}_{\ell}(\BA)}\lam\phi(\epsilon_{\alpha,\beta}\eta unh)\psi^{-1}_{\ell,w_{0}}(un)\ud u=0
$$
for all choices of data.
\end{lem}

\begin{proof}
If there exists a simple root subgroup $U$ of $Z_{\ell}$ such that $\epsilon U\epsilon^{-1}$ lies
inside $U_{\alpha,\ell-\alpha-t,t}$ for some $\eps\in W_{\GL_{\alpha}\times\GL_{\ell-\alpha-t}\times\GL_{t}}\bks W_{\GL_{\ell}}$,
then the subgroup $\epsilon_{\alpha,\beta}\eta_{\eps,\gamma} U(\epsilon_{\alpha,\beta}\eta_{\eps,\gamma})^{-1}$ lies inside $U_{j}$.
Since the automorphic function $\lam\phi$ is invariant on $U_{j}(\BA)$ and $\psi_{\ell,w_{0}}$ is not trivial on $U(\BA)$,
\begin{eqnarray*}
&&\int_{U(F)\bks U(\BA)}\lam\phi(\epsilon_{\alpha,\beta}\eta zunh)\psi^{-1}_{\ell,w_{0}}(z)\ud z\\
&=&\lam\phi(\epsilon_{\alpha,\beta}\eta un h)\cdot\int_{E\bks \BA_{E}}\psi^{-1}(x)\ud x
\end{eqnarray*}
is identically zero.

If for each simple root subgroup $U$ of $\GL_{\ell}$, $\epsilon U\epsilon^{-1}$ does not lie inside $U_{\alpha,\ell-\alpha-t,t}$,
then according to \cite[Lemma 5.1]{GRS11}, we choose, under the action of the Weyl group of $M_{\alpha,\ell-\alpha-t,t}$,
\begin{equation}\label{eq:ep}
\epsilon=\begin{pmatrix} &&I_{\alpha}\\&I_{\ell-\alpha-t}&\\ I_{t}&&\end{pmatrix}.
\end{equation}

Since $\alpha\neq 0$ ($\ell<\tilde{m}$), we choose a nontrivial subgroup $S$ of $N_{\ell}$ consisting of elements of form
$$
\begin{pmatrix}
I_{\ell-\alpha}&&&&\\
&I_{\alpha}&y&*&\\
&&I_{m-2\ell}&y'&\\
&&&I_{\alpha}&\\
&&&&I_{\ell-\alpha}
\end{pmatrix}
$$
where $(\tilde{y}_{1}~\tilde{y}_{2}~\tilde{y}_{3})=(0_{r\times (\beta-\alpha)}~y_{2}~y_{3})(w^{t'}_{q}\gamma)^{-1}$,
and $y_{2}$ and $y_{3}$ are of size $\alpha\times (m-2(\ell+\beta-\alpha))$ and $\alpha\times (\beta-\alpha)$, respectively;
and when $G_n$ is split, even orthogonal, $\ell+\beta-\alpha=n$ and the representative $w=\eps^{w_{q}}_{\alpha,\beta}$, we have that $t'=1$, otherwise,
we always have that $t'=0$.
Since $w_{0}$ is anisotropic, $w_{0}$ is not orthogonal to $V_{0}\oplus V^{-}_{\ell,\beta-\alpha}$ and
$\psi_{\ell,w_{0}}$ is not trivial on $S(\BA)$.
By~\eqref{eq:stable-j}, we have $(\epsilon_{\alpha,\beta}\eta_{\eps,\gamma})S(\epsilon_{\alpha,\beta}\eta_{\eps,\gamma})^{-1}$
lies inside $U_{j}$ and then
\begin{eqnarray*}
&&\int_{S(F)\bks S(\BA)}\lam\phi(\epsilon_{\alpha,\beta}\eta xunh)\psi^{-1}_{\ell,w_{0}}(x)\ud x\\
&=&\lam\phi(\epsilon_{\alpha,\beta}\eta un h)\cdot\int_{S(F)\bks S(\BA)}\psi^{-1}_{\ell,w_{0}}(x)\ud x
\end{eqnarray*}
is identically zero. This proves the lemma.
\end{proof}

By Lemma~\ref{lm:r>0} and \eqref{rs}, when $\Re(s)$ is large, the Bessel-Fourier coefficient
$\CB^{\psi_{\ell,w_{0}}}(E(\phi_{\tau\otimes\sigma},\lam))(h)$ is equal to
\begin{equation}\label{bfc0s}
\sum_{\epsilon_{0,\beta};\eta; \delta }
\int_{N^{\eta}_{\ell}(\BA)\bks N_{\ell}(\BA)}\int_{N^{\eta}_{\ell}(F)\bks N^{\eta}_{\ell}(\BA)}\lam\phi(\epsilon_{0,\beta}\eta\delta unh)\psi^{-1}_{\ell,w_{0}}(un)\ud u\ud n.
\end{equation}
In particular, we may choose the $\eps$ in \eqref{eq:ep}, which is part of the representation $\eta_{\eps,\gamma}$ in \eqref{eta},
to be of the form:
$
\eps=\begin{pmatrix}
&I_{\ell-t}\\ I_{t}&
\end{pmatrix}.
$
Note that $\eps$ is one of the representatives of
$W_{\GL_{\alpha}\times\GL_{\ell-\alpha-t}\times\GL_{t}}\bks W_{\GL_{\ell}}$.
The following lemma will help us to eliminate more terms in \eqref{bfc0s}.

\begin{lem}\label{lm:open}
If $\beta>\max\cpair{j-\ell,0}$ and $\gamma w_{0}$ is not orthogonal to $V^{-}_{\ell,\beta}$ for $\gamma\in P'_{w}\bks G_{n-\ell}/H_{n-\ell}$,
then the inner integral in \eqref{bfc0s} has the property:
$$
\int_{N^{\eta}_{\ell}(F)\bks N^{\eta}_{\ell}(\BA)}\lam\phi(\epsilon_{0,\beta}\eta_{\eps,\gamma} unh)\psi^{-1}_{\ell,w_{0}}(un)\ud u=0
$$
for all choices of data.
\end{lem}

\begin{proof}
Consider the subgroup $S$ of $N_{\ell}$ consisting of elements of form
$$
\begin{pmatrix}
I_{t}&&&&\\
&I_{\ell-t}&y&*&\\
&&I_{m-2\ell}&y'&\\
&&&I_{\ell-t}&\\
&&&&I_{\beta}
\end{pmatrix},
$$
where $y=(0_{(\ell-t)\times (m-2\ell-\beta)}~y_{5})(w^{t'}_{q}\gamma)^{-1}$ with $t'$ as defined before.
and $y_{5}$ is of size $(\ell-t)\times \beta$. By $\ell-t=\ell-j+\beta>0$ and $\beta>0$, $y_{5}$ is not trivial. Since $\gamma w_{0}$ is not orthogonal to $V^{-}_{\ell,\beta}$, $\psi_{\ell,w_{0}}$ is not trivial on $S(\BA_{F})$.
By~\eqref{eq:stable-j}, $\phi$ is invariant on $(\epsilon_{0,\beta}\eta_{\eps,\gamma}) S(\BA)(\epsilon_{0,\beta}\eta_{\eps,\gamma})^{-1}$.
it follows that as an inner integration, the following integral
\begin{eqnarray*}
&&\int_{S(F)\bks S(\BA)}\lam\phi(\epsilon_{0,\beta}\eta_{\eps, \gamma} xunh)\psi^{-1}_{\ell,w_{0}}(x)\ud x\\
&=&\lam\phi(\epsilon_{0,\beta}\eta_{\eps, \gamma} unh)\cdot\int_{S(F)\bks S(\BA)}\psi^{-1}_{\ell,w_{0}}(x)\ud x
\end{eqnarray*}
is identically zero. This finishes the proof.
\end{proof}

We summarize the above calculation as

\begin{prop}\label{bfcc}
For $\Re(s)$ large, the Bessel-Fourier coefficient of the Eisenstein series, $\CB^{\psi_{\ell,w_{0}}}(E(\phi_{\tau\otimes\sigma},\lam))(h)$,
is equal to
\begin{equation*}
\sum_{\epsilon_{\beta}}\sum_{\eta}\sum_{\delta}
\int_{N^{\eta}_{\ell}(\BA)\bks N_{\ell}(\BA)}\int_{N^{\eta}_{\ell}(F)\bks N^{\eta}_{\ell}(\BA)}\lam\phi(\epsilon_{\beta}\eta\delta unh)\psi^{-1}_{\ell,w_{0}}(un)\ud u\ud n,
\end{equation*}
where
\begin{itemize}
\item $\eps_{\beta}=\eps_{0,\beta}\in\CE_{j,\ell}$, the set of representatives of all double cosets in $P_{j}(F)\bks G_n(F)/P_{\ell}(F)$,
with $\alpha=0$ and the properties that
if $\beta>\max\cpair{j-\ell,0}$, then $\gamma w_{0}$ is orthogonal to $V^{-}_{\ell,\beta}$ for $\gamma\in P'_{w}(F)\bks G_{n-\ell}(F)/H_{n-\ell}(F)$; or otherwise,
$\beta=\max\cpair{j-\ell,0}$;
\item $\eta=\diag(\eps,\gamma,\eps^{*})$ belongs to $\CN_{\beta,\ell,w_0}$ with $\alpha=0$, the set
of representatives of $P^{\epsilon_{\beta}}_{\ell}(F)\bks P_{\ell}(F)/ R_{\ell,w_{0}}(F)$ with $\eps=\begin{pmatrix}
&I_{\ell-t}\\ I_{t}&
\end{pmatrix}$ and $t=j-\beta$;
\item $\delta$ belongs to $H_{n-\ell}^{\eta}(F)\bks H_{n-\ell}(F)$.
\end{itemize}
\end{prop}

We are going to apply the formula in Proposition \ref{bfcc} to the calculation of the global zeta integral
$\CZ(s,\phi_{\tau\otimes\sigma},\varphi_\pi,\psi_{\ell,w_{0}})$ and use the cuspidality of $\varphi_\pi$ to prove that
the global zeta integral $\CZ(s,\phi_{\tau\otimes\sigma},\varphi_\pi,\psi_{\ell,w_{0}})$ is eulerian.

\subsection{Global zeta integrals}
By applying Proposition \ref{bfcc} to the global zeta integral in \eqref{bpes}, we get
\begin{align}
& \CZ(s,\phi_{\tau\otimes\sigma},\varphi_\pi,\psi_{\ell,w_{0}})\label{eq:P-E} \\
=&\int_{H_{n-\ell}(F)\bks H_{n-\ell}(\BA)}\CB^{\psi_{\ell,w_{0}}}(E(\phi_{\tau\otimes\sigma},s))(h)\varphi(h)\ud h\nonumber\\
=&\sum_{\epsilon_{\beta};\eta;\delta}
\int_{[H_{n-\ell}]}\varphi(h)\int_{N^{\eta}_{\ell}(\BA)\bks N_{\ell}(\BA)}\int_{[N^{\eta}_{\ell}]}
\lam\phi(\epsilon_{\beta}\eta\delta unh)\psi^{-1}_{\ell,\kappa}(un)\ud u\ud n\ud h \nonumber
\end{align}
where $[H_{n-\ell}]:=H_{n-\ell}(F)\bks H_{n-\ell}(\BA)$ and $[N^{\eta}_{\ell}]:=N^{\eta}_{\ell}(F)\bks N^{\eta}_{\ell}(\BA)$; and
the summations $\sum_{\epsilon_{\beta};\eta;\delta}$ and other conditions for the representatives are given in Proposition \ref{bfcc}.

We combine the summation on $\delta$ and the integration ${\rm d} h$ and obtain that $\CZ(s,\phi_{\tau\otimes\sigma},\varphi_\pi,\psi_{\ell,w_{0}})$
is equal to
\begin{equation}\label{eq:P-E1}
\sum_{\epsilon_{\beta};\eta}
\int_{H_{n-\ell}^{\eta}(F)\bks H_{n-\ell}(\BA)}\varphi(h)\int_{n}\int_{[N^{\eta}_{\ell}]}
\lam\phi(\epsilon_{\beta}\eta unh)\psi^{-1}_{\ell,\kappa}(un)\ud u\ud n\ud h,
\end{equation}
where the integration $\int_n$ is over $N^{\eta}_{\ell}(\BA)\bks N_{\ell}(\BA)$.
The following lemma is to make use of the cuspidality of $\varphi_\pi$.

\begin{lem}\label{lm:close}
Let $\alpha=0$ and $\gamma$ be a representative in $P'_{w}\bks G_{n-\ell}/H_{n-\ell}$. For a representative $\eta=\eta_{\eps,\gamma}$,
if the stabilizer $H^{\eta}_{n-\ell}$ is a proper maximal parabolic subgroup of $H_{n-\ell}$,
then the corresponding summand in \eqref{eq:P-E1} has the property:
$$
\int_{H^{\eta}_{n-\ell}(F)\bks H_{n-\ell}(\BA)}\varphi(h)\int_{n}
\int_{[N^{\eta}_{\ell}]}
\lam\phi(\epsilon_{\beta}\eta_{\eps,\gamma} unh)\psi^{-1}_{\ell,w_{0}}(un)\ud u\ud n\ud h=0
$$
for all choices of data.
\end{lem}

\begin{proof}
Let $H^{\eta}_{n-\ell}=M'U'$, where $U'$ is the unipotent radical of the parabolic subgroup
$H^{\eta}_{n-\ell}$ of $H_{n-\ell}$.
Since $\phi$ is left-invariant with respect to the image under the adjoint action by $\eps_{0,\beta}\eta_{\eps,\gamma}$ of
the unipotent radical $U'(\BA)$ of $H^{\eta}_{n-\ell}(\BA)$, we deduce that
\begin{align*}
&\int_{H^{\eta}_{n-\ell}(F)\bks H_{n-\ell}(\BA)}\varphi(h)\int_{n}
\int_{[N^{\eta}_{\ell}]}
\lam\phi(\epsilon_{\beta}\eta_{\eps,\gamma} unh)\psi^{-1}_{\ell,w_{0}}(un)\ud u\ud n\ud h\\
=&\int_{h}\int_{[U']}\varphi(u'h)\ud u'
\int_{n}\int_{[N^{\eta}_{\ell}]}
\lam\phi(\epsilon_{\beta}\eta_{\eps,\gamma}  un h)\psi^{-1}_{\ell,w_{0}}(un)\ud u\ud n\ud h\\
\end{align*}
where $\int_h$ is over $M'(F)U'(\BA)\bks H_{n-\ell}(\BA)$.
By the cuspidality of $\pi$, we have that
$$
\int_{U'(F)\bks U'(\BA)}\varphi(u'h)\ud u'=0,
$$
and hence the whole integral is zero. This proves the lemma.
\end{proof}

By Proposition \ref{bfcc}, the representatives $\eps_{\beta}$ have the restrictions that either $\beta=\max\{0,j-\ell\}$ or  $\beta>\max\{0,j-\ell\}$ with $\gamma w_{0}$
being orthogonal to $V^{-}_{\ell,\beta}$ for $\gamma\in P'_{w}(F)\bks G_{n-\ell}(F)/H_{n-\ell}(F)$.
Next, we discuss the double cosets decomposition $\gamma\in P'_{w}(F)\bks G_{n-\ell}(F)/H_{n-\ell}(F)$.

\begin{lem}[Proposition 4.4, \cite{GRS11}]\label{lm:P-L}
Let $X$ be a non-trivial totally isotropic subspace of $W_\ell$ and $P$ be the maximal parabolic subgroup of $G_{n-\ell}$
preserving $X$. Then
\begin{enumerate}
\item If $\dim_E X<\wi(W_\ell)$, then the set $P\bks G_{n-\ell}/H_{n-\ell}$ consists of two elements.
\item Assume that $\wi(w^{\perp}_{0})=\dim_E X=\wi(W_{\ell})$.
\begin{enumerate}
\item If $G_{n-\ell}$ is unitary, then $P\bks G_{n-\ell}/H_{n-\ell}$ consists of two elements.
\item If $G_{n-\ell}$  is orthogonal and $\dim W_\ell\geq 2\dim X+2$, then $P\bks G_{n-\ell}/H_{n-\ell}$
consists of two elements.
\item If $G_{n-\ell}$  is orthogonal and $\dim W_\ell=2\dim X+1$, then $P\bks G_{n-\ell}/H_{n-\ell}$
consists of three elements.
\end{enumerate}
\item If $\dim_EX=\wi(W_\ell)$ and $\wi(w^{\perp}_{0})=\dim_E X-1$,
then $P\bks G_{n-\ell}/H_{n-\ell}$  consists of one element.
\item If $\dim_E W_\ell=2\dim_E X$, then $\wi(w^{\perp}_{0})=\dim X-1$, and, in particular, $P\bks G_{n-\ell}/H_{n-\ell}$
consists of one element.
\end{enumerate}
\end{lem}

We consider the case when $G_{n-\ell}$ is not the $F$-split even orthogonal group or the case
when $G_{n-\ell}$ is the $F$-split even orthogonal group with $\ell+\beta<n$. In these cases,
we must have that $\dim X=\beta$.

If $\ell+\beta<\tilde{m}$, then $P'_{w}\bks G_{n-\ell}/ H_{n-\ell}$ consists of two elements. It remains to consider that
$\ell+\beta=\tilde{m}$. If $\ell+\beta<n$, we must have that $\ell+\beta=\tilde{m}<n$ and hence $G_{n-\ell}$ can not be the $F$-split even special orthogonal group.

In this case $\ell+\beta=\tilde{m}<n$, if $G_{n-\ell}$ is an $F$-quasisplit even unitary group, then $\wi(W_{\ell}\cap y^{\perp}_{\kappa})=\wi(W_{\ell})-1$ and
$P'_{w}\bks G_{n-\ell}/ H_{n-\ell}$ has only one element;
if $G_{n-\ell}$ is an odd special orthogonal group, then
$$
{^{\#}P'_{w}\bks G_{n-\ell}/ H_{n-\ell}}
=
\begin{cases}
3, & \text{if}\  \wi(w_{0}^{\perp}\cap W_{\ell})=\tilde{m}-\ell,\\
1, & \text{if}\  \wi(w_{0}^{\perp}\cap W_{\ell})=\tilde{m}-\ell-1;
\end{cases}
$$
and if $G_{n-\ell}$ is an $F$-quasisplit even special orthogonal group (with $\dim V_{0}=2$) or an $F$-quasisplit odd unitary group,
then
$$
{^{\#}P'_{w}\bks G_{n-\ell}/ H_{n-\ell}}
=
\begin{cases}
2, & \text{if}\  \wi(w_{0}^{\perp}\cap W_{\ell})=\tilde{m}-\ell,\\
1, & \text{if}\  \wi(w_{0}^{\perp}\cap W_{\ell})=\tilde{m}-\ell-1.
\end{cases}
$$

It remains to consider the case when $G_{n-\ell}$ is an $F$-split even special orthogonal group with $\ell+\beta=n$. In this case,
$P'_{w}\bks G_{n-\ell}/H_{n-\ell}$ consists of two elements.

Then, we apply Lemmas~\ref{lm:close} and \ref{lm:P-L} to find the nonvanishing summand in the summation \eqref{eq:P-E1}.

 For $\max\{0,j-\ell\}\leq \beta<\tilde{m}-\ell$, $P'_{w}\bks G_{n-\ell}/H_{n-\ell}$ consists of two elements $\gamma_{1}$ and $\gamma_{2}$
such that $\gamma_{1} w_{0}$ is orthogonal to $V^{-}_{\ell,\beta}$ and $\gamma_{2} w_{0}$ is not orthogonal to $V^{-}_{\ell,\beta}$.
If $\gamma w_{0}$ is orthogonal to $V^{-}_{\ell,\beta}$, the  stabilizer $H^{\gamma}_{n-\ell}=H^{\eta}_{n-\ell}$
is a maximal parabolic subgroup of $H_{n-\ell}$, which preserves the isotropic subspace
$w^{t}_{q}V^{+}_{\ell,\beta}\cap w_{0}^{\perp}$.

In this case, by Lemmas \ref{lm:open} and \ref{lm:close}, there may be left with nonzero summands in the
summation \eqref{eq:P-E1}, which are with the representative $\eps_{\beta}$ for $\beta=\max\{0,j-\ell\}$ and with
the representative $\eta=\eta_{\eps,\gamma}$ having the property that $\gamma w_{0}$ is not orthogonal to $V^{-}_{\ell,\beta}$.

For $\beta=\tilde{m}-\ell$, there are six different cases. Also, we have that $\beta=\tilde{m}-\ell>\max\{0,j-\ell\}$.

If $G_n$ is the $F$-split even special orthogonal group, then there are two
$(P_j,P_\ell)$-double cosets corresponding to the pair $(0,\beta)$ and the chosen representatives are $\epsilon_{0,\beta}$ and
$\tilde{\eps}_{0,\beta}$. For these two cases, their stabilizer preserves two maximal isotropic subspace of $W_{\ell}$
with different orientations, and
$P'_{w}\bks G_{n-\ell}/H_{n-\ell}$ consists of one element in both cases with its stabilizer
$H^{\gamma}_{n-\ell}=H^{\eta}_{n-\ell}$ being
a maximal parabolic subgroup. Hence by Lemma \ref{lm:close}, the corresponding summands are all zero.

If $G_n$ is not the $F$-split even special orthogonal group and $\wi(W_{\ell}\cap y^{\perp}_{\kappa})=\wi(W_{\ell})-1$, there is only one double coset whose stabilizer is a maximal parabolic subgroup of $H_{n-\ell}$.
Hence by Lemma \ref{lm:close}, the corresponding summand is zero.

If $\wi(W_{\ell}\cap y^{\perp}_{\kappa})=\wi(W_{\ell})$ and
$G_n$ is the odd unitary group or $F$-quasi-split even special orthogonal group, the stabilizers are similar to the case $\beta<\tilde{m}-\ell$ as discussed above.
Hence by Lemmas \ref{lm:open} and \ref{lm:close}, the corresponding summands are all zero.

If $G_n$ is the odd special orthogonal group and $\wi(W_{\ell}\cap y^{\perp}_{\kappa})=\wi(W_{\ell})-1$,  then $P'_{w}\bks G_{n-\ell}/H_{n-\ell}$ consists of three elements and the representatives are chosen
in \cite[(4.33)]{GRS11}. Two stabilizers are maximal parabolic subgroups of $H_{n-\ell}$, and
the third representative $\gamma$ satisfies the property that $\gamma w_{0}$ is not orthogonal to $V^{-}_{\ell,\beta}$.
Hence by Lemmas \ref{lm:open} and \ref{lm:close}, the corresponding summands are all zero.

By the discussions above, we deduce that the corresponding summands are all zero, because of Lemmas \ref{lm:open} and \ref{lm:close}.

In conclusion, we are left with the case where $\beta=\max\{0,j-\ell\}$ and $\gamma$ with the property that
the corresponding stabilizer is not a proper maximal parabolic subgroup of $H^{\gamma}_{n-\ell}$, i.e. $\gamma w_{0}$ is not orthogonal to $V^{-}_{\ell,\beta}$.

In this case, the representative $\eta=\eta_{\eps,\gamma}$ is uniquely determined by $\beta=\max\{0,j-\ell\}$. In fact,
if $j\leq \ell$, then $\beta=0$. It follows that $\eta=\eta_{\eps,\gamma}$ with $\gamma=I_{m-2\ell}$ and
\begin{equation}\label{m:eps}
\eps=\begin{pmatrix} &I_{\ell-j}\\ I_{j}& \end{pmatrix};
\end{equation}
and if $j>\ell$, then $\beta=j-\ell$. It implies that $\eta=\eta_{\eps,\gamma}$ with $\eps=I_{\ell}$ and
\begin{equation}\label{m:gamma}
\gamma=\begin{pmatrix}
&I_{j-\ell}&&&\\I_{\tilde{m}-j}&&&&\\&&I_{V_{0}}&&\\&&&&I_{\tilde{m}-j}\\&&&I_{j-\ell}&
\end{pmatrix}.
\end{equation}
Therefore, we are left with only one summand in the summation \eqref{eq:P-E1} with the above representative, accordingly.

Next we are going to write the only integral more explicitly and get ready to prove that it is eulerian in the next subsection.

If $j\leq \ell$, then $\beta=0$. In this case we have that $P'_{w}=G_{n-\ell}$ and $H^\gamma_{n-\ell}=H_{n-\ell}$
with $\eps$ and $\gamma$ given above. Then the global zeta integral in \eqref{eq:P-E1} has the following expression:
\begin{align}
\CZ(s,\phi_{\tau\otimes\sigma},\varphi_\pi,\psi_{\ell,w_{0}}) \label{eq:P-0}
=&\int_{[H_{n-\ell}]}\varphi(h)\int_{N^{\eta}_{\ell}(\BA)\bks N_{\ell}(\BA)}\nonumber\\
&\int_{[N^{\eta}_{\ell}]}
\lam\phi(\epsilon_{0,0}\eta unh)\psi^{-1}_{\ell,w_0}(un)\ud u\ud n\ud h.
\end{align}
where $[H_{n-\ell}]:=H_{n-\ell}(F)\bks H_{n-\ell}(\BA)$ and $[N^{\eta}_{\ell}]:=N^{\eta}_{\ell}(F)\bks N^{\eta}_{\ell}(\BA)$.
The stabilizers are, respectively, given by
\begin{equation} \label{eq:0-R-l}
R^{\eta}_{\ell,w_0}=\begin{pmatrix}
c&0&0&0&0\\&b&y_{4}&z_{4}&0\\&&e&y'_{4}&0\\&&&b^{*}&0\\&&&&c^{*}
\end{pmatrix}
\end{equation}
with $c,c^*$ being of size $j\times j$, $b,b^*$ of size $(\ell-j)\times(\ell-j)$, and
$e$ of size $(m-2\ell)\times(m-2\ell)$;
and
\begin{equation}\label{eq:0-R-j}
(\eps_{0,0}\eta_{\eps,\gamma})R^{\eta}_{\ell,w_0}(\eps_{0,0}\eta_{\eps,\gamma})^{-1}=
\begin{pmatrix}
c^{*}&0&0&0&0\\&b&y_{4}&z_{4}&0\\&&e&y'_{4}&0\\&&&b^{*}&0\\&&&&c
\end{pmatrix}
\end{equation}
with $c\in Z_{j}$ and $b\in Z_{\ell-j}$. ($Z_f$ is the maximal upper-triangular unipotent subgroup of $\GL_f$.)

If $j>\ell$, then $\beta=j-\ell$. In this case, $\eps=I_{\ell}$ and $\gamma$ is given in \eqref{m:gamma}.
The double coset decomposition $P'_{w}\bks G_{n-\ell}/H_{n-\ell}$ produces two representatives
which, as given in \cite[Section 4.4]{GRS11}, are $\gamma=I_{m-2\ell}$ and the $\gamma$ as given in \eqref{m:gamma}.

For the representative $\gamma=I_{m-2\ell}$, the corresponding stabilizer $H^\gamma_{n-\ell}$ is a proper maximal
parabolic subgroup. Then, the corresponding integral in~\eqref{eq:P-E1} is zero by Lemma~\ref{lm:close}.

Now for the $\gamma$ as given in \eqref{m:gamma}, we have that the global zeta integral is expressed as
\begin{align}
\CZ(s,\phi_{\tau\otimes\sigma},\varphi_\pi,\psi_{\ell,w_{0}}) \label{eq:P-j-l}
=&\int_{H^{\eta}_{n-\ell}(F)\bks H_{n-\ell}(\BA)}\varphi(h)
\int_{N^{\eta}_{\ell}(\BA)\bks N_{\ell}(\BA)}\nonumber\\
&\int_{[N^{\eta}_{\ell}]}
\lam\phi(\epsilon_{\beta}\eta unh)\psi^{-1}_{\ell,w_0}(un)\ud u\ud n\ud h,
\end{align}
where $[N^{\eta}_{\ell}]=N^{\eta}_{\ell}(F)\bks N^{\eta}_{\ell}(\BA)$.
The stabilizers are given, respectively,
\begin{equation}\label{eq:j-l-R-l}
\eta^{-1}_{\eps,\gamma}R^{\eta}_{\ell,w_0}\eta_{\eps,\gamma}=\begin{pmatrix}
c&0&0&y_{6}&0\\
&d&u&v&y'_{6}\\
&&e&u'&0\\
&&&d^{*}&0\\
&&&&c^{*}
 \end{pmatrix}
 \end{equation}
with $c,c^*$ being of size $\ell\times\ell$, $d,d^*$ of size $(j-\ell)\times(j-\ell)$, and $e$ of size $(m-2j)\times(m-2j)$;
and
\begin{equation} \label{eq:j-l-R-j}
 (\epsilon_{0,\beta}\eta_{\eps,\gamma})R^{\eta}_{\ell,w_0}(\epsilon_{0,\beta}\eta_{\eps,\gamma})^{-1}=
\begin{pmatrix}
d& y'_{6}&u&0&v\\&c^{*}&0&0&0\\&&e&0&u'\\&&&c&y_{6}\\&&&&d^{*}
\end{pmatrix}
\end{equation}
where $c\in Z_{\ell}$.

We conclude this subsection with the following proposition which summarizes the calculations discussed up to this
point.

\begin{prop}\label{pro:GG}
Take notation as given above. If $j\leq \ell$, then $\beta=0$ and the global zeta integral has the following expression:
\begin{align*}
\CZ(s,\phi_{\tau\otimes\sigma},\varphi_\pi,\psi_{\ell,w_{0}})
=&\int_{[H_{n-\ell}]}\varphi(h)\int_{N^{\eta}_{\ell}(\BA)\bks N_{\ell}(\BA)}\\
&\int_{[N^{\eta}_{\ell}]}
\lam\phi(\epsilon_{0,0}\eta unh)\psi^{-1}_{\ell,w_0}(un)\ud u\ud n\ud h,
\end{align*}
where $[H_{n-\ell}]:=H_{n-\ell}(F)\bks H_{n-\ell}(\BA)$ and $[N^{\eta}_{\ell}]:=N^{\eta}_{\ell}(F)\bks N^{\eta}_{\ell}(\BA)$;
and with $\eta=\eta_{\eps,\gamma}$ given explicitly above.
If $j>\ell$, then $\beta=j-\ell$ and the global zeta integral has the following expression:
\begin{align*}
\CZ(s,\phi_{\tau\otimes\sigma},\varphi_\pi,\psi_{\ell,w_{0}})
=&\int_{H^{\eta}_{n-\ell}(F)\bks H_{n-\ell}(\BA)}\varphi(h)
\int_{N^{\eta}_{\ell}(\BA)\bks N_{\ell}(\BA)}\\
&\int_{[N^{\eta}_{\ell}]}
\lam\phi(\epsilon_{0,\beta}\eta unh)\psi^{-1}_{\ell,w_0}(un)\ud u\ud n\ud h,
\end{align*}
where $[N^{\eta}_{\ell}]=N^{\eta}_{\ell}(F)\bks N^{\eta}_{\ell}(\BA)$; and with $\eta=\eta_{\eps,\gamma}$ given explicitly above.
\end{prop}

We are going to show that the global zeta integrals are eulerian based on Proposition \ref{pro:GG}.
This is done for the two cases, separately.

\subsection{Eulerian products: $0<\ell< j$ case}
In this case, we have that $\beta=j-\ell$. By Proposition~\ref{pro:GG}, the global zeta integral
$\CZ(s,\phi_{\tau\otimes\sigma},\varphi_\pi,\psi_{\ell,w_{0}})$ is equal to the following integral
\begin{equation}\label{gzi31}
\int_{H^{\eta}_{n-\ell}(F)\bks H_{n-\ell}(\BA)}\varphi(h)
\int_{N^{\eta}_{\ell}(\BA)\bks N_{\ell}(\BA)}
\int_{[N^{\eta}_{\ell}]}
\lam\phi(\epsilon_{0,\beta}\eta unh)\psi^{-1}_{\ell,w_0}(un)\ud u\ud n\ud h,
\end{equation}
where $[N^{\eta}_{\ell}]=N^{\eta}_{\ell}(F)\bks N^{\eta}_{\ell}(\BA)$; and with $\eta=\eta_{\eps,\gamma}$ given explicitly above.

First, we want to understand the Fourier coefficient of $\lam\phi$:
\begin{equation}\label{fceta1}
\int_{[N^{\eta}_{\ell}]}
\lam\phi(\epsilon_{0,\beta}\eta uh)\psi^{-1}_{\ell,w_0}(u)\ud u.
\end{equation}
By conjugating the element $\eps_{0,\beta}\eta$ across the variable $u$ and changing the variable by
$$
(\eps_{0,\beta}\eta)u(\eps_{0,\beta}\eta)^{-1}\mapsto \hat{z'},
$$
the Fourier coefficient in \eqref{fceta1} reduces to
\begin{equation}\label{fceta2}
\int_{[Z_\ell']}
\lam\phi(\hat{z'}\epsilon_{0,\beta}\eta h)\psi^{-1}_{\ell,w_0}((\eps_{0,\beta}\eta)^{-1}\hat{z'}(\eps_{0,\beta}\eta))\ud z',
\end{equation}
where $Z_\ell'=(\eps_{0,\beta}\eta)N_\ell^\eta(\eps_{0,\beta}\eta)^{-1}$, whose elements $z'$ are of form
\begin{equation}\label{eq:Z'}
z'=\begin{pmatrix} I_{\beta}&y\\&z\end{pmatrix} \in \Res_{E/F}(\GL_{j})
\end{equation}
with $z\in Z_{\ell}$, where $g\in\Res_{E/F}(\GL_j)$ is identified with its embedding $\hat{g}=(g,I_{m-2j})$
into the Levi subgroup $\Res_{E/F}(\GL_j)\times G_{n-j}$ of
$G_n$. It follows from the choice of the representatives $\eps_{0,\beta}$ and $\eta$ that the character has following expression:
\begin{equation}\label{eq:psi''}
\psi^{-1}_{\ell,w_0}((\eps_{0,\beta}\eta)^{-1}\hat{z'}(\eps_{0,\beta}\eta))
=
\psi(z_{1,2}+\cdots+z_{\ell-1,\ell}+(-1)^{m+1}\frac{\kappa}{2}y_{\beta,1}),
\end{equation}
where $z=(z_{e,f})_{\ell\times\ell}$. If we write elements $z'$ of $Z'_\ell$ as $z'=(z'_{e,f})_{j\times j}$,
then this character can be written as
\begin{equation}\label{chz'}
\psi_{Z'_\ell,\kappa}(z')
:=
\psi((-1)^{m+1}\frac{\kappa}{2}z_{\beta,\beta+1}+z_{\beta+1,\beta+2}+\cdots+z_{j-1,j}).
\end{equation}
In this way, the Fourier coefficient in \eqref{fceta2} can be written as
\begin{equation}\label{fcz'}
\phi_\lam^{\psi_{Z'_\ell,\kappa}}(h)
:=
\int_{[Z_\ell']}
\lam\phi(\hat{z'}h)\psi_{Z'_\ell,\kappa}(z')\ud z'.
\end{equation}
Hence the global zeta integral $\CZ(s,\phi_{\tau\otimes\sigma},\varphi_\pi,\psi_{\ell,w_{0}})$, which is expressed as in \eqref{gzi31},
is equal to the following integral
\begin{equation}\label{gzi32}
\int_{H^{\eta}_{n-\ell}(F)\bks H_{n-\ell}(\BA)}\varphi(h)
\int_{N^{\eta}_{\ell}(\BA)\bks N_{\ell}(\BA)}
\phi_\lam^{\psi_{Z'_\ell,\kappa}}(\eps_{0,\beta}\eta nh)
\psi^{-1}_{\ell,w_0}(n)\ud n\ud h,
\end{equation}
with $\eta=\eta_{\eps,\gamma}$ given explicitly above.

Next, we want to understand the structure of the subgroup $H^{\eta}_{n-\ell}$.
By \eqref{hgamma}, $H^{\eta}_{n-\ell}=H_{n-\ell}\cap\gamma^{-1}P'_w\gamma$ with
$\eta=\eta_{\eps,\gamma}$, and
$P'_{w}=G_{n-\ell}\cap P^{\eps_{0,\beta}}_{\ell}$ is the parabolic subgroup of $G_{n-\ell}$, preserving
the totally isotropic subspace $V^{+}_{\ell,\beta}$ as in \eqref{vpm}.
Denote by
$$
P^{'\eta}_{w}=P'_{w}\cap \eta H_{n-\ell}\eta^{-1}=P'_{w}\cap \gamma H_{n-\ell}\gamma^{-1}.
$$
Then the elements of $P^{'\eta}_{w}$ are of form:
\begin{equation}\label{eq:stab-P-L}
\begin{pmatrix}
I_{\ell}&&&&&&\\
&d&d_{1}&u&v_{1}&v&\\
&&1&0&0&v'_{1}&\\
&&&e&0&u'&\\
&&&&1&d'_{1}&\\
&&&&&d^{*}&\\
&&&&&&I_{\ell}
\end{pmatrix}
\end{equation}
with $d_{1}+(-1)^{m+1}\frac{\kappa}{2}v_{1}=0$,
where $d_{1}$ and $v_{1}$ are column vectors of size $\beta-1$; $d,d^*$ are of size $(\beta-1)\times(\beta-1)$; and
$e$ belongs to $G_{n-j}$.  Hence we have
\begin{equation}
P^{'\eta}_{w}=(\GL(V^{+}_{\ell,\beta-1})\times G_{n-j})\rtimes U^{\eta}(V^{+}_{\ell,\beta-1}),
\end{equation}
where $U^{\eta}(V^{+}_{\ell,\beta-1})$ is the subgroup of $U(V^{+}_{\ell,\beta-1})$ consisting elements which fixes the vector
$\gamma y_{\kappa}$.
Here $U(V^{+}_{\ell,\beta-1})$ is the unipotent radical of the parabolic subgroup $P(V^{+}_{\ell,\beta-1})$ of
$G_{n-\ell}$ preserving the totally isotropic subspace $V^{+}_{\ell,\beta-1}$.

Let $Q_{\beta-1,\eta}$ be the parabolic subgroup of $H_{n-\ell}$, which preserves the totally isotropic subspace
$(\eta^{-1} V^{+}_{\ell,\beta})\cap y^{\perp}_{\kappa}$ of $W_{\ell}\cap y^{\perp}_{\kappa}$ and has the Levi decomposition
$$
Q_{\beta-1,\eta}=L_{\beta-1,\eta}V_{\beta-1,\eta}.
$$
Recall that the space $W_{\ell}\cap y_{\kappa}^{\perp}$ has the polar decomposition
$$
W_{\ell}\cap y_{\kappa}^{\perp}
=V^+_{\ell,\tilde{m}-\ell-1}\oplus W_0\oplus V^-_{\ell,\tilde{m}-\ell-1},
$$
where $W_0$ is a non-degenerate subspace of $W_{\ell}\cap y_{\kappa}^{\perp}$
with the same anisotropic kernel as $W_{\ell}\cap y_{\kappa}^{\perp}$
and with $\dim_EW_0=\dim_EV_0+1\leq 3$. In particular, if
$w_0=y_\kappa=e_{\tilde{m}}+(-1)^{m+1}\frac{\kappa}{2}e_{-{\tilde{m}}}$,
then $W_0=\Span\cpair{y_{-\kappa}}\oplus V_0$, otherwise, $W_0$ has Witt index $1$.
Then it is easy to check that
$$
(\eta^{-1} V^{+}_{\ell,\beta})\cap y^{\perp}_{\kappa}
=\Span\cpair{e_{\tilde{m}-j+\ell+1},\dots,e_{\tilde{m}-1}}
=V^+_{\tilde{m}-\beta,\beta-1},
$$
and
$$
L_{\beta-1,\eta}=\GL(V^+_{\tilde{m}-\beta,\beta-1})\times H_{n-j+1},
$$
where $H_{n-j+1}:=\RU(q_{_{W_{j-1}\cap y^{\perp}_{\kappa}}})$ with
$$
W_{j-1}\cap y^{\perp}_{\kappa}
= V^{+}_{\ell,\tilde{m}-j}\oplus W_0\oplus V^{-}_{\ell, \tilde{m}-j}.
$$
It follows that
$$
\GL(V^+_{\tilde{m}-\beta,\beta-1})=\GL((\eta^{-1} V^{+}_{\ell,\beta})\cap y^{\perp}_{\kappa})=\eta^{-1}\GL(V^{+}_{\ell,\beta-1})\eta\subset
H^{\eta}_{n-\ell},
$$
and
$$
V_{\beta-1,\eta}=\eta^{-1}U^{\eta}(V^{+}_{\ell,\beta-1})\eta\subset H^{\eta}_{n-\ell}.
$$
It is easy to check that
$$
\eta^{-1}W_{j}=V^{+}_{\ell,\tilde{m}-j}\oplus V_{0}\oplus V^{-}_{\ell, \tilde{m}-j}=y_{-\kappa}^{\perp}\cap (W_{j-1}\cap y^{\perp}_{\kappa}).
$$
Hence we have
$$
\RU(q_{_{\eta^{-1}W_{j}}})
=\eta^{-1}\RU(q_{_{W_j}})\eta
=\eta^{-1}G_{n-j}\eta\subset H^{\eta}_{n-\ell}.
$$
Putting together all these subgroups, we obtain the structure of $H^{\eta}_{n-\ell}$:
\begin{equation}\label{heta}
H^{\eta}_{n-\ell}
=(\GL(V^+_{\tilde{m}-\beta,\beta-1})\times \RU(q_{_{\eta^{-1}W_{j}}}))\rtimes V_{\beta-1,\eta}.
\end{equation}

Finally, we are ready to consider partial Fourier expansion of the cuspidal automorphic forms $\varphi_\pi$ on
$H_{n-\ell}(\BA)$. Let $Z^\eta_{\ell,\beta-1}$ be the maximal unipotent subgroup of $\GL(V^+_{\tilde{m}-\beta,\beta-1})$ consisting
of elements of following type:
$$
\eta^{-1}\begin{pmatrix}
I_{\ell}&&&&\\
&d&&&\\
&&I_{m-2j+2}&&\\
&&&d^{*}&\\
&&&&I_{\ell}
\end{pmatrix}\eta
$$
with $d\in Z_{\beta-1}$. Then $N^\eta_{\ell,\beta-1}=Z^\eta_{\ell,\beta-1}V_{\beta-1,\eta}$ is a unipotent subgroup of $H_{n-\ell}$ of the
type as defined in \eqref{nell} with the corresponding character defined as in \eqref{chw0}, but using $y_{-\kappa}$.
It is easy to check that the corresponding stabilizer $H^{y_{-\kappa}}_{n-j+1}$ is equal to $\RU(q_{_{\eta^{-1}W_{j}}})$, which
is isomorphic to $G_{n-j}$.

Define $C_{\beta-1,\eta}:=V_{\beta-1,\eta}\cap V_{\beta,\eta}$, which is also equal to
$$
\cpair{u\in V_{\beta-1,\eta}\mid u\cdot e_{\tilde{m}}=e_{\tilde{m}}}
$$
and is a normal subgroup of $H^\eta_{n-\ell}$. It follows that
$$
C_{\beta-1,\eta}\bks H^{\eta}_{n-\ell}\cong P^{1}_{\beta}\times H^{y_{-\kappa}}_{n-j+1},
$$
where $P^{1}_{\beta}$ is the mirabolic subgroup of $\Res_{E/F}(\GL_{\beta})$ given by
$$
P^{1}_{\beta}=\cpair{\begin{pmatrix} d&d_{1}\\0&1\end{pmatrix}\in \Res_{E/F}(\GL_{\beta})}.
$$
Going back to the expression \eqref{gzi32} of
$\CZ(s,\phi_{\tau\otimes\sigma},\varphi_\pi,\psi_{\ell,w_{0}})$, the inner integral
\begin{equation}\label{Phi}
\Phi(h):=\int_{N^{\eta}_{\ell}(\BA)\bks N_{\ell}(\BA)}
\phi_\lam^{\psi_{Z'_\ell,\kappa}}(\eps_{0,\beta}\eta nh)
\psi^{-1}_{\ell,w_0}(n)\ud n
\end{equation}
as function in $h$, is left $C_{\beta-1,\eta}(\BA)$-invariant.
We recall that $N_{\ell}$ consists of elements of form
$$
\begin{pmatrix}
c&x_{1}&x_{2}&x_{3}&y_{6}&x_{4}&x_{5}\\
&I_{\tilde{m}-j}&&&&&x'_{4}\\
&&I_{j-\ell}&&&&y'_{6}\\
&&&I_{m-2\tilde{m}}&&&x'_{3}\\
&&&&I_{j-\ell}&&x'_{2}\\
&&&&&I_{\tilde{m}-j}&x'_{1}\\
&&&&&&c^{*}
\end{pmatrix}
$$
where $c\in Z_{\ell}$ and the stabilizer $N^{\eta}_{\ell}$ consists element of form
$$
\begin{pmatrix}
c&0&0&0&y_{6}&0&0\\
&I_{\tilde{m}-j}&&&&&0\\
&&I_{j-\ell}&&&&y'_{6}\\
&&&I_{m-2\tilde{m}}&&&0\\
&&&&I_{j-\ell}&&0\\
&&&&&I_{\tilde{m}-j}&0\\
&&&&&&c^{*}
\end{pmatrix}.
$$
Then $N^{\eta}_{\ell}\bks N_{\ell}$ is isomorphic to the subgroup consisting of elements of form
$$
\begin{pmatrix}
I_{\ell}&x_{1}&x_{2}&0&x_{3}\\
&I_{j-\ell}&&&0\\
&&I_{m-2j}&&x'_{2}\\
&&&I_{j-\ell}&x'_{1}\\
&&&&I_{\ell}
\end{pmatrix}
$$
and $\psi_{\ell,\kappa}$ is not trivial on $x_{2}$. In details,
$\psi_{\ell,\kappa}\vert_{N^{\eta}_{\ell}\bks N_{\ell}}=\psi((x_{2})_{\ell,j-\ell}).$

The stabilizer $(\eps_{0,\beta}\eta)N^{\eta}_{\ell}(\eps_{0,\beta}\eta)^{-1}$ in $P_{j}$
consists of elements of form
$$
\begin{pmatrix}
I_{j-\ell}&y'_{6}&&&\\
&c^{*}&&&\\
&&e&&\\
&&&c&y_{6}\\
&&&&I_{j-\ell}
\end{pmatrix}
$$
The integral domain $N^{\eta}_{\ell}\bks N_{\ell}$ under adjoint action of $\eps_{0,\beta}\eta$ is
a subgroup $U^{-}_{j,\eta}$ of $U^{-}_{j}$ (opposite of the unipotent radical $U_{j}$) consisting
of elements of form
\begin{equation}\label{eq:U-j-eta}
\begin{pmatrix}
I_{j-\ell}&&&&\\
&I_{\ell}&&&\\
&x'_{2}&I_{m-2j}&&\\
x_{1}&x_{3}&x_{2}&I_{\ell}&\\
&x'_{1}&&&I_{j-\ell}
\end{pmatrix}.
\end{equation}
Denote by $\psi_{(j+1,j)}$ the character over
$(\eps_{0,\beta}\eta) N^{\eta}_{\ell}\bks N_{\ell} (\eps_{0,\beta}\eta)^{-1}$, given by
$\psi_{(j+1,j)}(n)=\psi(n_{j+1,j})$. Indeed, this character is associated to the negative root of
the simple root $e_{j}-e_{j+1}$ corresponding to the maximal parabolic subgroup $P_{\ell}$.

Recall that $\eta^{-1}C_{\beta-1,\eta}\eta$ consists of elements of form
$$
\begin{pmatrix}
I_{\ell}&&&&&&\\
&I_{\beta-1}&0&0&u&0&\\
&&1&0&0&u'&\\
&&&I_{m-2j}&0&0&\\
&&&&1&0&\\
&&&&&I_{\beta-1}&\\
&&&&&&I_{\ell}
\end{pmatrix}.
$$
It follows that $N^{\eta}_{\ell,\beta-1}=Z^{\eta}_{\ell,\beta-1}V_{\beta-1,\eta}=Z_{\beta}C_{\beta-1,\eta}$.
As a subgroup of $P_{j}$,
the stabilizer $(\eps_{0,\beta}\eta)N^{\eta}_{\ell,\beta-1}(\eps_{0,\beta}\eta)^{-1}$ consists of
elements of form
$$
\begin{pmatrix}
d&d_{1}&0&u&0&v_{1}&v\\
&1&&&&&v'_{1}\\
&&I_{\ell}&&&&0\\
&&&I_{m-2j}&&&u'\\
&&&&I_{\ell}&&0\\
&&&&&1&d'_{1}\\
&&&&&&d^{*}
\end{pmatrix},
$$
where $d\in Z_{\beta-1}$. Note that $(\eps_{0,\beta}\eta)Z_{\beta}(\eps_{0,\beta}\eta)^{-1}$
consists of elements of the above form with all matrices being zero except $d$ and $d_{1}$.

It follows that the expression in \eqref{gzi32} of the global zeta integral
$\CZ(s,\phi_{\tau\otimes\sigma},\varphi_\pi,\psi_{\ell,w_{0}})$ is equal to
\begin{equation}\label{gzi33}
\int_{H^{\eta}_{n-\ell}(F)C_{\beta-1,\eta}(\BA)\bks H_{n-\ell}(\BA)}\Phi(h)\int_{[C_{\beta-1,\eta}]}
\varphi_\pi(ch)\ud c \ud h,
\end{equation}
where $[C_{\beta-1,\eta}]:=C_{\beta-1,\eta}(F)\bks C_{\beta-1,\eta}(\BA)$, as before.

We denote the inner integration $\int_c$ by
$$
\varphi_\pi^{C_{\beta-1,\eta}}(h)
=\int_{[C_{\beta-1,\eta}]}\varphi_\pi(ch)\ud c.
$$
The integral in \eqref{gzi33} becomes
\begin{equation}\label{gzi34}
\int_{H^{\eta}_{n-\ell}(F)C_{\beta-1,\eta}(\BA)\bks H_{n-\ell}(\BA)}\Phi(h)
\varphi_\pi^{C_{\beta-1,\eta}}(h)\ud h.
\end{equation}

Now we are in the standard step in the global unfolding process using partial Fourier expansion
along the mirabolic subgroup $P^{1}_{\beta}$.
Both functions $\Phi(h)$ and $\varphi^{C_{\beta-1,\eta}}(h)$ are automorphic on $P^{1}_{\beta}(\BA)$ and
$\varphi^{C_{\beta-1,\eta}}(h)$ is
cuspidal because of the cuspidality of $\varphi_\pi(h)$. Following the standard Fourier expansion
of cuspidal automorphic forms
on general linear group (\cite{S74} and \cite{PS79}, see also \cite{JL12}), we have
\begin{equation} \label{eq:FE}
\varphi_\pi^{C_{\beta-1,\eta}}(h)
=\sum_{d\in Z_{\beta-1}(F)\bks \GL_{\beta-1}(E)}\CB^{\psi^{-1}_{\beta-1,y_{-\kappa}}}(\varphi_\pi)
\ppair{\eta^{-1}\begin{pmatrix}I_{\ell}&&\\&d&\\&&1 \end{pmatrix}^{\wedge}\eta h},
\end{equation}
which converges absolutely and uniformly in $g$ varying in compact subsets.
Recall that the Bessel-Fourier coefficient with respect to $\psi_{\beta-1,y_{-\kappa}}$ is defined as
in \eqref{bfc} by
$$
\CB^{\psi^{-1}_{\beta-1,y_{-\kappa}}}(\varphi_\pi)(h)
=
\int_{N^\eta_{\ell,\beta-1}(F)\bks N^\eta_{\ell,\beta-1}(\BA)}\varphi_{\pi}(nh)\psi_{\beta-1,y_{-\kappa}}(n)\ud n.
$$
By using \eqref{eq:FE}, the expression \eqref{gzi34} of
$\CZ(s,\phi_{\tau\otimes\sigma},\varphi_\pi,\psi_{\ell,w_{0}})$ is equal to
\begin{equation}\label{gzi35}
\int_{Z_{\beta}(F)H^{y_{-\kappa}}_{n-j+1}(F)C_{\beta-1,\eta}(\BA)\bks H_{n-\ell}(\BA)}\Phi(h)
\CB^{\psi^{-1}_{\beta-1,y_{-\kappa}}}(\varphi_\pi)(h)\ud h.
\end{equation}
By pulling out the integration on $Z_{\beta}(F)\bks Z_{\beta}(\BA)$ and using the fact that
$\CB^{\psi^{-1}_{\beta-1,y_{-\kappa}}}(\varphi_\pi)(h)$ is
left $(Z_{\beta}(\BA),\psi^{-1}_{\beta-1,y_{-\kappa}})$-quasi-invariant, the integral in \eqref{gzi35} is
equal to
\begin{equation}\label{gzi36}
\int_{H^{y_{-\kappa}}_{n-j+1}(F)N^\eta_{\ell,\beta-1}(\BA)\bks H_{n-\ell}(\BA)}
\CB^{\psi^{-1}_{\beta-1,y_{-\kappa}}}(\varphi_\pi)(h)
\int_{[Z_{\beta}]}\Phi(zh)\psi^{-1}_{\beta-1,y_{-\kappa}}(z)\ud z\ud h,
\end{equation}
where $N^\eta_{\ell,\beta-1}=Z_{\beta}C_{\beta-1,\eta}$ is as before and $[Z_{\beta}]:=Z_{\beta}(F)\bks Z_{\beta}(\BA)$.

The inner integration
$$
\int_{[Z_{\beta}]}\Phi(zh)\psi^{-1}_{\beta-1,y_{-\kappa}}(z)\ud z
$$
can be calculated more explicitly. By \eqref{Phi}, it is equal to
\begin{equation}\label{Phi1}
\int_{[Z_{\beta}]}\int_{N^{\eta}_{\ell}(\BA)\bks N_{\ell}(\BA)}
\phi_\lam^{\psi_{Z'_\ell,\kappa}}(\eps_{0,\beta}\eta nzh)
\psi^{-1}_{\ell,w_0}(n)\ud n\psi^{-1}_{\beta-1,y_{-\kappa}}(z)\ud z.
\end{equation}

The element $(\eps_{0,\beta}\eta)z(\eps_{0,\beta}\eta)^{-1}$ is given as above.
Combining this subgroup with $N^{\eta}_{\ell}$, one obtains a subgroup
$(\eps_{0,\beta}\eta)N^{\eta}_{\ell}Z_{\beta}(\eps_{0,\beta}\eta)^{-1}$ of $P_{j}$ consisting of elements of form
\begin{equation}
\begin{pmatrix}
d&d_{1}&(y'_{6})_{*,*}&&&&\\
&1&(y'_{6})_{\beta,*}&&&&\\
&&c^{*}&&&&\\
&&&I_{m-2j}&&&\\
&&&&c&(y_{6})_{*,\beta}&(y_{6})_{*,*}\\
&&&&&1&d'_{1}\\
&&&&&&d
\end{pmatrix}.
\end{equation}
Define
\begin{align*}
\phi^{Z_{j},\kappa}(h)=\int_{[Z_{\beta}]}\phi^{\psi_{Z'_{\ell},\kappa}}_{\lam}(\eps_{0,\beta}\eta z h)\psi^{-1}_{\beta-1,y_{-\kappa}}\ud z.
\end{align*}
Then,
$$
\phi^{Z_{j},\kappa}(h)=\int_{[Z_{j}]}\phi_{\lam}(zh)\psi_{Z_{j},\kappa}\ud z,
$$
where $\psi_{Z_{j},\kappa}(z)$ is given by
\begin{equation}
\psi(-z_{1,2}-\cdots-z_{\beta-1,\beta}+(-1)^{m+1}\frac{\kappa}{2}z_{\beta,\beta+1}+z_{\beta+1,\beta+2}+\cdots+z_{j-1,j})
\end{equation}
with $\beta=j-\ell$. Hence,
\begin{eqnarray*}
\int_{[Z_{\beta}]}\Phi(zh)\ud z
&=&
\int_{N_{\ell}^{\eta}(\BA)\bks N_{\ell}(\BA)}\phi^{Z_{j},\kappa}(\eps_{0,\beta}\eta n h)
\psi^{-1}_{\ell,\kappa}(n)\ud n\\
&=&
\int_{U^{-}_{j,\eta}(\BA)}\phi^{Z_{j},\kappa}(n\eps_{0,\beta}\eta h)\psi_{j+1,j}(n)\ud n.
\end{eqnarray*}
Denote the last integral by
$$
\CJ_{\ell,\kappa}(\phi^{Z_{j,\kappa}})(h)=\int_{U^{-}_{j,\eta}(\BA)}\phi^{Z_{j},\kappa}(nh)\psi_{j+1,j}(n)\ud n.
$$
Recall that the group $U^{-}_{g,\eta}$ consists of elements of form \eqref{eq:U-j-eta}.

Therefore, we obtain, from \eqref{gzi33} and \eqref{gzi34}, that the global zeta integral
$\CZ(s,\phi_{\tau\otimes\sig},\varphi_{\pi},\psi_{\ell,w_{0}})$ equals
$$
\int_{R^{\eta}_{\ell,\beta-1}(\BA)\bks H_{n-\ell}(\BA)}
\int_{[H^{\eta}_{n-\ell}]}\CB^{\psi^{-1}_{\beta-1,y_{-\kappa}}}(\varphi_\pi)(xh)
\CJ_{\ell,\kappa}(\phi^{Z_{j,\kappa}})(\eps_{0,\beta}\eta xh)\ud x\ud h,
$$
where $[H^{\eta}_{n-\ell}]:=H^{\eta}_{n-\ell}(F)\bks H^{\eta}_{n-\ell}(\BA)$.

\begin{prop}[Case $(j>\ell)$]\label{prop:j>l}
Let $E(\phi_{\tau\otimes\sig},s)$ be the Eisenstein series on $G_n(\BA)$ as in
\eqref{es} and $\pi$ be an irreducible cuspidal automorphic representation of
$H_{n-\ell}(\BA)$. The global zeta integral
$\CZ(s,\phi_{\tau\otimes\sig},\varphi_{\pi},\psi_{\ell,w_{0}})$ as in
\eqref{bpes} is equal to
$$
\int_{R^{\eta}_{\ell,\beta-1}(\BA)\bks H_{n-\ell}(\BA)}
\int_{[H^{\eta}_{n-\ell}]}\CB^{\psi^{-1}_{\beta-1,y_{-\kappa}}}(\varphi_\pi)(xh)
\CJ_{\ell,\kappa}(\phi^{Z_{j,\kappa}})(\eps_{0,\beta}\eta xh)\ud x\ud h
$$
with $[H^{\eta}_{n-\ell}]:=H^{\eta}_{n-\ell}(F)\bks H^{\eta}_{n-\ell}(\BA)$.
\end{prop}

In order to show the integral expression for the global zeta integral
$\CZ(s,\phi_{\tau\otimes\sig},\varphi_{\pi},\psi_{\ell,w_{0}})$ as in Proposition \ref{prop:j>l},
it is enough to show that the inner integral
\begin{equation}\label{ii}
\int_{[H^{\eta}_{n-\ell}]}\CB^{\psi^{-1}_{\beta-1,y_{-\kappa}}}(\varphi_\pi)(xh)
\CJ_{\ell,\kappa}(\phi^{Z_{j,\kappa}})(\eps_{0,\beta}\eta xh)\ud x
\end{equation}
is an eulerian product. In fact, for a fixed $h$, as a function in $x$,
$\CJ_{\ell,\kappa}(\phi^{Z_{j,\kappa}})(\eps_{0,\beta}\eta xh)$ belongs to the
space of automorphic representation $\sig$ of $G_{n-j}(\BA)$. Hence, for a fixed $h$,
this above inner integral is essentially the standard Bessel period for the pair $(\pi,\sig)$
as defined in \eqref{bp}. By the local uniqueness of the Bessel models
(\cite{AGRS10}, \cite{SZ12}, \cite{JSZ11} and also \cite{GGP12}), integral \eqref{ii}
can be written as an eulerian product:
\begin{equation}\label{ep1}
\prod_{\nu}<\CB_{\nu}^{\psi^{-1}_{\beta-1,y_{-\kappa_{\nu}}}}(\varphi_{\pi,v})(h),
\CJ_{\ell,\kappa_{\nu}}(\phi_{\tau\otimes\sig,v}^{Z_{j,\kappa_{\nu}}})(\eps_{0,\beta}\eta h)>_{G_{n-j}}.
\end{equation}
Here the local pairing is a linear functional in the $\Hom$-space
$$
\Hom_{G_{n-j}(F_{\nu})}(\CB_{\nu}^{\psi^{-1}_{\beta-1,y_{-\kappa_{\nu}}}}(\pi_{\nu})\otimes\sigma_{\nu},\BC)
$$
with $\CB_{\nu}^{\psi^{-1}_{\beta-1,y_{-\kappa_{\nu}}}}(\pi_{\nu})$ is the local Bessel functional of $\pi_{\nu}$.
The local uniqueness of the Bessel models
(\cite{AGRS10}, \cite{SZ12}, \cite{JSZ11} and also \cite{GGP12}) asserts that the above $\Hom$-space
is at most one-dimensional. One can normalize the local pairing suitable at unramified
local places, so that the eulerian product makes sense.
Hence we obtain the following theorem.

\begin{thm}\label{thm:j>l}
Let $E(\phi_{\tau\otimes\sig},s)$ be the Eisenstein series on $G_n(\BA)$ as in
\eqref{es} and let $\pi$ be an irreducible cuspidal automorphic representation of
$H_{n-\ell}(\BA)$. Assume that the real part of $s$, $\Re(s)$, is large; and that $\pi$ and $\sig$ have a non-zero Bessel period, i.e.
$\CP^{\psi^{-1}_{\beta-1,y_{-\kappa}}}(\varphi_\pi,\varphi_\sig)$ is nonzero for a some choice of data. Then the global zeta integral
$\CZ(s,\phi_{\tau\otimes\sig},\varphi_{\pi},\psi_{\ell,w_{0}})$ is eulerian, i.e. is equal to
$$
\prod_v
\int_{h}
<\CB_v^{\psi^{-1}_{\beta-1,y_{-\kappa}}}(\varphi_{\pi,v})(h),
\CJ_{\ell,\kappa}(\phi_{\tau\otimes\sig,v}^{Z_{j,\kappa}})(\eps_{0,\beta}\eta h)>_{G_{n-j}}\ud h,
$$
where the integration is taken over $R^{\eta}_{\ell,\beta-1}(F_v)\bks H_{n-\ell}(F_v)$, and
the product is taken over all local places.
\end{thm}

The main local result of the paper is to calculate the unramified local integral explicitly
in terms of the local $L$-functions. For the purpose of our investigation of the global tensor product $L$-functions
$L(s,\pi\times\tau)$, it is enough
to consider the case when $j=\ell+1$. We define the local zeta integral
$\CZ_v(s,\phi_{\tau\otimes\sig},\varphi_{\pi},\psi_{\ell,w_{0}})$ to the local eulerian $v$-factor in the product in Theorem \ref{thm:j>l},
which is
\begin{equation}\label{lzi}
\int_{h}
<\CB_v^{\psi^{-1}_{\beta-1,y_{-\kappa}}}(\varphi_{\pi,v})(h),
\CJ_{\ell,\kappa}(\phi_{\tau\otimes\sig,v}^{Z_{j,\kappa}})(\eps_{0,\beta}\eta h)>_{G_{n-j}}\ud h,
\end{equation}
where the integration is taken over $R^{\eta}_{\ell,\beta-1}(F_v)\bks H_{n-\ell}(F_v)$.

\begin{thm}[$L$-function for case $j=\ell+1$]\label{urmL}
With all data being unramified, the local unramified zeta integral $\CZ_v(s,\phi_{\tau\otimes\sig},\varphi_{\pi},\psi_{\ell,w_{0}})$
is equal to the following product
\begin{eqnarray}
&&\prod_{i=1}^r\frac{L(s+\frac{1}{2},\tau_{i,v}\otimes\pi_v)}
{L(s+1,\tau_{i,v}\times\sig_v)L(2s_{i}+1,\tau_{i,v},Asai\otimes\xi^{m})}\nonumber\\
&&\times\prod_{1\leq i<j\leq r}\frac{1}{L(2s+1,\tau_{i,v}\times\tau_{j,v})}\apair{f_{\pi},f_{\sig}}_{G_{n-j}(F_{\nu})},
\end{eqnarray}
where $\apair{f_{\pi},f_{\sig}}_{G_{n-j}(F_{\nu})}$ is independent with $s$.
\end{thm}

This theorem will be proved in Section 4. It is also interesting to understand the local zeta integrals when $j>\ell+1$. We will come back
to this issue in our future consideration.

\subsection{Eulerian product: $j\leq\ell$ case}
In this section, we consider the case $j\leq\ell<\tilde{m}$. By Proposition~\ref{pro:GG},
we only need to consider the representative $\eps_{0,0}$ and $\eta_{\eps,I_{m-2\ell}}$,
where $\eps$ is defined in \eqref{m:eps}. For simplicity, we denote by
$\eta=\eta_{\eps,I_{m-2\ell}}$.

By \eqref{eq:0-R-l} and \eqref{eq:0-R-j}, we decompose $N^{\eta}_{\ell}$ as $Z_{j}N_{j,\ell-j}$,
where  $Z_{j}$ is identified as a subgroup of $G_n$, which is the maximal unipotent subgroup of
$\GL(V^{+}_{j})$, and
$$
N_{j,\ell-j}=\cpair{\begin{pmatrix}
I_{j}&&&&\\
&b_{\ell-j}&y_{4}&z_{4}&\\
&&I_{m-2\ell}&y'_{4}&\\
&&&b^{*}&\\
&&&&I_{j}
\end{pmatrix} \mid b\in Z_{\ell-j}}.
$$
Note that $N_{j,\ell-j}$ is the unipotent subgroup of $G_{n-j}$ as defined in \eqref{nell}
and the character $\psi_{\ell,\kappa}$ restricted on $N_{j,\ell-j}$ is the character
$\psi_{\ell-j,\kappa}$ of the subgroup $N_{\ell-j}$ (of $G_{n-j}$) as defined in \eqref{chw0},
which is denoted by $\psi_{n-j,\ell-j;\kappa}$.

\begin{align}
&\CZ(s,\phi_{\tau\otimes\sigma},\varphi_\pi,\psi_{\ell,w_{0}}) \label{eq:P-1}\\
=&\int_{[H_{n-\ell}]}\varphi(h)\int_{N^{\eta}_{\ell}(\BA)\bks N_{\ell}(\BA)}
\int_{[N^{\eta}_{\ell}]}
\lam\phi(\epsilon_{0,0}\eta unh)\psi^{-1}_{\ell,w_0}(un)\ud u\ud n\ud h. \nonumber
\end{align}
where $[H_{n-\ell}]:=H_{n-\ell}(F)\bks H_{n-\ell}(\BA)$ and $[N^{\eta}_{\ell}]:=N^{\eta}_{\ell}(F)\bks N^{\eta}_{\ell}(\BA)$.
The inner integral
\begin{equation}\label{ii2}
\int_{[N^{\eta}_{\ell}]}
\lam\phi(\epsilon_{0,0}\eta unh)\psi^{-1}_{\ell,w_0}(u)\ud u
\end{equation}
can be written as the following integral
$$
\int_{[N_{j,\ell-j}]}\int_{[Z_{j}]}
\lam\phi(\eps_{0,0}\eta c u n h)\psi_{\ell,\kappa}^{-1}(cu)\ud c \ud u.
$$
Since $\tau$ is generic, we have a nonzero Whittaker function
$$
\phi_{\lam}^{\psi_{Z_{j},\kappa}}(h)=\int_{[Z_{j}]}\lam\phi(\hat{z}g)\psi_{Z_{j},\kappa}(z)\ud z,
$$
where $\psi_{Z_{j},\kappa}$ is the restriction of $\psi_{\ell,\kappa}$ on $Z_{j}$. Hence the inner integral \eqref{ii2} can be written as
\begin{equation}\label{bp2}
\int_{[N^{\eta}_{\ell}]}
\lam\phi(\epsilon_{0,0}\eta unh)\psi^{-1}_{\ell,w_0}(u)\ud u
=
\CB^{\psi_{n-j,\ell-j,\kappa}^{-1}}(\phi_{\lam}^{\psi_{Z_{j},\kappa}})(\epsilon_{0,0}\eta nh),
\end{equation}
where $\CB^{\psi_{n-j,\ell-j,\kappa}^{-1}}$ is the Bessel period on the group $G_{n-j}(\BA)$ with respect to the subgroup $N_{j,\ell-j}$
and the character $\psi_{n-j,\ell-j,\kappa}$.

Therefore, the global zeta integral has the expression:
\begin{align*}
&\CZ(s,\phi_{\tau\otimes\sigma},\varphi_\pi,\psi_{\ell,w_{0}})\\
=&\int_{[H_{n-\ell}]}\varphi(h)\int_{N^{\eta}_{\ell}(\BA)\bks N_{\ell}(\BA)}
\CB^{\psi_{n-j,\ell-j,\kappa}^{-1}}(\phi_{\lam}^{\psi_{Z_{j},\kappa}})(\epsilon_{0,0}\eta nh)\psi^{-1}_{\ell,w_0}(n)\ud n\ud h
\end{align*}

\begin{prop}[Case $(j\leq \ell)$]\label{pro:ell<j}
Let $E(\phi_{\tau\otimes\sig},s)$ be the Eisenstein series on $G_n(\BA)$ as in
\eqref{es} and $\pi$ be an irreducible cuspidal automorphic representation of
$H_{n-\ell}(\BA)$. The global zeta integral
$\CZ(s,\phi_{\tau\otimes\sig},\varphi_{\pi},\psi_{\ell,w_{0}})$ as in
\eqref{bpes} is equal to
$$
\int_{[H_{n-\ell}]}\varphi(h)\int_{N^{\eta}_{\ell}(\BA)\bks N_{\ell}(\BA)}
\CB^{\psi_{n-j,\ell-j,\kappa}^{-1}}(\phi_{\lam}^{\psi_{Z_{j},\kappa}})(\epsilon_{0,0}\eta nh)\psi^{-1}_{\ell,w_0}(n)\ud n\ud h.
$$
\end{prop}

It remains to show that the global zeta integral in Proposition~\ref{pro:ell<j} is eulerian. To this end, we need to
to switch the order of the integrations $\int_{h}$ and $\int_{n}$ in
$$
\int_{[H_{n-\ell}]}\varphi(h)\int_{N^{\eta}_{\ell}(\BA)\bks N_{\ell}(\BA)}
\CB^{\psi_{n-j,\ell-j,\kappa}^{-1}}(\phi_{\lam}^{\psi_{Z_{j},\kappa}})(\epsilon_{0,0}\eta nh)\psi^{-1}_{\ell,w_0}(n)\ud n\ud h.
$$
This can be deduced from the following lemma.

\begin{lem}
The automorphic function
$$
\Psi(h)=\int_{N^{\eta}_{\ell}(\BA)\bks N_{\ell}(\BA)}
\CB^{\psi_{n-j,\ell-j,\kappa}^{-1}}(\phi_{\lam}^{\psi_{Z_{j},\kappa}})(\epsilon_{0,0}\eta nh)\psi^{-1}_{\ell,w_0}(n)\ud n
$$
is uniformly moderate growth on $H_{n-\ell}(\BA)$.
\end{lem}
\begin{proof}
The proof is similar to the orthogonal case in Appendix 2 to \S 5  \cite{GPSR97}.
\end{proof}
Since $\varphi_{\pi}$ is rapidly decay, the global zeta integral is equal to
\begin{equation}\label{eq:j<ell-unfactorize}
\int_{N^{\eta}_{\ell}(\BA)\bks N_{\ell}(\BA)}
\int_{[H_{n-\ell}]}\varphi(h)\CB^{\psi_{n-j,\ell-j,\kappa}^{-1}}(\phi_{\lam}^{\psi_{Z_{j},\kappa}})(\epsilon_{0,0}\eta nh)\psi^{-1}_{\ell,w_0}(n)\ud h\ud n.
\end{equation}
Then the inner integral is an $H_{n-\ell}(\BA)$ invariant pairing between $\pi$ and $\CB^{\psi_{n-j,\ell-j,\kappa}^{-1}}(\sig)$. The local pairing is a linear functional in the $\Hom$-space
$$
\Hom_{G_{n-j}(F_{\nu})}(\pi_{\nu}\otimes\CB_\nu^{\psi^{-1}_{\beta-1,y_{-\kappa_\nu}}}(\sigma_\nu),\BC)
$$
with $\CB_\nu^{\psi^{-1}_{\beta-1,y_{-\kappa_\nu}}}(\sig_\nu)$ is the local Bessel functional of $\sig_\nu$.
By the local uniqueness of the Bessel models and a suitable normalization at unramified local places, we can factorize \eqref{eq:j<ell-unfactorize} as follows
$$
\prod_{\nu}\apair{\varphi_{\nu}*h_{\nu},\CB_{\nu}^{\psi_{n-j,\ell-j,\kappa}^{-1}}(\phi_{\lam}^{\psi_{Z_{j},\kappa}})(\epsilon_{0,0}\eta n_{\nu}h_{\nu})}_{\nu}.
$$

Note that in this case, $N^{\eta}_{\ell}\bks N_{\ell}$ consists of elements of form
$$
\begin{pmatrix}
I_{j}&x_{1}&x_{2}&x_{3}&x_{4}\\
&I_{\ell-j}&&&x'_{3}\\
&&I_{m-2\ell}&&x'_{2}\\
&&&I_{\ell-j}&x'_{1}\\
&&&&I_{j}\\
\end{pmatrix}.
$$
The restriction of $\psi_{\ell,\kappa}$ on $N^{\eta}_{\ell}\bks N_{\ell}$ is $\psi((x_{1})_{j,1})$. Under the adjoint action of $\eps_{0,0}\eta$, the integral domain $N^{\eta}_{\ell}\bks N_{\ell}$ is also denoted by $U^{-}_{j,\eta}$, which is a subgroup of the opposite $U^{-}_{j}$, consisting of elements of form
$$
\begin{pmatrix}
I_{j}&&&&\\
x'_{3}&I_{\ell-j}&&&\\
x'_{2}&&I_{m-2\ell}&&\\
x'_{1}&&&I_{\ell-j}&\\
x_{4}&x_{1}&x_{2}&x_{3}&I_{j}\\
\end{pmatrix}
$$
Therefore, the induced character on $U^{-1}_{j,\ell}$ is $\psi^{-1}(n_{m-j,1})$.

We remark that in the family of global zeta integrals, we only use the case when $j=\ell+1$ to calculate the local unramified
zeta integrals to obtain the local $L$-functions as we needed. In all other cases, the global zeta integrals are eulerian. The
unramified calculation will be taken up in our future consideration, and the potential applications to our explicit constructions
of endoscopy correspondences as discussed in \cite{J12} remain to be fully discussed.

\section{Unramified calculation and local $L$-functions}

In this section, we will calculate the local zeta integral for the case $j=\ell+1$ as defined in Theorem~\ref{urmL} over the unramified places.
The quasi-split orthogonal group cases were done in \cite{GPSR97}. In the following, we extend the idea and the method in \cite{GPSR97}
to the quasi-split unitary group cases. It turns out that the argument in this case is much more technically involved, due to the splitting of the
unramified local place of the number field $F$ to the quadratic extension $E$.

To achieve the goal of this section, we reformulate the local zeta integrals through the paring of Bessel models in Subsection 4.1, including
some general statements on the twisted Jacquet modules, which we recall from \cite[Chapter 5]{GRS11}. In Subsection 4.2, we discuss
unramified representations considered in the local zeta integrals and their Satake parameters, with which, we define unramified local
$L$-functions we need. In Subsection 4.3, we specify the local zeta integrals for unramified data by considering the cases when the
unramified local place $\nu$ of $F$ is split or not over $E$. By the Bernstein rationality, the unramified local zeta integrals are
expressed as a rational function with respect to the parameters coming from the relevant representations. This rational function
is explicitly calculated in Subsections 4.4 and 4.5 and identified with the expected local $L$-functions. Hence
we carry out the complete proof of Theorem~\ref{urmL}.

Throughout this section, denote by $\nu$ the local place of $F$. If $\nu$ is inert, then $E_{\nu}$ is the unramified quadratic extension of $F_{\nu}$.
If $\nu$ splits over $E$, then $E_{\nu}\cong F_{\nu}\times F_{\nu}$. For the simplicity of notation,
most of the time, we will omit the subscript $\nu$ from the corresponding notation. For instance, we may use $F$ for the local field $F_{\nu}$ and
use $\pi$ for $\pi_{\nu}$ and so on, when there are no confusions.

Let $\Fo$ be the ring of integers of $F$, and fix a prime element $\varpi$ of $\Fo$.
Let $q_{F}$ and $q_{E}$ be the cardinality of the residue fields of $F$ and $E$, respectively. If $\nu$ is inert,  one has that $q_{E}=q_{F}^{2}$,
and if $\nu$ is split, one has that $q_E=q_F$.
We fix the normalized absolute values $|x|_{F}=|x|_{\nu}$ for $x\in F$, and $|x|_{E}=|x\bar{x}|_{F}$ for $x\in E$ if $\nu$ is inert.

When $\nu$ splits in $E$, we need to write down more explicit structures of the unitary group $G_n(F)$, which are needed for the unramified
calculation of the local integrals. In this case, one may take that $\rho=d^{2}$ or $\sqrt{\rho}=d$ for some $d\in F^\times$, and hence has that
$E\cong F\times F$. This isomorphism is explicitly given by the following mapping: for $x,y\in F$,
$$
x\otimes 1+y\otimes\sqrt{\rho}\mapsto (x+yd,x-yd).
$$
When $x\in E$ is taken to $(x_{1},x_{2})\in F\times F$, the corresponding absolute values  are normalized so that $|x|_{E}=|x_{1}x_{2}|_{F}$.
It follows that $\GL_m(E)\cong\GL_m(F)\times\GL_m(F)$ given by
$$
g_{1}\otimes 1+g_{2}\otimes \sqrt{\rho}\mapsto (g_{1}+dg_{2},g_{1}-dg_{2}).
$$
Then the unitary group $G_n(F)$ consists of all elements $g=g_{1}\otimes 1+g_{2}\otimes \sqrt{\rho} \in \GL_m(E)$ satisfying
$$
(g_{1}+dg_{2})J_{m} {}^{t}(g_{1}-dg_{2})=J_{m}.
$$
The restriction of the above isomorphism to $G_n(F)$ gives the isomorphism: $G_{n}(F)\cong \GL_{m}(F)$, given explicitly by
\begin{equation}\label{eq:iso-split}
g_{1}\otimes 1+g_{2}\otimes \sqrt{\rho}\mapsto (g_{1}+dg_{2},g_{1}-dg_{2})\mapsto g_{1}+dg_{2}.
\end{equation}

Next, we explain the data in the unramified local integral as needed for Theorem \ref{urmL}.
We take a normalized parabolically induced representation
$$
\Pi(\tau,\sig,s)=\Ind^{G_n(F)}_{P_{j}(F)}(|\det|_E^{s}\tau\otimes\sig),
$$
where $\tau$ and $\sig$ are irreducible admissible representations of $\GL_{j}(E)$ and $G_{n-j}(F)$, respectively.
Let $\pi$ be an irreducible admissible representation of $H_{n-\ell}(F)$. Recall that the unitary group $H_{n-\ell}$ is defined
in \eqref{Lellw0}.

When $\nu$ splits in $E$, the induced representation $\Pi(\tau,\sig,s)$ can be made more specific. In this case,
the representation $\tau$ can be expressed as $\tau_{1}\otimes\tau_{2}$, where $\tau_{i}$ are irreducible representations of $\GL_{j}(F)$. The representation $\sig$ is an irreducible representation of $\GL_{m-2j}(F)$. The representation $\Pi(\tau,\sig,s)$
 can be realized as the representation of $\GL_{m}(F)$, induced from the standard parabolic subgroup $P_{j,m-2j,j}(F)$ with the
following representation
$$
\begin{pmatrix}
g_{1}&x&y\\
&h&z\\
&&g_{2}
\end{pmatrix}\mapsto
\left\vert \frac{\det(g_{1})}{\det(g_{2})}\right\vert^{s}\tau_{1}(g_{1})\otimes\sig(h)\otimes\tau_{2}(g^{*}_{2}),
$$
where $g_{1}, g_{2}\in \GL_{j}(F)$ and $g^{*}_{2}=J_{j}{}^{t}g^{-1}J^{-1}_{j}$.

\subsection{Local zeta integrals and twisted Jacquet modules}
We introduce a local zeta integral in general at any local place, although only a special case will contribute to the proof of Theorem \ref{urmL}.

Let $W_{j}$ be a Whittaker model attached to a nonzero member in the space
$$
\Hom_{\GL_{j}(F)}(\tau,\Ind^{\GL_{j}(E)}_{Z_{j}(E)}(\psi_{Z_{j},\kappa})).
$$
This produces a partial Whittaker function
$$
W_{j}(f)\in \Ind ^{G_{n}(F)}_{Z_{j}(E)\times G_{n-j}(F)\ltimes U_{j}(F)}(\psi_{Z_{j},\kappa}\otimes\sig)
$$
for $f\in\Pi(\tau,\sig,s)$. As suggested by the global calculation in Section~\ref{sec:eulerian}, we can formally define
(over the open cell) the following function
$$
\CJ(f)(g):=\int_{N^{\eta}_{\ell}(F)\bks U_{\ell}(F)}W_{j}(f)(\eps_{0,j-\ell}\eta u g)\psi^{-1}_{(\ell,\kappa)}(u)\ud u.
$$
Following the exact argument in Appendix 2 to \S 5 \cite{GPSR97}, the integral defining $\CJ(f)$ is convergent for $\Re(s)$ sufficient large
and is analytic in $s$. In addition, $\CJ(f)$ is a function on $H_{n-\ell}(F)$ belonging to the space
$$
\Ind^{H_{n-\ell}(F)}_{R^{\eta}_{\ell,\beta-1}(F)}(\psi^{-1}_{\beta-1,y_{-\kappa}}\otimes\sig^{w^{\ell}_{q}}),
$$
where $\sig^{w^{\ell}_{q}}$ is a representation of $G_{n-j}(F)$ conjugated by $w^{\ell}_{q}$.

Let $\CB_{\beta-1}$ be a Bessel model attached to a non-trivial member in the Hom-space
$$
\Hom_{H_{n-\ell}(F)}(\pi,\Ind^{H_{n-\ell}(F)}_{R^{\eta}_{\ell,\beta-1}(F)}(\psi_{\beta-1,y_{-\kappa}}\otimes \tilde{\sig}^{w^{\ell}_{q}})),
$$
where $\tilde{\sig}$ is the dual of $\sig$.
Let $\apair{\cdot,\cdot}_{\sig}$ be an invariant pairing of $\sig$ and $\tilde{\sig}$. By the uniqueness of the local Bessel models
(\cite{AGRS10}, \cite{GGP12}, \cite{SZ12} and \cite{JSZ11}),  $\CB_{\beta-1}$ is unique up to a constant.
We may define a pairing
$$
\apair{\CJ(f),\CB_{\beta-1}(v)}=\int_{R^{\eta}_{\ell,\beta-1},y_{-\kappa}(F)\bks H_{n-\ell}(F)}\apair{\CJ(f)(h),\CB_{\beta-1}(v)(h)}_{\sig}\ud h.
$$
\begin{lem}
The pairing $\apair{\CJ(f),\CB_{\beta-1}(v)}$  is absolutely convergent for suitable data $\tau$ and $\sig$, and $\Re(s)$ sufficiently large.
\end{lem}
\begin{proof}
The proof is similar to Theorem A of Appendix (I) to \S 5 in \cite{GPSR97}.
\end{proof}

It is easy to check that this pairing, if exists, defines a linear functional of the Gross-Prasad type in the Hom-space
\begin{equation} \label{eq:hom-local}
\Hom_{N_{\ell}\times H^{\triangle}_{n-\ell}}(\Pi(\tau,\sig,s)\otimes\pi,\psi_{\ell,\kappa}).
\end{equation}
Again, by the uniqueness of local Bessel functionals, the dimension of this Hom-space is at most one, when $\Pi(\tau,\sig,s)$ is irreducible.
Therefore, the local zeta integral is defined by
\begin{equation}
\CZ(s,f,v,\psi_{\ell,\kappa})=\apair{\CJ(f),\CB_{\beta-1}(v)},
\end{equation}
for $f\in \Pi(\tau,\sig,s)$ and $v\in \pi$, which is proportional to the local zeta integral defined as an eulerian factor of the
the global zeta integral as in Section 3. For the unramified data, we may normalize the pairing, so that this proportional constant is one.

In order to proceed the explicit calculation of the local integrals, we have to understand those Bessel models involved in the
local zeta integrals from the representation-theoretic point of view. This means to see more precisely the structures of those
twisted Jacquet models. We recall relevant results from \cite[Chapter 5]{GRS11}.

Let $(\Pi, V_{\Pi})$ be a smooth representation of $G_{n}(F)$.
Let $J_{\psi_{\ell,\kappa}}(\Pi)$ be the twisted Jacquet module of $\Pi$ with respect to $N_{\ell}(F)$ and its character $\psi_{\ell,\kappa}$,
the space of which is defined by
\begin{equation}\label{eq:twist-Jacquet}
V_{\Pi}/\Span\cpair{\Pi(n)v-\psi_{\ell,\kappa}(n)v\mid n\in N_{\ell}(F), v\in V_{\Pi}}.
\end{equation}
Note that $J_{\psi_{\ell,\kappa}}(\Pi)$ is a smooth representation of $H_{n-\ell}(F)$. The same definition may apply to
the twisted Jacquet modules for different groups throughout the section.

Next, we study the twisted Jacquet module $J_{\psi_{\ell,\kappa}}(\Pi)$ for the induced representation $\Pi=\Pi(\tau,\sig,s)$. To do so,
we consider the structure of the restriction of the induced representation $\Pi$ to the standard parabolic subgroup $P_\ell$, which
is denoted by $\Res_{P_{\ell}}(\Pi)$. This reduces to consider the generalized Bruhat decomposition $P_j\bks G_n/P_\ell$, which was discussed
in Section 3. Hence, as a representation of $P_\ell$, $\Res_{P_{\ell}}(\Pi)$ can be expressed (up to semi-simplification) as a finite direct sum
$\Pi_{\eps_{\alpha,\beta}}$ parameterized by the set of representatives $\{\eps_{\alpha,\beta}\}$ as discussed in Section 3.1.

Let $\tau^{(t)}$ denote the $t$-th Bernstein-Zelevinski derivative of $\tau$ along the subgroup $Z'_{t}$ defined in \eqref{eq:Z'} with the character
$$
\psi'_{t}\begin{pmatrix} I_{\beta}&y\\0&z\end{pmatrix}=\psi^{-1}(z_{1,2}+z_{2,3}+\cdots+z_{t-1,t}).
$$
We embed $\GL_\beta$ into $\GL_j$ through the map $g\in \GL_{\beta}\mapsto \diag(g,I_{t})\in \GL_{j}$. The image, which is still denoted by
$\GL_{\beta}$, normalizes the character $\psi'_{t}$. Hence $\tau^{(t)}$ is the representation of $\GL_{\beta}$ via
the twisted Jacquet module $J_{\psi'_{t}}(\tau)$.
We also define the following character of $Z'_{t}$,
$$
\psi''_{t}\begin{pmatrix} I_{\beta}&y\\0&z\end{pmatrix}=\psi^{-1}(z_{1,2}+z_{2,3}+\cdots+z_{t-1,t}+y_{\beta,1}),
$$
which is conjugate to the character $\psi_{Z'_{\ell},\kappa}$ as defined in \eqref{chz'} for any nonzero $\kappa$, by an element in the subgroup $\GL_{\beta}$.
Denote the corresponding Jacquet module $J_{\psi''_{t}}(\tau)$ by $\tau_{(t)}$, which is a representation of the mirabolic subgroup  of $\GL_{\beta}$.

Recall that $P'_{\beta}=H^{\eta_{\epsilon,I_{m-2\ell}}}_{n-\ell}$ is as defined in \eqref{hgamma}. By the discussion in Page~\pageref{page:hgamma}, when $\ell+\beta<\tilde{m}$, $P'_{\beta}$ is a maximal parabolic subgroup of $H_{n-\ell}$. For the proof of Theorem \ref{urmL}, which
only concerns the case of $j=\ell+1$, we may assume that $\ell<j$ in the following discussion.
Put $P''_{j-\ell}=H^{\eta_{\epsilon,\gamma}}_{n-\ell}$ for $\gamma$ as defined in \eqref{m:gamma}.
Note that $P'_{w}\gamma H_{n-\ell}$ is  the open  double coset discussed in Page~\pageref{m:gamma}, and
$P''_{j-\ell}$ is not a proper maximal parabolic subgroup. Although we only need in this paper the case when $\ell<j$, we recall from
\cite{GRS11} the following general result.

\begin{prop}[{\cite[Theorem 5.1]{GRS11}}]\label{pro:Jac-mod}
Assume that $0\leq \ell<\tilde{m}$ and $1\leq j<m$. If $\nu$ is inert, then, up to semi-simplification, the following isomorphism holds
$$
J_{\psi_{\ell,\kappa}}(\Ind^{G_n}_{P_j}(\tau \otimes \sig))
\equiv
\Upsilon_1\oplus\Upsilon_2\oplus\Upsilon_3
$$
where
$$
\Upsilon_1
=
\oplus_{\substack{j-\ell \leq \beta< \tilde{m}-\ell \\ 0\leq \beta\leq j}}
\ind^{H_{n-\ell}}_{P'_{\beta}} (|\det| _{E}^{\frac{1-t}{2}} \tau^{(t)} \otimes J_{\psi'_{\ell-t,\kappa}}(\sigma^{w^t_q})),
$$
$$
\Upsilon_2=\begin{cases}
\ind^{H_{n-\ell}}_{P''_{j-\ell}} (|\det|^{-\frac{\ell}{2}}_E \tau_{(\ell)} \otimes \sigma^{w^{\ell}_q}), & \ell<j, \\
0, & \ell\not<j,
\end{cases}
$$
and
$\Upsilon_3$ is the representation of $H_{n-\ell}$ supported on the other double cosets.
\end{prop}

We note that the detail of $\Upsilon_3$ is not needed in the explicit unramified calculation and is referred to
\cite[Theorem 5.1]{GRS11}.

If $\nu$ is split,  let $\udl{\ell}=[\ell_{1},\ell_{2},\ell_{3}]$ be a partition of a positive integer $N$ and
consider the twisted Jacquet module $J_{\tilde{\psi}}(\Ind^{\GL_{N}}_{P_{j},N-j}\tau_{1}\times\tau_{2})$
in ~\cite[Section 3.6]{GRS11}.
In order to simplify our calculation, up to a suitable conjugation,
we will use the Gelfand-Grave character defined in ~\cite[Section 3.6]{GRS11}.
Let $N_{\ell}$ consist of elements of form
$$
n=\begin{pmatrix}
z^{(1)}&y^{(1)}&x\\&I_{m-2\ell}&y^{(2)}\\&&z^{(2)}
\end{pmatrix}\in \GL_{m}(F),
$$
where $z^{(1)},z^{(2)}\in Z_{\ell}(F)$. We will take the character $\psi_{\ell,\kappa}$ to be the following character
$$
\tilde{\psi}(n)=\psi(\sum^{\ell-1}_{i=1}(z^{(1)}_{i,i+1}+z^{(2)}_{i,i+1})+y^{(1)}_{\ell,1}+y^{(2)}_{1,1}).
$$
The stabilizer of the character $\tilde{\psi}(n)$ inside $G_{n-\ell}(E_{\nu})\cong \GL_{m-2\ell}(F)$ is
$$
\tilde{L}_{\ell}=\cpair{\diag\cpair{I_{\ell},\gamma,I_{\ell}}\in \GL_{m}(F)\mid \gamma=\begin{pmatrix}1&\\&g \end{pmatrix},~g\in \GL_{m-2\ell-1}(F)}.
$$

Define
$$
\tau^{[\ell_{1}-\alpha]}_{2}:= [(\tau_{2}^{*})^{\ell_{1}-\alpha}]^{*}
\text{ and }
(\tau_{2})_{[\ell_{1}]}:=[(\tau^{*}_{2})_{(\ell_{1})}]^{*}.
$$
which are representations of $\GL_{\ell_{2}-\beta+\alpha}(F)$ and the mirabolic subgroup of $\GL_{N-j-\ell_{1}}(F)$, respectively,
where the inner $*$ denotes composition with the map
$$
g\to J'_{\ell_{1}-\alpha} {}^{t}g^{-1}{J'}_{\ell_{1}-\alpha}^{-1},
$$
where  $J'_{\ell_{1}-\alpha}=\diag(J_{\ell_{1}-\alpha},J_{\ell_{2}-\beta+\alpha})$, and the outer $*$ denotes composition with the map
$$
g\to J_{\ell_{2}-\beta+\alpha}\cdot {}^{t}g^{-1}J^{-1}_{\ell_{2}-\beta+\alpha}.
$$
More information about $\tau^{[\ell_{1}-\alpha]}_{2}$ and $(\tau_{2})_{[\ell_{1}]}$ can be found in \cite[Pages 113 and 115]{GRS11}.
%

\begin{prop}[{\cite[Theorem 5.7]{GRS11}}] \label{pro:Jac-mod-split}
Up to semi-simplification, the following isomorphism holds
$$
J_{\tilde{\psi}}(\Ind^{\GL_{N}(F)}_{P_{j,N-j}}\tau_{1}\times\tau_{2})
\equiv
\CL_1\oplus\CL_2\oplus\CL_3\oplus\CL_4\oplus\CL_5
$$
where $\CL_1$ is given by the following direct sum
$$
\oplus_{\substack{j-\ell_{3}<\beta<\ell_{2}\\ 0\leq \beta\leq j}}
\Ind^{\GL_{\ell_{2}-1}}_{P_{\beta,\ell_{2}-\beta-1}}(|\cdot|^{\frac{1-(j-\beta)+\ell_{3}-\ell_{1}}{2}}\tau^{(j-\beta)}_{1})\otimes |\cdot|^{\frac{j-\beta}{2}}J_{\psi_{(\ell_{1},\ell_{2}-\beta,\ell_{3}-j+\beta)}}(\tau_{2});
$$
$\CL_2$ is given by the following direct sum
$$
\oplus_{\substack{0<r<j-\ell_{3}\\j-\ell_{2}-\ell_{3}\leq r\leq \ell_{1}}}
\Ind^{\GL_{\ell_{2}-1}}_{P_{j_{\ell_{3}-r-1},\ell_{2}+\ell_{3}-j+r}}
(|\cdot|^{-\frac{\ell_{1}-r}{2}}J_{\psi_{(r,j-\ell_{3}-r,\ell_{3})}}(\tau_{1}))
\otimes |\cdot|^{\frac{\ell_{3}-r-1}{2}}\tau_{2}^{[\ell_{1}-r]};
$$
$\CL_3$ is given by the following representation
$$
\begin{cases}
\Ind^{\GL_{\ell_{2}-1}}_{P_{j-\ell_{3}-1,\ell_{2}+\ell_{3}-j}}(|\cdot|^{-\frac{\ell_{1}}{2}}(\tau_{1})_{(\ell_{3})})
\otimes|\cdot|^{\frac{\ell_{3}-1}{2}}\tau_{2}^{[\ell_{1}]}, & \text{ if } 0<j-\ell_{3}\leq \ell_{2},\\
\tau^{(j)}_{1}\otimes |\det|^{-\frac{\ell_{3}}{2}}\tau_{2[\ell_{1}]}, & \text{ if } \ell_{3}=j,\\
0, &\text{ otherwise};
\end{cases}
$$
$\CL_4$ is given by the following representation
$$
\begin{cases}
\Ind^{\GL_{\ell_{2}-1}}_{P_{j-\ell_{3},\ell_{2}+\ell_{3}-j-1}}(|\cdot|^{\frac{1-\ell_{1}}{2}}(\tau_{1})^{(\ell_{3})})
\otimes|\cdot|^{\frac{\ell_{3}}{2}}\tau_{2[\ell_{1}]}, & \text{ if } 0<j-\ell_{3}<\ell_{2},\\
0, &\text{ otherwise};
\end{cases}
$$
and $\CL_5$ is given by the following representation
$$
\begin{cases}
ind^{\GL_{\ell_{2}-1}}_{P'_{j-\ell_{3}-1,1,\ell_{2}+\ell_{3}-j-1}}(|\cdot|^{-\frac{\ell_{1}}{2}}(\tau_{1})_{(\ell_{3})})\otimes |\cdot|^{\frac{\ell_{3}}{2}}\tau_{2[\ell_{1}]}, &\text{ if } 0<j-\ell_{3}<\ell_{2}\\
0, &\text{ otherwise.}
\end{cases}
$$
\end{prop}
The other notation in this proposition is referred to \cite[Section 5.2]{GRS11}.
We are going to apply the case of $\udl{\ell}=[\ell,m-2j,\ell]$ to the unramified calculation.

\subsection{Unramified representations and local $L$-functions of unitary groups}
Let $B_{H}=T_{H}N_{H}$ be a Borel subgroup of $H_{n-\ell}$ with the maximal $F$-torus $T_{H}$ and the unipotent radical $N_{H}$ .
 Let $K_{G}=G_{n}(\Fo_{F})$ (resp.~$K_{H}=H_{n-\ell}(\Fo_{F})$) be the standard maximal open compact
subgroup of $G_{n}$ (resp.~$H_{n-\ell}$). Denote by $W(G_{n})=N(T)/T$ the Weyl group of $G_{n}$.
When $\nu$ is inert over $E$, $W(G_{n})$ is the Weyl group associated to a root system of type $B$.
When $\nu$ is split over $E$, $W(G_{n})$ is the Weyl group associated to a root system of type $A$.

From now on, we assume that the representations $\tau$, $\sigma$, and $\pi$ are unramified.
Let $\chi_{\tau}$ and $\chi_{\sig}$ be the unramified characters corresponding to the spherical representations $\tau$ and $\sig$. Then
$\chi_{\tau}=\otimes^{j}_{i=1}\chi_{i}$ and $\chi_{\sig}=\otimes^{\tilde{m}}_{i=j+1}\chi_{i}$. Define
$\chi_s:=|\cdot|^{s}\chi_{\tau}\otimes\chi_{\sig}$.
Let $\Pi_s:=\Pi(\chi_s)$ and $\pi:=\pi(\mu)$ be the unramified constituents of the normalized induced representations
$$
\Ind^{G_{n}(F)}_{P_{j}(F)}(|\det|^{s}\tau\otimes\sig)\ \  \text{and}\ \ \ \Ind^{H_{n-\ell}(F)}_{B_{H}(F)}(\mu),
$$
respectively.

If $\nu$ is inert over $E$, $\chi_{i}$ and $\mu_{i}$ are unramified characters of $E^\times=F(\sqrt{\rho})^\times$.

If $\nu$ is split over $E$, $H_{n-\ell}(F)\cong \GL_{m-2\ell-1}(F)$ and $\mu_{i}$ splits into a product $\theta_{i}\vartheta_{i}$
of two unramified characters of $F^\times$. Moreover, if $m-2\ell-1$ is odd, $\mu$ splits as
$\otimes^{(m-2\ell-2)/2}_{i=1}\theta_{i}\otimes\vartheta_{i}\otimes\mu_{0}$. Here $\mu_{0}$ is also an unramified character of $F^{\times}$.
In particular, $\pi(\mu)$ is the unramified constituent of the following induced representation
$$
\Ind^{H_{n-\ell}(F)}_{B_{H}}((\otimes^{\widetilde{m}_{H}}_{i=1}\theta_{i})\otimes
(\otimes^{\widetilde{m}_{H}}_{i=1}\vartheta^{-1}_{\widetilde{m}_{H}+1-i}))
$$
if $m$ is odd, and of the following induced representation
$$
\Ind^{H_{n-\ell}(F)}_{B_{H}}((\otimes^{\widetilde{m}_{H}}_{i=1}\theta_{i})\otimes\mu_{0}\otimes
(\otimes^{\widetilde{m}_{H}}_{i=1}\vartheta^{-1}_{\widetilde{m}_{H}+1-i}))
$$
if $m$ is even, where $\tilde{m}_{H}$ is the Witt index of the hermitian vector subspace
$(W_{\ell}\cap w^{\perp}_{0}, q_{W_{\ell}\cap w^{\perp}_{0}})$, which defines $H_{n-\ell}$. Since $E\cong F\times F$, we must have
$$
\GL_{j}(E)\cong \GL_{j}(F)\times\GL_{j}(F)
$$
and $\chi_{\tau}$ splits as a product $\Xi_{\tau}\Theta_{\tau}$ of unramified characters with
\begin{eqnarray*}
\Xi_{\tau}&=&\otimes^{j}_{i=1}\Xi_{i},\\
\Theta_{\tau}&=&\otimes^{j}_{i=1}\Theta_{i}.
\end{eqnarray*}
The representation $\tau$ is the unramified constituent of the induced representation
$$
\Ind^{\GL_{j}(F)}_{B_{\GL_{j}}(F)}(|\det|^{s}\Theta_{\tau})\otimes\Ind^{\GL_{j}(F)}_{B_{\GL_{j}}(F)}(|\det|^{-s}\Xi_{\tau}^{-1}).
$$
Also, if we set $\chi_{i}=|\cdot|^{s}\Theta_{i}$ and $\chi_{m+1-i}=|\cdot|^{-s}\Xi^{-1}_{i}$ for $1\leq i\leq j$, then the
representation $\Pi(\chi_s)$ of $G_n(F)$ becomes the corresponding representation of $\GL_{m}(F)$.

In the following, we write down the Satake parameters for the unramified representations discussed above and write the relevant unramified
local $L$-functions, following the arguments in \cite{BS09} or \cite{KK11} for instance.

The Langlands dual group ${}^{L}\RU_{m}$ of $\RU_{m}$  is $\GL_{m}(\BC)\rtimes \Gamma(E/F)$,
where $\Gamma(E/F)$ is the Galois group on $E$ and the nontrivial element $\iota$ acts on $\GL_{m}(\BC)$ via $\iota(g)=J_{m}{}^{t}g^{-1}J^{-1}_{m}$.
A $2m$-dimensional complex representation $\rho_{2m}$ of the Langlands dual group ${}^{L}\RU_{m}$ is given by
$$
(g;1)\mapsto \begin{pmatrix} g&0\\0&g^{*} \end{pmatrix}
\text{ and }
(I_{m};\iota) \mapsto \begin{pmatrix} 0&I_{m}\\I_{m}&0 \end{pmatrix},
$$
for any $g\in\GL_{m}(\BC)$.
The Langlands dual group ${}^{L}\Res_{E/F}\GL_{j}$ of $\Res_{E/F}\GL_{j}$ is $(\GL_{j}(\BC)\times\GL_{j}(\BC))\rtimes \Gamma(E/F)$. The element $\iota$ acts on $\GL_{j}(\BC)\times\GL_{j}(\BC)$ by $\iota(g_{1},g_{2})=(g_{2},g_{1})$.
Considering a $j^{2}$ dimensional representation of ${}^{L}\Res_{E/F}\GL_{j}$, which is realized in the space of all $j\times j$ matrices,
$M_{j\times j}$, by
\begin{eqnarray*}
(g_{1},g_{2};1)(x)&\mapsto& g_{1}\cdot x\cdot{^{t}g_{2}},\\
(I_{j},I_{j};\iota)(x)&\mapsto& {}^{t}x,
\end{eqnarray*}
and is called the $Asai$ representation of ${}^{L}\Res_{E/F}\GL_{j}$.

In addition, the Langlands dual group ${}^{L}(\RU_{m}\times\Res_{E/F}\GL_{j})$ is $$
(\GL_{m}(\BC)\times\GL_{j}(\BC)\times\GL_{j}(\BC))\rtimes\Gamma(E/F).
$$
The element $\iota$ acts on it by $\iota(g,g_{1},g_{2})=(g^{*},g_{2},g_{1})$.
A $2mj$-dimensional complex representation $\rho_{2mj}$ of ${}^{L}(\RU_{m}\times\Res_{E/F}\GL_{j})$ is given by
\begin{eqnarray*}
(g,g_{1},g_{2},1)&\mapsto& \begin{pmatrix}
g\otimes g_{1}&0\\0& g^{*}\otimes g_{2}
\end{pmatrix},\\
(I_{m},I_{j},I_{j},\iota)&\mapsto& \begin{pmatrix}
0&I_{mj}\\ I_{mj}&0
\end{pmatrix},
\end{eqnarray*}
where $g\otimes g_{i}$ is the Kronecker product.

We first consider the irreducible unramified representation $\pi(\mu)$ of $H_{n-\ell}(F)$.
When $\nu$ is inert over $E$, the Satake parameter of $\pi(\mu)$ is the semi-simple conjugacy class in ${}^{L}H_{n-\ell}$ of type
$$
c(\pi(\mu))=(\diag(\mu_{1}(\varpi_E),\mu_{2}(\varpi_E),\dots,\mu_{\widetilde{m}_{H}}(\varpi_E),1,\dots,1);\iota),
$$
where $\varpi_E$ is the $\nu$-uniformizer of $E$. To simplify the notation, we may use $\mu_i$ for $\mu_i(\varpi_E)$ in the following,
if it does not cause any confusion.

When $\nu$ is split over $E$, the Satake parameter of $\pi(\mu)$ is the semi-simple conjugacy class in ${}^{L}H_{n-\ell}$ of type
$$
c(\pi(\mu))=
(\diag(\theta_{1}(\varpi),\cdots,\theta_{\widetilde{m}_{H}}(\varpi),\vartheta_{1}^{-1}(\varpi),\dots\vartheta_{\widetilde{m}_{H}}^{-1}(\varpi));1)
$$
if $m$ is odd, and of type
$$
c(\pi(\mu))=
(\diag(\theta_{1}(\varpi),\cdots,\theta_{\widetilde{m}_{H}}(\varpi),\mu_0(\varpi),
\vartheta_{1}^{-1}(\varpi),\dots\vartheta_{\widetilde{m}_{H}}^{-1}(\varpi));1)
$$
if $m$ is even, where $\varpi$ is the $\nu$-uniformizer of $F$.

Next, we consider the irreducible unramified representation $\tau(\chi_\tau)$ of $\Res_{E/F}(\GL_{j})(F)$,
where $\chi_\tau=\otimes^{j}_{i=1}\chi_{i}$.
When $\nu$ is inert over $E$, the Satake parameter of $\tau(\chi_\tau)$ is the semi-simple conjugacy class in
${}^{L}\Res_{E/F}(\GL_{j})$ of type
$$
c(\tau(\chi_\tau))=(\diag(\chi_{1}(\varpi_E),\chi_{2}(\varpi_E),\dots,\chi_{j}(\varpi_E)), I_{j};\iota).
$$
Again, we use $\chi_i$ for $\chi_{i}(\varpi_E)$ if it does not cause any confusion.

When $\nu$ is split over $E$, the Satake parameter of $\tau(\chi_\tau)$ is the semi-simple conjugacy class in ${}^{L}\Res_{E/F}(\GL_{j})$ of type
$$
c(\tau(\chi_\tau))=(\diag(\Theta_{1},\dots,\Theta_{j}),\diag(\Xi_{1},\dots,\Xi_{j});1),
$$
where $\Theta_{i}$ is used for $\Theta_{i}(\varpi)$ and $\Xi_{i}$ is used for $\Xi_{i,\nu}(\varpi)$, to simplify the notation.

Therefore, if $\nu$ is inert over $E$ and $E$ is the unramified quadratic field extension of $F$, the unramified tensor product local $L$-function
$L(s,\pi\times\tau)$ is defined to be
\begin{equation}\label{eq:Lme}
\prod_{\substack{1\leq i\leq j\\ 1\leq i'\leq \tilde{m}_{H}}}(1-\chi_{i}\mu_{i'}q_{F}^{-2s})^{-1}(1-\chi_{i}\mu_{i'}^{-1}q_{F}^{-2s})^{-1}\prod_{1\leq k\leq n}(1-\chi_{k}q_{F}^{-2s})^{-1},
\end{equation}
if $m$ is even; and to be
\begin{equation}\label{eq:Lmo}
\prod_{\substack{1\leq i\leq j\\ 1\leq i'\leq \tilde{m}_{H}}}(1-\chi_{i}\mu_{i'}q_{F}^{-2s})^{-1}(1-\chi_{i}\mu_{i'}^{-1}q_{F}^{-2s})^{-1},
\end{equation}
if $m$ is odd. When $\nu$ is split over $E$, the unramified tensor product local $L$-function
$L(s,\pi\times\tau)$ is defined to be
\begin{equation}
L(s,\pi\times\tau)=L(s,\pi\times\tau_{1})L(s,\tilde{\pi}\times\tau_{2}),
\end{equation}
where $\tau_1$ and $\tau_2$ are defined according to $\Theta_1,\cdots,\Theta_j$ and $\Xi_1,\cdots,\Xi_j$, respectively.

Moreover, the unramified local $Asai$ $L$-function of $\tau$ is defined as, when $\nu$ is inert,
\begin{equation}
L(s,\tau,Asai)=\prod_{1\leq i_{1}<i_{2}\leq j}(1-\mu_{i_{1}}\mu_{i_{2}}q_{F}^{-2s})^{-1}\prod_{1\leq i\leq j}(1-\mu_{i}q_{F}^{-s})^{-1};
\end{equation}
and when $\nu$ is split
\begin{equation}
L(s,\tau,Asai)=L(s,\tau_{1}\times\tau_{2})=\prod_{1\leq i,k\leq j}(1-\Theta_{i}\Xi_{k}q_F^{-s})^{-1}.
\end{equation}

In the same way we define the unramified tensor product local $L$-function $L(s,\sigma\times\tau)$.

\subsection{Unramified local zeta integrals}
Let $f_{\chi_s}$ and $f_{\mu}$ be the spherical functions in $\Pi(\chi_s)$ and $\pi(\mu)$, normalized by
$f_{\chi}(e_{{G}})=f_{\mu}(e_{{H}})=1$. Denote by $f_{\tau}$ and $f_{\sig}$ the unramified function in $\tau$ and $\sig$ accordingly.
We are going to calculate explicitly the unramified local zeta integral $\CZ_{\nu}(s,f_{\chi_s},f_{\mu},\psi_{\ell,\kappa})$.

By Bernstein rationality theorem (\cite{GPSR87} and see also \cite{Bn98}), $\CZ_{\nu}(s,f_{\chi_s},f_{\mu},\psi_{\ell,\kappa})$ is a rational function of the parameters $\chi_s$ and $\mu$.
 Thus, we can assume that
\begin{equation}\label{eq:Zeta}
\CZ_{\nu}(s,f_{\chi_s},f_{\mu},\psi_{\ell,\kappa})=\frac{P(\chi_s,\mu)}{Q(\chi_s,\mu)},
\end{equation}
where $P(\chi_s,\mu)$ and $Q(\chi_s,\mu)$ are polynomials of variables in $\chi_{i}$, $\mu_{i}$ and $q_E^{-s}$.
We are going to calculate the polynomials $P(\chi_s,\mu)$ and $Q(\chi_s,\mu)$ explicitly in the following two subsections.

\subsection{Calculation of $Q(\chi_s,\mu)$}
For a technical reason, which will be mentioned in the argument below, we assume that $j=\ell+1$. This is enough to produce
the unramified local $L$-functions as needed. The method used here is an extension of that in \cite{GPSR97} to the unitary group case.
The idea to calculate $Q(\chi_s,\mu)$ is to find a proper Hecke algebra element $\Phi_0$ in the extended spherical Hecke algebra of
$H_{n-\ell}$ as defined below, so that for any section $f_{\chi_s}$ in the unramified induced representation
$$
\Ind^{G_{n}(F)}_{P_{j}(F)}(|\det|^{s}\tau\otimes\sig),
$$
the convolution $\CJ(f_{\chi_s}*\Phi_0)$ is supported in the Zariski open orbit, which will be specified below and
has the property that
$$
\CZ_{\nu}(s,f_{\chi_s}*\Phi_{0},f_{\mu},\psi_{\ell,\kappa})
=
Q(\chi_s,\mu)\cdot\CZ_{\nu}(s,f_{\chi_s},f_{\mu},\psi_{\ell,\kappa}).
$$
Since $\CJ(f_{\chi_s}*\Phi_0)$ is supported in the Zariski open orbit, the local zeta integral
$\CZ_{\nu}(s,f_{\chi_s}*\Phi_{0},f_{\mu},\psi_{\ell,\kappa})$ is
entire in $s$ and hence is expected to be $P(\chi_s,\mu)$ essentially.

Let $\CH(H_{n-\ell},K_{H})$ be the spherical Hecke algebra with convolution $\circ$ of all $K_H$-bi-invariant (smooth) functions
with compact supports on $H_{n-\ell}$.
Let $X_{i}$ for all $1\leq i\leq \tilde{m}_{H}$ be generators of the Hecke algebra $\CH(H_{n-\ell},K_{H})$. By the Satake isomorphism,
if $\nu$ is inert over $E$ and $E$ is the unramified quadratic field extension of $F$, the Hecke algebra can be realized as follows:
$$
\CH(H_{n-\ell},K_{H})\simeq \BC\bpair{X_{1},X_{1}^{-1},\dots,X_{\tilde{m}_{H}},X^{-1}_{\tilde{m}_{H}}}^{W(H_{n-\ell})};
$$
and if $\nu$ is split over $E$, the Hecke algebra can be realized as follows:
$$
\CH(H_{n-\ell},K_{H})\simeq \BC\bpair{X^{\pm 1}_{1},X^{\pm 1}_{2},\dots,X^{\pm 1}_{m-2\ell-1}}^{S_{m-2\ell-1}}.
$$
Here $S_{m-2\ell-1}$ is the symmetric group on the sets
$\cpair{X_{1},\dots,X_{m-2\ell-1}}$
and $\cpair{X^{-1}_{1},\dots,X^{-1}_{m-2\ell-1}}$.

Define an extended Hecke algebra as in \cite{GPSR97}:
$$
\CA_{H_{n-\ell}}:=\BC\bpair{X,X^{-1}}\otimes\CH(H_{n-\ell},K_{H}).
$$
Let $\Pi(\chi_s)$ be the unramified representation of $G_n(F)$ as defined in \S 4.2. We consider the subspace of all $K_{H}$-invariant vectors
$$
J^{*}_{\psi_{\ell,\kappa}}(\chi_s):=(J_{\psi_{\ell,\kappa}}(\chi_s))^{K_{H}}
$$
of the twisted Jacquet module $J_{\psi_{\ell,\kappa}}(\chi_s):=J_{\psi_{\ell,\kappa}}(\Pi(\chi_s))$. Although it is naturally
a module of the Hecke algebra $\CH(H_{n-\ell},K_{H})$, we may extend it to be a module of the extended Hecke algebra $\CA_{H_{n-\ell}}$
as follows: for $\phi\in J^{*}_{\psi_{\ell,\kappa}}(\chi_s)$ and $X\otimes \Phi\in \CA_{H_{n-\ell}}$,
$$
\phi*(X\otimes \Phi)=q_E^{-s}(\phi\circ\Phi),
$$
where $\phi\circ\Phi$ is the left action on $\phi$ via convolution. As in \cite{GPSR97}, define the \emph{support ideal} as follows:
$$
\CI_{supp}(\chi_s)=\cpair{\Phi\in \CA_{H_{n-\ell}}\mid J^{*}_{\psi_{\ell,\kappa}}(\chi_s)*\Phi \subseteq \Lam},
$$
where $\Lam$ is the smooth representation of $H_{n-\ell}(F)$ consisting of functions in $\Pi(\chi_s)$ supported in the open double cosets $P_{j}\eps_{0,1}\eta R_{\ell,w_{0}}$. More precisely, by Proposition~\ref{pro:Jac-mod} and~\ref{pro:Jac-mod-split}, the smooth
representation $\Lam$ can be realized via the following isomorphisms:
$$
\Lam\cong
 \ind^{H_{n-\ell}(F)}_{P'_{1,\ell}(F)} (|\det|^{-\frac{\ell}{2}+s}_E \tau_{(\ell)} \otimes \sigma^{w^{\ell}_b})
 $$
if $\nu$ is inert over $E$; and
 $$
 \Lam\cong
\ind^{\GL_{m-2j+1}(F)}_{\GL_{m-2j}(F)} (\sig)
$$
if $\nu$ is split over $E$. Here we use the assumption that $j=\ell+1$.

First we consider the case when $\ell=0$, which implies that $j=\ell+1=1$.
In the case, the twisted Jacquet functor is just the restriction to the subgroup $H_n(F)$ of $G_n(F)$.
By restricting to the subgroup $H_n(F)$, the induced representation
$$
\Pi=\Ind^{G_{n}}_{P_{1}}(|\cdot|_E^s\chi \otimes\sig)
$$
decomposes via an exact sequence of $H_n(F)$-modules, according to Proposition~\ref{pro:Jac-mod}.

If $\nu$ is inert over $E$, the case is similar to
\cite{GPSR97} and we have
$$
0\rightarrow \Ind^{H_{n}}_{G_{n-1}}(\sig)\rightarrow J_{\psi_{0,\kappa}}(\Pi)\rightarrow \Ind^{H_{n}}_{P'_{1}}(|t|_{E}^{\frac{1}{2}+s}\chi\otimes J_{\psi'_{0,\kappa}}(\sig))\rightarrow 0.
$$

If $\nu$ is split over $E$, more explanation is needed. The double coset decomposition
$$
P_{1,m-2,1}\bks \GL_{m}/H_{n}
$$
has 6 representatives for $m>2$, which are denoted by $\gamma_{i}$ for $1\leq i\leq 6$. Let $P_{1,m-2,1}\gamma_{1}H_{n}$ be  the open orbit, and $P_{1,m-2,1}\gamma_{i}H_{n}$ for $i=2$ or $i=3$ be the orbits with the greatest dimension in those orbits except the open orbit.
Using Proposition~\ref{pro:Jac-mod-split} repeatedly, we have
$$
0\rightarrow \ind^{\GL_{m-1}}_{\GL_{m-2}}(\sig) \rightarrow \Omega\rightarrow\Sigma\rightarrow 0,
$$
where
\begin{equation}\label{eq:Omega}
\Omega:=\{f\in \Pi\mid supp(f)\subseteq\cup^{3}_{i=1}P_{1,m-2,1}\gamma_{i} H_{n}\},
\end{equation}
and
$$
\Sigma:=
\Ind^{H_{n}}_{P_{1,m-2}}(|\cdot|^{\frac{1}{2}+s}\Theta\otimes\sig_{[0]})\oplus
\Ind^{H_{n}}_{P_{m-2,1}}(\sig_{[0]}\otimes|\cdot|^{-\frac{1}{2}-s}\Xi^{-1}).
$$

\begin{lem}\label{lm:l=0}
Assume that $\ell=0$ and $j=\ell+1=1$.
The support ideal $\CI_{supp}(\chi_s)$ contains
$$
\Phi_{0}=
\prod_{i}(1-q_E^{-\frac{1}{2}}\chi(\varpi)XX_{i})(1-q_E^{-\frac{1}{2}}\chi(\varpi)XX^{-1}_{i})
$$
if $\nu$ is inert over $E$, and
$$
\Phi_{0}=
\prod_{i}(1-q_E^{-\frac{1}{2}}\Theta(\varpi)XX_{i})(1-q_E^{-\frac{1}{2}}\Xi(\varpi)XX^{-1}_{i})
$$
if $\nu$ is split over $E$.
\end{lem}

\begin{proof}
The proof follows the same argument used in \cite[\S 2, Lemma 2.1]{GPSR97}, which uses the Satake Isomorphism for $F$-quasisplit classical groups
and the definition of the support ideal $\CI_{supp}(\chi_s)$. We omit the details here.
\end{proof}

Next, we deal with the general case with $j=\ell+1$ for the relation between $H_{n-\ell}$ and $G_{n-\ell}$.

If $\nu$ is inert over $E$, by Proposition~\ref{pro:Jac-mod}, we have the exact sequence of $H_{n-\ell}(E)$ modules for $j=\ell+1$,
\begin{eqnarray*}
0\rightarrow \ind^{H_{n-\ell}}_{P''_{1}} |\cdot|^{s-\frac{\ell}{2}}_E \tau_{(\ell)} \otimes \sigma^{w^{\ell}_q}
\rightarrow\Pi_{\eps_{0,1}\eta_{\eps,I_{m-2\ell}}} \rightarrow  \CY\rightarrow 0
\end{eqnarray*}
where
$$
\CY:=
\ind^{H_{n-\ell}}_{P'_{1}} |\cdot| _{E}^{\frac{1-\ell}{2}+s} \tau^{(\ell)} \otimes J_{\psi'_{0,\kappa}}(\sig^{w^{\ell}_{q}}),
$$
and $\Pi_{\eps_{0,1}\eta_{\eps,I_{m-2\ell}}}$ is the smooth representation of $H_{n-\ell}$ consisting of functions in $\pi(\chi_s)$ which are
supported in $P_{j}\eps_{0,1}\eta_{\eps,I_{m-2\ell}}N_{\ell}G_{n-\ell}$.
Recall that $\tau^{(\ell)}$ is the $\ell$-th Bernstein-Zelevinski derivative of $\tau$, which is a representation of $\GL_{1}(E)$. Up to semi-simplification,
$$
\tau^{(\ell)}=\oplus^{j}_{i=1}\chi_{i}\otimes |\cdot|_{E}^{\frac{\ell}{2}},
$$
and then
$\CY\equiv \oplus^{j}_{i=1} \ind^{H_{n-\ell}}_{P'_{1}} |\det| _{E}^{\frac{1}{2}+s} \chi_{i} \otimes J_{\psi'_{0,\kappa}}(\sig^{w^{\ell}_{q}})$.

If $\nu$ is split, we apply Proposition~\ref{pro:Jac-mod-split} repeatedly and obtain the exact sequence
\begin{eqnarray*}
0\rightarrow \ind^{\GL_{m-2\ell-1}}_{\GL_{m-2j}}\sig_{(0)}\rightarrow \Omega
\rightarrow\CU\rightarrow 0,
\end{eqnarray*}
where $\CU$ is defined to be the following representation
$$
\Ind^{\GL_{m-2\ell-1}}_{P_{1,m-2j}}(|\det|^{\frac{1-\ell}{2}+s}(\tau_{1})^{(\ell)}\otimes \sig)
\oplus
\Ind^{\GL_{m-2\ell-1}}_{P_{1,m-2j}} (\sig\otimes |\det|^{\frac{\ell-1}{2}-s}(\tau^{*}_{2})^{[\ell]})
$$
and $\Omega$ is defined in \eqref{eq:Omega} consisting of functions supported in the first greatest orbits.
In this case, we have, up to semi-simplification,
$$
\tau^{(\ell)}_{1}=\oplus^{j}_{i=1}\Theta_{i}\otimes|\cdot|^{\frac{\ell}{2}}
\text{ and }
(\tau^{*}_{2})^{[\ell]}= \oplus^{j}_{i=1}\Xi^{-1}_{i}\otimes |\cdot|^{-\frac{\ell}{2}}.
$$

Note that $\Phi\in \CI_{supp}(\chi_s)$ if and only if $\Phi$ annihilates all the boundary components of $J_{\psi_{\ell,\kappa}}(\Pi)$, that is,
all the summands in Proposition~\ref{pro:Jac-mod} and  Proposition~\ref{pro:Jac-mod-split} except the space
$\ind^{H_{n-\ell}}_{P'_{1,\ell}} (|\det|^{-\frac{\ell}{2}+s}_E \tau_{(\ell)} \otimes \sigma^{w^{\ell}_b})$ and the space
$\ind^{\GL_{m-2\ell-1}}_{\GL_{m-2j}}\sig$, respectively. It is sufficient to annihilate the quotients in $\Pi_{\eps_{0,1}}$ and $\Omega$.

In order to annihilate $K_{H}$-fixed vectors in the space
$$
\oplus^{j}_{i=1}\ind^{H_{n-\ell}}_{P'_{1}} (|\det| _{E}^{\frac{1}{2}+s} \chi_{i} \otimes J_{\psi'_{0,\kappa}}(\sig^{w^{\ell}_{b}}))
$$
if $\nu$ is inert, and in the space
$$
\oplus^{j}_{i=1}\Ind^{\GL_{m-2\ell-1}}_{P_{1,m-2j}}(|\cdot|^{\frac{1}{2}+s}\Theta_{i}\oplus |\cdot|^{-\frac{1}{2}-s}\Xi^{-1}_{i})\otimes\sig
$$
(up to isomorphism) if $\nu$ is split,
as in Lemma~\ref{lm:l=0}, we may take the following specific element in $\CA_{H_{n-\ell}}$,
\begin{equation}
\Phi_{0}=\begin{cases}
\prod^{j}_{i=1}\prod^{\tilde{m}_{H}}_{i'=1}(1-q_E^{-\frac{1}{2}}\chi_{i}XX_{i'})(1-q_E^{-\frac{1}{2}}\chi_{i}XX^{-1}_{i'})&
\text{ if  $\nu$ is inert,}\\
\prod^{j}_{i=1}\prod^{m-2\ell-1}_{i'=1}(1-q_E^{-\frac{1}{2}}\Theta_{i}XX_{i'})(1-q_E^{-\frac{1}{2}}\Xi_{i} X X^{-1}_{i'})
& \text{ if  $\nu$ is split,}
\end{cases}
\end{equation}
which is an element of the support ideal $\CI_{supp}(\chi_s)$.
In addition, all the other boundary components of the Jacquet module $J_{\psi_{\ell,\kappa}}(\chi_{s})$ are of form
$$
\ind^{H_{n-\ell}}_{P'_{\beta}}(|\det|^{\frac{1-t}{2}+s}_{E}\tau^{(t)}\otimes \sig')
$$
if $\nu$ is inert, and of form
$$
\ind^{\GL_{m-2\ell-1}}_{P_{\beta,m-2\ell-1-\beta}}(|\det|^{\frac{1-t}{2}+s}\Xi_{\tau}\otimes\sig')
$$
or
$$
\ind^{\GL_{m-2\ell-1}}_{P_{m-2\ell-1-\beta,\beta}}
(\sig'\otimes |\det|^{-\frac{1-t}{2}-s}\Theta_{\tau})
$$
if $\nu$ is split,
where $\sig'$ is a suitable representation independent of $s$, more details of which can be found in Proposition~\ref{pro:Jac-mod} and  Proposition~\ref{pro:Jac-mod-split}.
It is easy to check that $\Phi_{0}$ also annihilates $K_{H}$-fixed vectors in those boundary components.

\begin{prop}\label{pro:Q}
With $\Phi_0\in\CI_{supp}(\chi_s)$ as chosen above, the following identity holds:
\begin{equation}
\CZ_{\nu}(s,f_{\chi}*\Phi_{0},f_{\mu},\psi_{\ell,\kappa})=
Q(\chi_s,\mu)\cdot\CZ_{\nu}(s,f_{\chi},f_{\mu},\psi_{\ell,\kappa})
\end{equation}
where
$Q(\chi_s,\mu)$ is defined to be
$$
\prod^{j}_{i=1}\prod^{\tilde{m}_{H}}_{i'=1}(1-q_E^{-\frac{1}{2}-s}\chi_{i}\mu_{i'})(1-q_E^{-\frac{1}{2}-s}\chi_{i}\mu^{-1}_{i'})
$$
if  $\nu$ is inert, and to be
$$
\prod^{j}_{i=1}\prod^{m-2\ell-1}_{i'=1}(1-q_E^{-\frac{1}{2}-s}\Theta_{i}\mu_{i'})(1-q_E^{-\frac{1}{2}-s}\Xi_{i}\mu^{-1}_{i'})
$$
if  $\nu$ is split.
Moreover, $\CZ_{\nu}(s,f_{\chi_s}*\Phi_{0},f_{\mu},\psi_{\ell,\kappa})$ is a polynomial function of parameters $\chi_{\tau}$ and in $q_E^{-s}$.
\end{prop}

\begin{proof}
The proof is similar to the proof of Theorem 5.1 in \cite{GPSR97}.
In fact, $\CJ(f_{\chi_s}*\Phi_{0})(h)$, as defined in Section 4.1, belongs to the space $\Lam$, which is independent of the choice of $\sig$.
Also $\CJ(f_{\chi_s}*\Phi_{0})(h)$ is analytic in $s$ because of the support of $\CJ(f_{\chi_s}*\Phi_{0})$.
The local zeta integral is equal to the pairing the function $\CJ(f_{\chi_s}*\Phi_{0})(h)$ with a Bessel function as in (4.3),
and is absolutely convergent for all $s$. Hence
the zeta function $\CZ_{\nu}(s,f_{\chi_s}*\Phi_{0},f_{\mu},\psi_{\ell,\kappa})$ is a polynomial function of $q_E^{s}$ and $q_E^{-s}$
for all choice of $\pi$ and all $s$.
\end{proof}

Remark that the proof of this proposition only uses the genericity of $\tau$, which is true because $\tau$ is the unramified local component of
the corresponding irreducible automorphic representation of $\GL_j(\BA_E)$ as given in (2.13). Hence it holds for all choices $\chi_\sig$ and $\mu$, and therefore, for all irreducible unramified representations $\sig$ and $\pi$.

Following the definition of the unramified local tensor product $L$-functions as in \eqref{eq:Lme} and \eqref{eq:Lmo}, and Proposition
\ref{pro:Q}, one must have the following identity:
$$
Q(\chi_s,\mu)=\begin{cases}
L^{-1}(\frac{1}{2}+s,\tau\times\pi)d(\chi_{\tau},s) &\text{ if $\nu$ is inert,}\\
L^{-1}(\frac{1}{2}+s,\tau\times\pi) &\text{ if $\nu$ is split.}
\end{cases}
$$
where
$$
d(\chi_{\tau},s)=\begin{cases}
\prod^{j}_{i=1}(1-q_E^{-\frac{1}{2}-s}\chi_{i})^{-1} & \text{ if $m$ is even and $\nu$ is inert};\\
1 & \text{ otherwise}.
\end{cases}
$$
Note that $d(\chi_{i},s)=(1-q_E^{-\frac{1}{2}-s}\chi_{i})^{-1}$.
Thus, based on the calculation of $Q(\chi_s,\mu)$, we have a unique choice of $P(\chi_s,\mu)$.

\subsection{Calculation of $P(\chi_s,\mu)$}\label{sec:P}
In this section, we will first calculate the numerator $P(\chi_s,\mu)$ when $\Pi$ and $\pi$ are \emph{generic spherical}, and then
extend the results to general case by {\it Density Principle} in Appendix IV to \cite[Section 5]{GPSR97}.

Choose suitable functions $\varphi_{1}\in \CS(G_{n})$ and $\varphi_{2}\in \CS(H_{n-\ell})$ such that
\begin{align*}
&f_{\chi_s}(g)=\int_{B}\varphi_{1}(bg)\chi_s^{-1}\delta^{\frac{1}{2}}_{B}(b)\ud_{l}b,\\
&f_{\mu}(g)=\int_{B_{H}}\varphi_{2}(bg)\mu^{-1}\delta^{\frac{1}{2}}_{B_{H}}(b)\ud_{l}b,
\end{align*}
where ${\rm d}_{l}b$ is the left invariant Haar measure on Borel subgroups. Then we define a linear functional in the hom space \eqref{eq:hom-local},
\begin{align*}
T(f_{\chi_s},f_{\mu}):=&\int_{R_{\ell,\kappa}}f_{\chi_s}(\eps_{0,1}\eta nm )f_\mu(m)
\psi_{\ell,\kappa}(n)\ud n\ud m.
\end{align*}
Remark that properties of $T$ are studied in \cite{K08} when $\nu$ is inert.

\begin{lem}
$$
\CZ_{\nu}(s,f_{\chi_s},f_{\mu},\psi_{\ell,\kappa})=T(f_{\chi_s}, f_{\mu}).
$$
\end{lem}
\begin{proof}
For all unramified places, the proof is similar to the orthogonal case as Theorem (A) in Appendix I to \cite[Chapter 5]{GPSR97}.
\end{proof}

{\bf Case $\ell=0$:}\
First of all, we consider the case $\ell=0$, and hence $j=1$. The Bessel period is also studied by Gan, Gross, and Prasad in \cite{GGP12}.
Referring to \cite[Proposition 2.5]{Har12}, we have the following inductive formula
$$
T(f_{\chi_{1}\otimes\chi_{\sig}},f_{\mu})=
\frac{L(\frac{1}{2},\chi_{1}\times\pi)}{L(1,\chi_{1}\times\sig)L(1,\xi^{m}_{E/F}\otimes\chi_{1})}T(f_{\mu},f_{\chi_{\sig}})
$$
if $\nu$ is inert, and
$$
T(f_{\chi_{1}\otimes\chi_{\sig}},f_{\mu})=
\frac{L(\frac{1}{2},\Theta_{1}\times\pi)L(\frac{1}{2},\Xi_{1}\times\tilde{\pi})}{L(1,\Theta_{1}\times\tilde{\sig})L(1,\Xi_{1}\times\sig)L(1,\Theta_{1}\Xi_{1})} $$
if $\nu$ is inert, for any quasi-character $\chi_1$. Correspondingly, one has
$$
Q(\chi_{1}\otimes\chi_{\sig},\mu)=\begin{cases}
[L\ppair{\frac{1}{2},\chi_{1}\times\pi}d(\chi_{1})]^{-1} & \text{ if $\nu$ is inert,}\\
[L(\frac{1}{2},\Theta_{1}\times\pi)L(\frac{1}{2},\Xi_{1}\times \tilde{\pi})]^{-1}& \text{ if $\nu$ is split}.
\end{cases}
$$
which is the same as the result of Proposition~\ref{pro:Q}. Hence one has
\begin{equation}
P(\chi_{1}\otimes\chi_{\sig},\mu)=\frac{d(\chi_{1})}{L(1,\chi_{1}\times\sig)L(1,\chi_{1}\otimes \xi^{m}_{E/F})}T(\mu,\chi_{\sig}). \label{eq:l=0-P}
\end{equation}
Note that $P(\chi_{1}\otimes\chi_{\sig},\mu)$ is a polynomial function of the parameter $\chi_{1}$, and $L(1,\chi_{1}\otimes\xi^{m}_{E/F})=L(1,\Theta_{1}\otimes\Xi_{1})$. A comment with the notation $\chi_1$ is in order. The above discussion
holds for all quasi-characters $\chi_1$ and hence the variable $s$ is carried by this character $\chi_1$ here.

{\bf General Case $\ell>0$:}\
In the discussion below, we also assume that $\chi$ is a general quasi-character, i.e. we take $\chi$ to be $\chi_s$ here,
since the proof works for any quasi-character $\chi$. Hence in the
discussion, there will be no variable $s$. However, the variable $s$ will be put back to the final formula.

Let $\omega$ be an element of Weyl group $W(H_{n-\ell})$ and $I_{\omega}$ be the intertwining operator mapping $\Pi(\chi)$ to $\Pi(\ome\chi)$.
By the uniqueness of Bessel model, we have a local gamma factor $\gamma_{\ome}(\chi,\gamma)$ defined by
$$
T(I_{\ome}(f_{\chi}),f_{\mu})=\gamma_{\ome}(\chi,\mu)T(f_{\chi},f_{\mu}).
$$
Note that the definition of $\gamma_{\ome}$ is independent with non-trivial choice of $T$.
In order to calculate the general case $\ell>0$. We need to calculate the local gamma factor $\gamma_{\ome}$.

When $\nu$ is inert, let $\{\beta_{i}\mid 1\leq i\leq \tilde{m}\}$ be a set of simple roots of $G_{n}$.
Then the sets $\{\beta_{i}\mid 1\leq i\leq \ell\}$ and $\{\beta_{i}\mid \ell+2\leq i\leq \tilde{m}\}$ are also sets of simple roots of $\GL_{\ell+1}(E)$ and $H(W_{\ell+1})$ respectively.

When $\nu$ is split, let $\{\beta{_{i}}\mid 1\leq i\leq m-1\}$ be a set of simple roots of $\GL_{m}$.
Recall that $P_{\ell+1,m-2\ell-2,\ell+1}$ is a standard parabolic subgroup of $\GL_{m}$ with the Levi subgroup $\GL_{\ell+1}\times \GL_{m-2\ell-2}\times \GL_{\ell+1}$.
Then the set $\{\beta_{i}\mid 1\leq i\leq \ell\}$ and $\{\beta_{i}\mid m-\ell\leq i\leq m-1\}$ are sets of simple roots of the general linear groups of the Levi subgroup, and $\{\beta_{i}\mid \ell+2\leq i\leq m-\ell-2\}$ is the set of simple roots of the subgroup $\GL_{m-2\ell-2}$ of the Levi subgroup.
Let $\ome_{i}$ be the simple reflection corresponding to the simple root $\beta_{i}$.

\begin{lem}\label{lm:gamma}
If $\nu$ is inert, then
$$
\gamma_{\ome_{i}}(\chi,\mu)=\begin{cases}
\frac{1-\chi_{i+1}\chi^{-1}_{i}q_E^{-1}}{1-\chi_{i}\chi_{i+1}^{-1}} &\text{ if } 1\leq i\leq \ell,\\
\gamma_{\ome_{i}}(\chi_{\ell+1}\otimes\chi_{\sig},\mu) &\text{ if } \ell+1\leq i\leq \tilde{m}.
\end{cases}
$$
If $\nu$ is split, then the gamma factor $\gamma_{\ome_{i}}(\chi,\mu)$ is equal to
$$
\begin{cases}
\frac{1-\chi_{i+1}\chi^{-1}_{i}q_E^{-1}}{1-\chi_{i}\chi_{i+1}^{-1}}&\text{ if } 1\leq i\leq \ell \text{ or } m-\ell\leq i\leq m,\\
\gamma_{\ome_{i}}(\chi_{\ell+1}\otimes\chi_{\sig}\otimes\chi_{m-\ell},\mu)&\text{ if } \ell+1\leq i\leq m-\ell-1 .
\end{cases}
$$
\end{lem}

\begin{proof}
Khoury proved the inert case in \cite[Proposition 11.1]{K08}.
For the split case, the proof is given in \cite{Z12}.
\end{proof}

Now, we normalize the numerator $P(\chi,\mu)$ by
$$
P^{*}(\chi,\mu)=\frac{\zeta(\chi,1)T(\mu,\chi_{\sig})}{P(\chi,\mu)}.
$$
Note that $T(\mu,\chi_{\sig})$ is the pairing for $\pi(\mu)$ and $\sig$.

Following \cite{CS80} and \cite[Section 3.5]{Sh10}, the functions $\zeta(\chi,t)$ can be defined as follows.
When $\nu$ is inert, if $m$ is even, $\zeta(\chi,t)$ is defined by
$$
\prod_{1\leq i_{1}<i_{2}\leq \tilde{m}}(1-\chi_{i_{1}}\chi_{i_{2}}^{-1}q^{-t})
(1-\chi_{i_{1}}\chi_{i_{2}}q^{-t})\cdot\prod_{1\leq i\leq \tilde{m}}(1-\chi_{i}q_{F}^{-t});
$$
and if $m$ is odd, $\zeta(\chi,t)$ is defined by
$$
\prod_{1\leq i_{1}<i_{2}\leq \tilde{m}}(1-\chi_{i_{1}}\chi_{i_{2}}^{-1}q^{-t})
(1-\chi_{i_{1}}\chi_{i_{2}}q^{-t})\cdot\prod_{1\leq i\leq \tilde{m}}(1+\chi_{i}q_{F}^{-t})(1-\chi_{i}q^{-t}).
$$
When $\nu$ is split, $\zeta(\chi,t)=\prod_{1\leq i_{1}<i_{2}\leq m}(1-\chi_{i_{1}}\chi^{-1}_{i_{2}}q^{-t})$.
Note that $q=q_E$ in the above formulas, in order to simplify the notation.
In addition, if $\tilde{m}=1$, $\zeta(\chi, t)=1$ for all cases.
Remark that $\zeta(\chi,t)$ is the zeta polynomial function associated to $G_{n}$ as in \cite[Page 157]{GPSR97}.

For the case $\ell=0$, according to~\eqref{eq:l=0-P}, we have
\begin{equation}\label{eq:l=0}
P^{*}(\chi_{1}\otimes\chi_{\sig},\mu)=\frac{\zeta(\chi_{\sig},1)}{d(\chi_{1})},
\end{equation}
where $\zeta(\chi_{\sig},1)$ is the zeta polynomial function associated with $H_{n}$, as in \cite[Page 157]{GPSR97}.

\begin{cor}\label{cor:gamma}
If $1\leq i\leq \ell$, or $m-\ell\leq i\leq m-1$ when $\nu$ is split, then
$$
P^{*}(\ome_{i}\chi,\mu)=P^{*}(\chi,\mu).
$$

If $i=\ell+1$ when $\nu$ is inert, or $i=\ell+1$ or $m-\ell-1$ when $\nu$ is split, then
$$
\frac{P^{*}(\chi,\mu)}{P^{*}(\ome_{i}\chi,\mu)}=\frac{\zeta(\chi_{\sig},1)d(\chi_{i})}{\zeta(\chi_{\sig'},1)d(\chi_{i+1})}.
$$
where $\chi_{\sig'}=\chi_{\ell+1}\otimes\chi_{\ell+3}\otimes\cdots\otimes\chi_{\tilde{m}}$ when $i=\ell_{1}$ and $\nu$ is inert, and $\chi_{\sig'}=\chi_{\ell+1}\otimes\chi_{\ell+3}\otimes\cdots\otimes\chi_{m-\ell-1}$ when $i=\ell$ and $\nu$ is split, and $\chi_{\sig'}=\chi_{\ell+2}\otimes\cdots\otimes\chi_{m-\ell-2}\otimes\chi_{m-\ell}$.

If $\ell+1<i\leq \tilde{m}$ when $\nu$ is inert or $\ell+2\leq  i\leq m-\ell-2$ when $\nu$ is split, then
$$
\frac{P^{*}(\chi,\mu)}{P^{*}(\ome_{i}\chi,\mu)}=\frac{\zeta(\chi_{\sig},1)}{\zeta((\ome\chi)_{\sig},1)}.
$$
\end{cor}
\begin{proof}
The proof is a straightforward calculation by Lemma~\ref{lm:gamma}.
%
%
\end{proof}

By Corollary~\ref{cor:gamma},
$$
\frac{P^{*}(\chi,\mu)d(\chi_{\tau})}{\zeta(\chi_{\sig},1)}
$$
is invariant under the action of the Weyl group $W(G_{n})$ on $\chi$.
In the rest of this section, we will show that the quotient above is equal to one, i.e.
\begin{equation} \label{eq:P*}
P^{*}(\chi,\mu)=\frac{\zeta(\chi_{\sig},1)}{d(\chi_{\tau})}.
\end{equation}

Let $T_{0}(\chi,\mu)=T(f_{\chi}^{0},f_{\mu})$, where
$$
f_{\chi}^{0}(g)=\int_{B}1_{B(\Fo)\ome_{G_{n}}B(\Fo)}(bg)\chi^{-1}\delta^{\frac{1}{2}}_{B}(b)\ud_{l}b
$$
$\ome_{G_{n}}$ is the longest Weyl element in $G_n$, and $1_{B(\Fo)w_{G_{n}}B(\Fo)}$ is the characteristic function over $B(\Fo)w_{G_{n}}B(\Fo)$
and also is an Iwahori-fixed function.

\begin{lem}
$$
T_{0}(\chi,\mu)=T_{0}(\chi\vert_{H_{n-\ell}},\mu).
$$
\end{lem}
\begin{proof}
The proof is similar to Proposition 8.1 in \cite{GPSR97}.
\end{proof}

By Appendix to \S 6 in \cite{GPSR97}, we have the following expansion,
$$
T(\chi,\mu)=\sum_{\ome\in W(G_{n})}\frac{\gamma_{\ome}(\ome^{-1}\chi,\mu)}{c_{\ome}(\ome^{-1}\chi)}c_{\ome_{G_{n}}}(\ome^{-1}\chi)T_{0}(\ome^{-1}\chi,\mu),
$$
where $c_{\ome}(\ome^{-1}\chi)$ is the Harish-Chandra $c$-function of the intertwining operator associated to the Weyl group element $\ome$.
In this formula, by replacing $\gamma_{\omega}(\ome^{-1}\chi,\mu)$ by the following expression:
$$
\gamma_{\omega}(\ome^{-1}\chi,\mu)=\frac{T(\chi,\mu)c_{\ome}(\ome^{-1}\chi)}{T(\ome^{-1}\chi,\mu)},
$$
canceling both sides the factor $T(\chi,\mu)$, and replacing $T(\ome^{-1}\chi,\mu)$ by
$$
T(\ome^{-1}\chi,\mu)=\frac{P(\ome^{-1}\chi,\mu)}{Q(\ome^{-1}\chi,\mu)},
$$
we obtain the following expression:
\begin{align*}
1=&\sum_{\ome\in W(G_{n})}\frac{Q(\ome^{-1}\chi,\mu)}{P(\ome^{-1}\chi,\mu)}c_{\ome_{G_{n}}}(\ome^{-1}\chi)T_{0}(\ome^{-1}\chi,\mu)\\
=&\sum_{\ome\in W(G_{n})}\frac{c_{\ome_{G_{n}}}(\ome^{-1}\chi)}{\zeta(\ome^{-1}\chi,1)}Q(\ome^{-1}\chi,\mu)P^{*}(\ome^{-1}\chi,\mu)\frac{T_{0}(\ome^{-1}\chi,\mu)}{T(\mu,(\ome^{-1}\chi)_{\sig})}.
\end{align*}

Define
$$
\Delta(\chi)=q^{\apair{\varrho,\chi}}\zeta(\chi,0)=
\prod^{\tilde{m}}_{i=1}\chi_{i}^{-(\frac{m+1}{2}-i)}\zeta(\chi,0),
$$
where $\varrho$ is the half of the sum of all positive roots.
Then it follows that $\Delta(\ome\chi)=\sgn(\ome)\Delta(\chi)$. Note that $c_{\ome_{G_{n}}}(\chi)=\zeta(\chi,1)\zeta^{-1}(\chi,0)$.
It follows that $\Delta(\chi)$ can be expressed as follows:
\begin{align}
&\sum_{\ome\in W(G_{n})}\sgn(\ome)q^{\apair{\varrho,\ome^{-1}\chi}}Q(\ome^{-1}\chi,\mu)P^{*}(\ome^{-1}\chi,\mu)\frac{T_{0}(\ome^{-1}\chi,\mu)}{T(\mu,(\ome^{-1}\chi)_{\sig})} \label{eq:Delta-1}\\
&=\frac{P^{*}(\chi,\mu)d(\chi_{\tau})}{\zeta(\chi_{\sig},1)}\sum_{\ome\in W(G_{n})}\sgn(\ome)q^{\apair{\varrho,\ome\chi}}Q(\ome\chi,\mu)\frac{T_{0}(\ome\chi,\mu)\zeta((\ome\chi)_{\sig},1)}{T(\mu,(\ome\chi)_{\sig})d((\ome\chi)_{\tau})}\nonumber.
\end{align}

In order to prove Equation~\eqref{eq:P*}, it is sufficient to show the following Lemma,
which is similar to the orthogonal case (\cite[Lemma 6.3]{GPSR97}).

\begin{lem}
\begin{equation}
\Delta(\chi)=\sum_{\ome\in W(G_{n})}\sgn(\ome)q^{\apair{\varrho,\ome\chi}}Q(\ome\chi,\mu)\frac{T_{0}((\ome\chi)\vert_{H_{n-\ell}},\mu)\zeta((\ome\chi)_{\sig},1)}{T(\mu,(\ome\chi)_{\sig})d((\ome\chi)_{\tau})}.
\label{eq:Delta-2}
\end{equation}
\end{lem}

\begin{proof}
We only give a proof for the inert case. For the split case, the proof is similar and we omit details here.

First, by Equation~\eqref{eq:l=0}, this identity holds for $\ell=0$.

Next, we consider  the general cases $\ell>0$. Since the terms $\zeta(\chi_{\sig},1)$, $Q(\chi,\mu)$, $d(\chi_{\tau})$, $T_{0}(\chi\vert_{H_{n-\ell}},\mu)$ and $T(\mu,(\ome\chi)_{\sig})$ are invariant under the action of
the Weyl group $W(\GL_{\ell+1})$, we have the right hand side of the identity,
\begin{align*}
RHS=&\sum_{\ome\in W(GL_{\ell}) \times W(G_{n-\ell})\bks W(G_{n})}\Sigma_{\ome_1}(\ome)
\cdot\Sigma_{\ome_2}(\ome)\cdot
 q^{\apair{\varrho_{U_{\ell}},\ome\chi}}\sgn(\ome).
\end{align*}
where
$$
\Sigma_{\ome_1}(\ome):=\sum_{\ome_1\in W(GL_{\ell})}\sgn(\ome_1)q^{\apair{\varrho_{\GL_{\ell}},\ome_{1}\ome\chi}},
$$
and
\begin{eqnarray*}
\Sigma_{\ome_2}(\ome)
&:=&
\sum_{\ome_2 \in W(G_{n-\ell})} \sgn(\ome_2)q^{\apair{\varrho_{G_{n-\ell}},\ome_{2}\ome\chi}}
Q(\ome_{2}\ome\chi,\mu)\\
&&\ \ \ \ \ \ \ \ \ \ \ \ \cdot \frac{T_{0}((\ome_{2}\ome\chi)\vert_{H_{n-\ell}},\mu)\zeta((\ome_{2}\ome\chi)_{\sig},1)}{T(\mu,(\ome_{2}\ome\chi)_{\sig})d((\ome_{2}\ome\chi)_{\tau})}.
\end{eqnarray*}
Decompose as $Q(\chi,\mu)=Q_{1}(\chi,\mu)Q(\chi_{\ell+1}\otimes\chi_{\sig},\mu)$, where
$$
Q_{1}(\chi,\mu)=\prod^{\ell}_{i=1}\prod^{\tilde{m}_{L}}_{i'=1}(1-q^{-\frac{1}{2}}\chi_{i}\mu_{i'})(1-q^{-\frac{1}{2}}\chi_{i}\mu^{-1}_{i'})
$$
and
$$
Q(\chi_{\ell+1}\otimes\chi_{\sig},\mu)=\prod^{\tilde{m}_{L}}_{i'=1}(1-q^{-\frac{1}{2}}\chi_{\ell+1}\mu_{i'})(1-q^{-\frac{1}{2}}\chi_{\ell+1}\mu^{-1}_{i'}).
$$
Thus, $Q(\ome_{2}\chi,\mu)=Q_{1}(\chi,\mu)Q(\ome_{2}(\chi_{\ell+1}\otimes\chi_{\sig}),\mu)$ for $\ome_{2}\in W(G_{n-\ell})$.

Define $\ome\chi=\chi^{(1)}\otimes\chi^{(2)}$, where $\chi^{(1)}=\ome\chi\vert_{\GL_{\ell}}$ and $\chi^{(2)}=\ome\chi\vert_{G_{n-\ell}}$.
Note that
\begin{eqnarray*}
\apair{\varrho_{H_{n-\ell}},\ome_{2}\ome\chi}&=&\apair{\varrho_{H_{n-\ell}},\ome_{2}\chi^{(2)}},\\ \zeta((\ome_{2}\ome\chi)_{\sig},1)&=&\zeta((\ome_{2}\chi^{(2)})_{\sig},1),
\end{eqnarray*}
and
\begin{align*}
d((\ome_{2}\ome\chi)_{\tau})=&d((\ome_{2}\chi^{(2)})_{\ell+1})\prod^{\ell}_{i=1}(1-q^{-1}\chi_{i}(\varpi_{E}))\\
=&d((\ome_{2}\chi^{(2)})_{\ell+1})d((\ome\chi)_{\tau})d^{-1}((\ome\chi)_{\ell+1}).
\end{align*}
Consider the summation
\begin{align*}
\Sigma_{\ome_2}(\ome)
=&Q_{1}(\ome\chi,\mu)\frac{d((\ome\chi)_{\ell+1})}{d((\ome\chi)_{\tau})}\\
&\cdot \sum_{\ome_2 \in W(G_{n-\ell})}\sgn(\ome_2)q^{\apair{\varrho_{G_{n-\ell}},\ome_{2}\chi^{2}}}Q(\ome_{2}\chi^{(2)},\mu)\\
&\ \ \ \ \ \ \ \ \ \ \ \ \ \cdot\frac{T_{0}(w_{2}\chi^{(2)},\mu)\zeta((\ome_{2}\chi^{(2)})_{\sig},1)}{T(\mu,\ome_{2}\chi^{(2)})d((\ome_{2}\chi^{(2)})_{\ell+1})}\\
=&Q_{1}(\ome\chi,\mu)\frac{d((\ome\chi)_{\ell+1})}{d((\ome\chi)_{\tau})}\cdot \Delta_{H_{n-\ell}}(\chi^{(2)}).
\end{align*}
The last identity holds by the case $\ell=0$. Note that $\chi^{(2)}=\ome\chi\vert_{G_{n-\ell}}$

Now, by replacing $\Sigma_{\ome_2}(\ome)$ by the expression above, the right hand side of (4.18) reduces to
\begin{align*}
RHS=&\sum_{\ome\in W(\GL_{\ell})\times W(G_{n-\ell})\bks W(G_{n})}
\Sigma_{\ome_1}(\ome)\\
&\ \ \ \ \ \ \ \ \ \cdot Q_{1}(\ome\chi,\mu)\frac{d((\ome\chi)_{\ell+1})}{d((\ome\chi)_{\tau})}\cdot \Delta_{H_{n-\ell}}(\chi^{(2)})
\sgn(\ome)q^{\apair{\varrho_{U_{\ell}},\ome\chi}}.
\end{align*}
By using the definition of $\Sigma_{\ome_1}(\ome)$ and the definition of $\Delta_{H_{n-\ell}}(\chi^{(2)})$, and
then by collapsing the three summations $\sum_\ome$, $\sum_{\ome_1}$ and $\sum_{\ome_2}$, we obtain that
\begin{align*}
RHS=&\sum_{\ome\in W(G_{n})}\sgn(\ome)q^{\apair{\varrho,\ome\chi}}Q_{1}(\ome\chi,\mu)\frac{d((\ome\chi)_{\ell+1})}{d((\ome\chi)_{\tau})}.
\end{align*}

Recall that
$$
Q_{1}(\chi,\mu)\frac{d((\chi)_{\ell+1})}{d((\chi)_{\tau})}=\prod^{\ell}_{i=1}\prod^{\tilde{m}_{H}}_{i'=1}(1-q^{-\frac{1}{2}}\chi_{i}\mu_{i'})(1-q^{-\frac{1}{2}}\chi_{i}\mu^{-1}_{i'}).
$$
Then
\begin{align*}
RHS=&\sum_{\ome\in W(G_{n})}\sgn(\ome)q^{\apair{\varrho,\ome\chi}}\\
&+\sum_{\vec{n}}c_{\vec{n}}\sum_{\ome\in W(G_{n})}\sgn(\ome)q^{\apair{\varrho,\ome\chi}}
\prod^{\ell}_{i=1}\chi^{n_{i}}_{i},
\end{align*}
where $\vec{n}=(n_{1},n_{2},\dots,n_{\ell})$ with $n_{i}\in \{0,1,2\}$ such that at least one $n_{i}$ is nonzero,
and $c_{\vec{n}}$ is the coefficient depending only on $\mu$.
Also note that
$$
q^{\apair{\varrho,\chi}}\prod^{\ell}_{i=1}\chi^{n_{i}}_{i}=\prod^{\tilde{m}}_{i=1}\chi^{-(\frac{m+1}{2}-i-n_{i})}_{i},
$$
where $n_{i}=0$ if $i>\ell$.
Thus, it is sufficient to show that
$$
\sum_{\ome\in W(G_{n})}\sgn(\ome)q^{\apair{\varrho,\ome\chi}}
\prod^{\ell}_{i=1}(\ome\chi)^{n_{i}}_{i}=0.
$$

Since $\sum^{\ell}_{i=1}n_{i}\neq 0$, $\ell>0$ and $\tilde{m}-\ell-1\geq1$, there exist at least two distinct integers $i$ and $i'$ with $i<i'$ such that $\frac{m+1}{2}-i-n_{i}=\frac{m+1}{2}-i'-n_{i'}$. Let $i_{0}$ be the maximal integer such that $i_{0}+n_{i_{0}}=i'_{0}+n_{i'_{0}}$. Consider the Weyl group $W(G_{n})$ as the subgroup of the permutation group on $\chi_{i}$ and $\chi^{-1}_{i}$ for $1\leq i\leq \tilde{m}$. Then, define a Weyl element $\ome'$ by the following rules:
$\ome'$ permutes $\chi_{i_{0}}$ and $\chi_{i'_{0}}$, and fixes $\chi_{i}$ for the rest $i$. Hence, $\sgn(\ome')=-1$ and $\ome'$ fixes $\prod^{\tilde{m}}_{i=1}\chi^{-(\frac{m+1}{2}-i-n_{i})}_{i}$.
Let $W(G_{n})_{\vec{n}}$ be the stabilizer of $W(G_{n})$ acting on $q^{\apair{\varrho,\ome\chi}}\prod^{\ell}_{i=1}\chi^{n_{i}}_{i}$.
By $\sgn(\ome')=-1$ and $\ome'\in W(G_{n})_{\vec{n}}$, we have the restriction of $\sgn$ on $W(G_{n})_{\vec{n}}$ is not trivial.

Therefore,
\begin{align*}
&\sum_{\ome\in W(G_{n})}\sgn(\ome)q^{\apair{\varrho,\ome\chi}}\prod^{\ell}_{i}(\ome\chi)^{n_{i}}_{i}\\
=&\sum_{\ome\in W(G_{n})}q^{\apair{\varrho,\ome\chi}}\prod^{\ell}_{i}(\ome\chi)^{n_{i}}_{i}\sum_{\ome'\in W(G_{n})_{\vec{n}}}\sgn(\ome\ome')\\
=&0.
\end{align*}
\end{proof}

Comparing \eqref{eq:Delta-1} and \eqref{eq:Delta-2}, we can get the identity \eqref{eq:P*}. Hence, after replacing back $\chi_s$ for $\chi$,
we obtain the following formulas
\begin{equation}\label{eq:Pi}
P(\chi_s,\mu)=
\frac{d((\chi_s)_{\tau})\zeta((\chi_s)_{\tau},1)}{L(s+1,\tau\times\sig)L(2s+1,\tau, Asai\otimes \xi^{m}_{E/F})}T(\mu,\chi_{\sig})
\end{equation}
if $\nu$ is inert, and
\begin{equation}\label{eq:Ps}
P(\chi_s,\mu)=
\frac{\zeta((\chi_s)_{\tau_{1}},1)\zeta((\chi_s)_{\tau_{2}},1)}{L(s+1,\tau_{1}\times\tilde{\sig})L(s+1,\tau_{2}\times\sig)L(2s+1,\tau_{1}\times\tau_{2})}
\end{equation}
if $\nu$ is split. Note that $(\chi_s)_\tau$ denotes the quasi-character which is the restriction of the quasi-character $\chi_s$ to the
$\tau$-part of the torus.

It is important to point out that from the beginning of this section up to this point,
we assume that $\Pi(\chi_s)$ and $\pi(\mu)$ are generic and spherical. The following theorem extends the above results to general spherical
$\Pi(\chi_s)$ and $\pi(\mu)$.

\begin{thm}
For all choices of $\chi$ and $\mu$, the following identity holds:
\begin{align}
&\CZ_{\nu}(s,f_{\chi},f_{\mu},\psi_{\ell,\kappa})=\\
&\frac{L(s+\frac{1}{2},\tau \times\pi )}{L(s+1,\tau \times\sig )L(2s+1,\tau,Asai\otimes\xi^{m}_{E/F})}\apair{f_{\mu},f_{\sig}}_{\sig}\zeta(\chi_{\tau},1), \nonumber
\end{align}
where $\apair{f_{\mu},f_{\sig}}_{\sig}$ and $\zeta(\chi_{\tau},1)$ are independent of $s$.
Moreover, if we normalize $W_{j}$ so that $W_{j}(f_{\chi})(e)=1$, then
\begin{align}
&\CZ_{\nu}(s,f_{\chi},f_{\mu},\psi_{\ell,\kappa})=\\
&\frac{L(s+\frac{1}{2},\tau \times\pi )}{L(s+1,\tau \times\sig )L(2s+1,\tau,Asai\otimes\xi^{m}_{E/F})}\apair{f_{\mu},f_{\sig}}_{\sig}, \nonumber
\end{align}
\end{thm}
\begin{proof}
This proof is similar to Theorem 5.2 in \cite{GPSR97}.
By Proposition~\ref{pro:Q}, it is sufficient to show that
$$
\CZ_{\nu}(s,f_{\chi}*\Phi_{0},f_{\mu},\psi_{\ell,\kappa})=
P(\chi,\mu)
$$
holds for all choices of $\chi$ and $\mu$.

Define
$$
f^{*}\ppair{\begin{pmatrix} g&&\\&h&\\&&g^{*}\end{pmatrix}u \eps_{0,1}\eta n k}
=\tau(g)f_{\tau}|\det g|^{s}_{E}\delta_{P_{j}}^{\frac{1}{2}}\sig(h)f_{\sig},
$$
where $g\in\Res_{E/F}(\GL_{j})$, $h\in G_{n-j}$, $u\in U_{j}$, $n\in U_{\ell}(\Fo)$ and $k\in K_{H}$.
Recall that $f_{\tau}$ and $f_{\sig}$ are the unramified spherical vectors in $\tau$ and $\sig$.
In addition, we assume that ${\rm supp}(f^{*})=P_{j} \eps_{0,1}\eta R_{\ell}(\Fo)$.
Then $f^{*}$ is in $\Lam$ and ${\rm supp}(f^{*})\subseteq G_{n-j}K_{H}$.
Since $\CJ(f^{*})(e)=W_{j}(f^{*})(\eps_{0,1}\eta)=\zeta(\chi_{\tau},1)f_{\sig}$, we obtain
$$
\CZ(s,f^{*},f_{\mu},\psi_{\ell,\kappa})=\zeta(\chi_{\tau},1)\apair{f_{\mu},f_{\sig}}_{\sig}.
$$

Define
$$f^{\sharp}=f_{\chi}*\Phi_{0}-\frac{d(\chi_{\tau})}{L(s+1,\tau_{\nu}\times\sig_{\nu})L(2s+1,\tau_{\nu},Asai\otimes\xi^{m}_{E/F})}f^{*}.$$
By~\eqref{eq:Pi} and \eqref{eq:Ps}, if $\chi$ and $\mu$ are in general position and $s$ is in a dense open set, then
$$
\CZ_{\nu}(s,f^{\sharp},f_{\mu},\psi_{\ell,\kappa})=0.
$$
By the same argument, one can extend the {\it Density Principle} in Appendix IV to \cite[Section 5]{GPSR97} to the unitary group case, which
implies that
$
\CJ(f^{\sharp})(g)=0
$
for all choices of $\chi$, $\mu$ and $s$. Therefore, we obtain the following identity
$$
\CZ_{\nu}(s,f_{\chi}*\Phi_{0},f_{\mu},\psi_{\ell,\kappa})
=\CZ_{\nu}(s,f^{*},f_{\mu},\psi_{\ell,\kappa})=P(\chi,\mu),
$$
for all choices of $\chi$, $\mu$ and $s$.
\end{proof}

This completes the proof of Theorem 3.9, which is the key result for unramified local zeta integrals. With Theorems 3.8 and 4.11, we have the
following main global result of this paper for $j=\ell+1$. In this case, $(H_{n-j+1},G_{n-j})$ is a spherical pair, and
the Bessel period $\CP^{\psi^{-1}_{\beta-1,y_{-\kappa}}}(\varphi_\pi,\varphi_\sig)$ reduces to a spherical Bessel period.

\begin{thm}[Main]
Assume that $j=\ell+1$. Let $E(\phi_{\tau\otimes\sig},s)$ be the Eisenstein series on $G_n(\BA)$ as in
\eqref{es} and let $\pi$ be an irreducible cuspidal automorphic representation of
$H_{n-\ell}(\BA)$. Assume that the real part of $s$, $\Re(s)$, is large, and that $\pi$ and $\sig$ have a non-zero spherical
Bessel period. Then the global zeta integral
$\CZ(s,\phi_{\tau\otimes\sig},\varphi_{\pi},\psi_{\ell,w_{0}})$ is eulerian, and is equal to
$$
c_{\pi,\sig}\CZ_S(s,\phi_{\tau\otimes\sig},\varphi_{\pi},\psi_{\ell,w_{0}})\frac{L^S(s+\frac{1}{2},\pi \times\tau )}{L^S(s+1,\sig \times\tau)
L^S(2s+1,\tau,Asai\otimes\xi^{m}_{E/F})},
$$
where $c_{\pi,\sig}$ is a constant depending on the Bessel period of $\pi$ and $\sig$ and on other normalization constants, but
independent of $s$, and
$$
\CZ_S(s,\phi_{\tau\otimes\sig},\varphi_{\pi},\psi_{\ell,w_{0}})
=
\prod_{v\in S}\CZ_v(s,\phi_{\tau\otimes\sig},\varphi_{\pi},\psi_{\ell,w_{0}})
$$
is the finite product of ramified local zeta integrals.
\end{thm}

It is clear that Theorem 4.12 extends the main result of \cite{GPSR97} to the generality considered in this paper. 
Note that when $\pi$ is an irreducible cuspidal automorphic representation of $\SO_{2(n-\ell)+1}(\BA)$, one has to replace the 
complex representation $Asai\otimes\xi_{E/F}^m$ by the corresponding exterior square representation $\wedge^2$; and when 
$\pi$ is an irreducible cuspidal automorphic representation of $\SO_{2(n-\ell)}(\BA)$, one has to replace the
complex representation $Asai\otimes\xi_{E/F}^m$ by the corresponding symmetric square representation $\Sym^2$. 

There is a standard method to prove from this global identity that the partial $L$-functions $L^S(s+\frac{1}{2},\pi \times\tau )$ has
meromorphic continuation to the whole complex plane. It is more important to develop the local theory which extends the
partial $L$-function to the complete $L$-function in this setting and hence to prove the functional equation and other analytic properties of
the complete $L$-functions of this type. This is our on-going project and will be reported in our future work.

\section{Final Remark}

We remark that the proof of Theorem  works for replacing the single variable $s$ by a multi-variable $(s_1,s_2,\cdots,s_r)$, and hence
the resulting global zeta integral represents the following product of tensor product $L$-functions
$$
L^S(s_1,\pi\times\tau_1)L^S(s_2,\pi\times\tau_2)\cdots L^S(s_r,\pi\times\tau_r).
$$
We will come back to this in our future work.

\end{document}